\documentclass[11pt,reqno]{amsart}
\usepackage{amsthm,amssymb,amsmath}

\usepackage{xcolor}

\usepackage{fix-cm}
\usepackage{enumitem,needspace}
\usepackage{ulem}
\usepackage{esint}
\usepackage{mlmodern}

\usepackage[T1]{fontenc}

\usepackage[colorlinks,citecolor=blue,hypertexnames=false]{hyperref}

\AtBeginDocument{
  \hypersetup{
    urlcolor=red,
    citecolor=blue,
    linkcolor=blue,
  }
}

\newtheorem{theorem}{Theorem}
\newtheorem{proposition}{Proposition}
\newtheorem*{theorem*}{Theorem}
\newtheorem{lemma}{Lemma}[section]
\newtheorem{corollary}{Corollary}

\theoremstyle{definition}

\newtheorem*{definition*}{\bf Definition}

\newtheorem{example}{\bf Example}[section]

\newtheorem{remark}{\sc Remark}[section]
\newtheorem*{remark*}{\sc Remark}

\newtheorem*{example*}{\bf Example}

\newcommand{\Real}{{\rm Re}\,}
\newcommand{\Imag}{{\rm Im}\,}

\newcommand{\br}[2]{\left\langle #1\right\rangle_{#2}}
\newcommand{\dd}{\,\mathrm{d}}
\newcommand{\ipair}{\sum_{i<j}}
\newcommand{\opair}{\sum_{i\ne j}}

\numberwithin{equation}{section}

\begin{document}

\title[Uniform spectral gap for particle systems]{Uniform many-particle spectral gap inequality and heat kernel bounds for strong attractive interactions}

\author{S.E.\,Boutiah, D.\,Kinzebulatov, and K.R.\,Madou}

\begin{abstract}
We prove a spectral gap inequality for Gaussian measures modified by
singular attractive pair interactions. The spectral gap constant is explicit and
uniform in the number of particles and the regularization parameter.
As an application, we use spectral gap inequality to obtain 
two-sided bounds on the transition density of the attractive logarithmic gas. In dimension two, which corresponds to the Keller-Segel finite particles, for a Gaussian centred at the origin
we prove spectral gap inequality for the entire range of strengths of attraction. 
\end{abstract}

\address{Universit\'{e} Laval, D\'{e}partement de math\'{e}matiques et de statistique, Qu\'{e}bec, QC, Canada and Laboratoire de Math\'{e}matiques Appliqu\'{e}es, Universit\'{e} Ferhat Abbas, S\'{e}tif 1, Campus El Bez,  S\'{e}tif, Algeria}
\email{sallah-eddine.boutiah.1@ulaval.ca}

\address{Universit\'{e} Laval, D\'{e}partement de math\'{e}matiques et de statistique, Qu\'{e}bec, QC, Canada}
\email{damir.kinzebulatov@mat.ulaval.ca}

\email{kodjo.madou@mail.mcgill.ca}

\thanks{The research of D.K. is supported by  NSERC grant (RGPIN-2024-04236)}



\fontsize{10.5pt}{4.4mm}\selectfont

\maketitle

\setcounter{tocdepth}{1}
\tableofcontents

\section{Introduction and main results}
\label{introduction}

\textbf{1.}~The subject of the present paper, which continues \cite{BK}, is a spectral gap inequality
for particles with singular attractive interactions, with constants that do not deteriorate as the number
of particles increases. Specifically, we consider the attractive logarithmic gas
\begin{equation}
\label{part1}
 dX_t^i=-\frac{\nu}{N}\sum_{j\ne i}
 \frac{X_t^i-X_t^j}{|X_t^i-X_t^j|^2}\,dt+\sqrt2\,dB_t^i,
 \qquad X_0^i=x^i\in\mathbb R^d,
\end{equation}
where $d\geq 2$, $N \geq 2$, and the $B^i$ are independent
$d$-dimensional Brownian motions. The positive constant $\nu$ measures
the strength of attraction between the particles.
The singularity of the interaction makes the dynamics of these particles 
different, even when $\nu$ is small, from the case of independent Brownian particles. This difference is reflected in the two-sided estimates on the transition density (heat kernel) of \eqref{part1} discussed below. 

The particle system \eqref{part1} is a basic model of aggregation under
thermal noise, with a certain balance between attraction and diffusion.
That is, sufficiently strong attraction prevents global weak existence:
for $\nu>2d$, the  argument in \cite{FJ,F} appealing to the theory of Bessel processes
shows that \eqref{part1} has no global weak solution. If $\nu$ is below smaller threshold(s), then \eqref{part1} has a weak solution  \cite{FJ,K} and these weak solutions determine a strongly continuous Feller semigroup \cite{KS}. See more detailed discussion of the SDE theory aspects of \eqref{part1} below.

Theorem \ref{full-gap} establishes a many-particle spectral gap inequality for Gaussian weights
multiplied by the logarithmic interaction weight. Both the admissible
range of $\nu$ and the spectral gap constant are independent of $N$.
The estimates are also uniform under regularization of the collision
singularities. This spectral gap inequality
gives an explicit rate of convergence to equilibrium after adding
harmonic confinement to \eqref{part1}, see Corollary \ref{confined-consequences}.

As an application, we give in Appendix \ref{direct-heat} a direct proof of the lower 
heat kernel bound for the particle system \eqref{part1} in dimensions $d\ge3$ with explicit collision weights.
It uses a variant of Theorem \ref{full-gap}, i.e.\,the spectral gap inequality for a cutoff of
the interaction weight. 
For completeness, we also give in Appendix \ref{upper_bound_app} the proof of the matching upper bound from \cite{BK}, which uses a weighted Sobolev inequality
and weighted Moser iterations. 
These two-sided heat kernel bounds quantify how the attraction splits from the diffusion. The constants in these bounds, in principle, depend on $N$.

\medskip

\textbf{2.}~We first define the weights appearing in the spectral gap
inequality. For $\varepsilon\ge0$, put
\[
 |z|_\varepsilon:=(|z|^2+\varepsilon)^{1/2},
 \qquad
 \psi_\varepsilon(x):=
 \prod_{1\le i<j\le N}|x^i-x^j|_\varepsilon^{-\nu/N},
 \qquad \psi:=\psi_0.
\]
The function $\psi_\varepsilon$ is the formal invariant density of the
regularized version of \eqref{part1}. It is not integrable over the whole
configuration space $\mathbb R^{dN}$ since it is unchanged by a common translation of
all particles. However, multiplication by a Gaussian gives a finite measure in
the range of $\nu$ of Theorem~\ref{full-gap}, see Lemma~\ref{psi-mass}.
(In dimension two, the condition for the finiteness of the measure is $\nu<4$, see Lemma~\ref{dim2_mass}.)

We write
\[
 \Gamma(t,x):=(4\pi t)^{-dN/2}e^{-|x|^2/(4t)},\qquad
 \Gamma_{s,\xi}(x):=\Gamma(s,x-\xi),\qquad
 \Gamma_s:=\Gamma_{s,0},
\]
where $\xi=(\xi^1,\ldots,\xi^N)\in\mathbb R^{dN}$.

For a weight $\rho$, we write
$\langle h\rangle_\rho:=\int_{\mathbb R^{dN}}h(x)\rho(x)\,dx$,
see Section~\ref{notations_sect} for notation.

\begin{theorem}[Uniform many-particle spectral gap inequality]
\label{full-gap}
Let $d\ge3$, $N\ge2$, and
\begin{equation}
 0<\nu<\frac{(d-2)^2}{3d}.
 \label{full-nu-condition}
\end{equation}
Then, for every $s>0$, $\varepsilon\ge0$, $\xi\in\mathbb R^{dN}$, and
real-valued
$f\in L^2_{\Gamma_{s,\xi}\psi_\varepsilon}(\mathbb R^{dN})
\cap W^{1,2}_{\mathrm{loc}}(\mathbb R^{dN})$, one has
\begin{equation}
 \br{\Gamma_{s,\xi}|\nabla f|^2}{\psi_\varepsilon}
 \ge\frac{M_\psi}{s}\inf_{c\in\mathbb R}
       \br{\Gamma_{s,\xi}|f-c|^2}{\psi_\varepsilon},
 \label{full-gap2}
\end{equation}
where
\begin{equation}
 M_\psi:=\frac{(d-2)^2-3d\nu}{2\bigl((d-2)^2-d\nu\bigr)}>0.
 \label{full-gap-constant}
\end{equation}
In particular, the constant is independent of $N$, $\varepsilon$, and
the Gaussian centre $\xi$.
\end{theorem}

The minimizing constant in \eqref{full-gap2} is
the weighted mean
\[
 \frac{\br{\Gamma_{s,\xi}f}{\psi_\varepsilon}}
      {\br{\Gamma_{s,\xi}}{\psi_\varepsilon}}.
\]
The proof of Theorem \ref{full-gap} is given in Sections~\ref{auxiliary} and~\ref{gap-proof}.

In dimension $d=2$, we consider Gaussian centred at the origin (the proof exploits this) and obtain the estimate
\[
 \br{\Gamma_s|\nabla f|^2}{\psi_\varepsilon}
 \ge\frac1{2s}\inf_{c\in\mathbb R}
       \br{\Gamma_s|f-c|^2}{\psi_\varepsilon},
 \qquad 0\le\nu<4,
\]
uniformly in $N\ge2$, $s>0$, and $\varepsilon\ge0$, see
Theorem~\ref{dim2_gap}. The coefficient $1/2s$ is the Gaussian spectral gap, and so the singular attraction does not decrease it.

\begin{remark}
The proof of Theorem~\ref{full-gap} uses  the many-particle Hardy inequality
\begin{equation}\label{multi_hardy}
 C_{d,N}\sum_{i<j}\int_{\mathbb R^{dN}}
       \frac{|h(x)|^2}{|x^i-x^j|^2}\,dx
 \le\int_{\mathbb R^{dN}}|\nabla h(x)|^2\,dx,
\end{equation}
where $C_{d,N}$ denotes the best constant. We prove  and use the lower estimate
\begin{equation}\label{C_dN_simple}
 C_{d,N}\ge\frac{(d-2)^2}{N},
\end{equation}
see Lemma~\ref{hardy}.
In fact, this estimate is contained in Hoffmann-Ostenhof, Hoffmann-Ostenhof, Laptev and Tidblom
\cite{HHLT}, who also proved 
\begin{equation}\label{HHLT_est}
 C_{d,N}\ge\frac{(d-2)^2}{N}
 \max\left\{1,\frac{N}
 {1+\sqrt{1+\frac{3(d-2)^2}{2(d-1)^2}(N-1)(N-2)}}\right\},
\end{equation}
which is a strict refinement of \eqref{C_dN_simple} when $3 \leq d \leq 6$ and $N$ is sufficiently large.
For comparison, direct summation of the one-particle Hardy inequalities
gives $(d-2)^2/[2(N-1)]$.

The lower estimate \eqref{HHLT_est} is close to being sharp, as follows from the following upper estimate:
\begin{equation*}
C_{d,N} \leq \frac{d(d-2)}{N}.
\end{equation*}
It was proved in \cite[Theorem~6.2]{KV} using a probabilistic argument: a larger Hardy constant, combined with the weak existence theorem for SDEs with form-bounded drifts in \cite{K,KV}, would imply global weak existence for \eqref{part1} in a range excluded by the Bessel process counterexample of \cite{FJ}\footnote{It should be noted that there is a small error in the calculations in \cite{HHLT}: their proof actually produces a stronger than claimed upper estimate $C_{d,N} \leq \frac{d(d-2)}{N-1}$. Both this estimate and the upper estimate of \cite{KV} can be strengthened further; the details will appear elsewhere.}. 
\end{remark}

\medskip

\textbf{3.}~Uniform spectral gap inequalities for 
particle systems have been studied under various convexity and
regularity assumptions. Guillin, Liu, Wu and Zhang \cite{GLWZ} obtain
explicit uniform constants using estimates on the conditional
one-particle measures and on the interaction Hessian.
Chafa\"i and Lehec \cite{CL} treat singular Gibbs measures with convex
interactions in ordered one-dimensional configurations.
In a related development, for the repulsive planar logarithmic gas with quadratic confinement,
Rosenzweig \cite{R26} recently obtained a logarithmic Sobolev inequality
uniform in the number of particles.
More recent criteria based on the free energy, including
\cite{BBD}, go beyond convexity while keeping assumptions on the
regularity and decomposition of the interaction.

These results do not directly give the present estimates uniformly
under regularization. Indeed, for an attractive logarithmic pair
potential,
\[
 \nabla^2\bigl(\nu\log|z|\bigr)
 =\frac{\nu}{|z|^2}
   \left(I-2\frac{z\otimes z}{|z|^2}\right).
\]
Its radial eigenvalue is $-\nu/|z|^2$, so the negative part of the
Hessian is unbounded at collisions. The proof of Theorem~\ref{full-gap} controls this term after
integration, using the many-particle Hardy inequality, rather than
requiring a pointwise lower bound on the Hessian.
The resulting spectral gap is uniform in $N$ and in $\varepsilon$,
and holds for all admissible functions, without a symmetry assumption
on the observable.

The proof in dimensions $d \geq3 $ uses the integrated Bochner criterion for the spectral gap,
see Ledoux \cite[Proposition~1.3]{Led01} and the exposition of
Cattiaux--Guillin \cite[Theorem~1.3]{CG23}. This criterion is also used in the spectral gap estimates in \cite{GLWZ}.

The proof exploits the connection to several complex variables. Namely, identifying
$\mathbb R^{2N}$ with $\mathbb C^N$, we note that the regularized
logarithmic interaction is plurisubharmonic. The H\"ormander weighted
$\bar\partial$ estimate~\cite{dim2_H} allows to bound the squared
$L^2$ distance of a (real-valued) function from the subspace of holomorphic functions by its
Dirichlet energy.
Simultaneous rotational invariance, in turn, bounds the variance of the function by twice the squared distance to this subspace.
This argument is related to the proofs of Chafa\"i~\cite{dim2_Chafai}
and Suzuki~\cite{dim2_Suzuki} for the repulsive Ginibre ensemble,
and of Rosenzweig~\cite[Section~4.1]{R26} for the repulsive planar
logarithmic gas. Here the sign of the interaction gives a lower bound
on the complex Hessian on the whole configuration space, so that the
estimate applies directly to all real-valued functions.

The attractive logarithmic gas in arbitrary dimension also appears in
the work of Chodron de Courcel, Rosenzweig and Serfaty \cite{CdCRS}
on stability, uniqueness and propagation of chaos on the torus.
Their results concern the mean field dynamics and motivate the study
of quantitative properties of the finite system.
In $d\ge3$, the logarithmic interaction considered here differs from
the usual $d$-dimensional Coulomb interaction.
In dimension two, \eqref{part1} is the finite-particle approximation of
the Keller--Segel model. See \cite{FJ,Fournier-Tardy,Tardy} for stochastic weak existence for \eqref{part1} and the analysis of collisions, and \cite{Perthame,Dolbeault} for the corresponding PDE theory.

\medskip

\textbf{4.}~The following is an immediate consequence of the spectral gap inequality of Theorem \ref{full-gap}.
Define the
probability measure
\begin{equation}
 d\mu_{s,\xi,\varepsilon}(x)
 :=\frac{\Gamma_{s,\xi}(x)\psi_\varepsilon(x)}
         {\br{\Gamma_{s,\xi}}{\psi_\varepsilon}}\,dx.
 \label{confined-measure}
\end{equation}
Theorem~\ref{full-gap} states  that
\begin{equation}
 \operatorname{Var}_{\mu_{s,\xi,\varepsilon}}(f)
 :=\int\left|f-\int f\,d\mu_{s,\xi,\varepsilon}\right|^2
            d\mu_{s,\xi,\varepsilon}
 \le\frac{s}{M_\psi}
       \int|\nabla f|^2\,d\mu_{s,\xi,\varepsilon}.
 \label{confined-poincare}
\end{equation}
Let $\mathcal A$ now denote the nonnegative self-adjoint
operator associated with the closure of
\[
 f\mapsto\int|\nabla f|^2\,d\mu_{s,\xi,\varepsilon},
 \quad f\in C_c^\infty(\mathbb R^{dN}).
\]
(Indeed, this form is closable: on compact sets its density is bounded
below by a positive constant, so convergence in its norm
identifies the gradient in the sense of distributions. Its
closure is a Dirichlet form. The finite-mass cutoff argument
shows that constants belong to its domain and have zero energy.)
For $\varepsilon>0$, the corresponding diffusion solves
\begin{equation}
 dX_t^i=-\frac{X_t^i-\xi^i}{2s}\,dt
 -\frac{\nu}{N}\sum_{j\ne i}
   \frac{X_t^i-X_t^j}{|X_t^i-X_t^j|^2+\varepsilon}\,dt
 +\sqrt2\,dB_t^i.
 \label{confined-sde}
\end{equation}
At $\varepsilon=0$ we use the closed form that we have just defined. The
same formula gives its diffusion operator away from collisions. The following corollary holds both for $\varepsilon>0$ and $\varepsilon=0$:

\begin{corollary}
\label{confined-consequences}
Under the hypotheses of Theorem~\ref{full-gap}, for every
$f\in L^2(\mu_{s,\xi,\varepsilon})$ and $t\ge0$,
\begin{equation}
 \left\|e^{-t\mathcal A}f
              -\int f\,d\mu_{s,\xi,\varepsilon}\right\|_{L^2(\mu_{s,\xi,\varepsilon})}
 \le e^{-M_\psi t/s}
       \left\|f-\int f\,d\mu_{s,\xi,\varepsilon}
             \right\|_{L^2(\mu_{s,\xi,\varepsilon})}.
 \label{confined-relaxation}
\end{equation}
As a consequence, for every Lipschitz function $g:\mathbb R^d\to\mathbb R$, we have the following law of large numbers type result:
\begin{equation}
 \operatorname{Var}_{\mu_{s,\xi,\varepsilon}}
       \left(\frac1N\sum_{i=1}^Ng(x^i)\right)
 \le\frac{s\,\operatorname{Lip}(g)^2}{M_\psi N}.
 \label{empirical-variance}
\end{equation}
\end{corollary}

\textbf{5.}~We next describe the application of the spectral gap inequality to proving heat kernel bounds for \eqref{part1}.

For $\varepsilon>0$, the regularized system is
\begin{equation}\label{part1_e}
 dX_{t,\varepsilon}^i
 =-\frac{\nu}{N}\sum_{j\ne i}
    \frac{X_{t,\varepsilon}^i-X_{t,\varepsilon}^j}
         {|X_{t,\varepsilon}^i-X_{t,\varepsilon}^j|_\varepsilon^2}\,dt
  +\sqrt2\,dB_t^i.
\end{equation}
Denote by $p_\varepsilon(t,x,y)$ its transition density (or heat kernel).
Classical theory provides a Markov semigroup $e^{-t\Lambda_\varepsilon}$,
strongly continuous on $L^r$ for $1\le r<\infty$ and on $C_\infty$,
whose integral kernel is $p_\varepsilon(t,x,y)$. Here
\begin{equation}\label{physical-generator}
 \Lambda_\varepsilon=-\Delta+b_\varepsilon\cdot\nabla,\qquad
 b_\varepsilon^i(x)
 =\frac{\nu}{N}\sum_{j\ne i}
       \frac{x^i-x^j}{|x^i-x^j|_\varepsilon^2}.
\end{equation}
Since $b_\varepsilon=-\nabla\log\psi_\varepsilon$, we also have
\begin{equation}\label{physical-weighted-generator}
 \Lambda_\varepsilon
 =-\psi_\varepsilon^{-1}
     \operatorname{div}(\psi_\varepsilon\nabla).
\end{equation}

For the limiting singular system \eqref{part1}, the following was proved in \cite{BK}. If $0\le\nu<2(d-2)$ and
$r\ge2$ satisfies
\[
 r>\frac{4}{4-\frac{2\nu}{d-2}},
\]
then
\begin{equation}\label{lim_exists}
 e^{-t\Lambda}
 =s\mbox{-}L^r\mbox{-}\lim_{\varepsilon\downarrow0}
       e^{-t\Lambda_\varepsilon},
\end{equation}
locally uniformly in $t\ge0$.
This semigroup has an integral kernel $p(t,x,y)$ (also with respect to the Lebesgue measure). The convergence in \eqref{lim_exists} is proved directly by showing that the sequence of solutions  of the approximating Kolmogorov equations $(\partial_t + \Lambda_{\varepsilon_n}) u_n=f$, $\varepsilon_n \downarrow 0$, is a Cauchy sequence in $L^\infty([0,T],L^r(\mathbb R^{dN}))$.

Fournier and Jourdain \cite{FJ} study weak existence by exploiting
the structure of the interaction kernel in dimension two.
In fact, the argument of \cite[Theorem~7 and Section~6]{FJ}
adapts to dimensions $d\ge2$ and yields global weak existence under
the condition
\[
 \nu\le 2d-\frac4{N-1},
\]
when the particles initially have an independent common atom-free law
with finite first moment. 
Furthermore, the pairwise It\^o estimate used in \cite[Theorem~5]{FJ} likewise
yields weak existence for exchangeable initial laws with finite first
moment, including laws with atoms, provided
$\nu<(d-1)N/(N-1)$.

In \cite[Theorem~6.1 and Example~6.1]{KV}, by combining the theory
of SDEs with form-bounded drifts developed there and in \cite{K,KS,KS2} with the many-particle Hardy
inequality, the authors proved for all $d \geq 3$ weak existence and constructed the corresponding
Feller semigroup under the $N$-independent sufficient condition
\[
    0 \leq \nu<2(d-2).
\]

The two-sided heat kernel bounds involve
a cutoff of the logarithmic interaction. Define
\begin{equation}\label{eta_def}
 \eta(r):=e^{(\nu/N)a(r)},\qquad
 a(r):=
 \begin{cases}
  -\log r,&0<r<1,\\
  r-1-2\log r,&1\le r\le2,\\
  1-2\log2,&r>2.
 \end{cases}
\end{equation}
Thus $\eta$ is non-increasing and belongs to
$C^{1,1}_{\mathrm{loc}}((0,\infty))$. Set
\[
 \varphi_\varepsilon(x):=\prod_{i<j}\eta(|x^i-x^j|_\varepsilon),
 \qquad \varphi:=\varphi_0,
\]
and, for $s>0$,
\begin{equation}\label{scaled-weights}
 \varphi_{s,\varepsilon}(x)
 :=\prod_{i<j}\eta(s^{-1/2}|x^i-x^j|_\varepsilon)
 =\varphi_{\varepsilon/s}(s^{-1/2}x).
\end{equation}
In particular,
\begin{equation}\label{phi_s_def}
 \varphi_s(x):=\varphi_{s,0}(x)=\varphi(s^{-1/2}x).
\end{equation}
The cutoff is useful when describing collisions at the diffusive scale
$\sqrt s$. Independently of $s$ and $\varepsilon$,
\begin{equation}\label{uniform-weight-lower}
 \varphi_{s,\varepsilon},\ \varphi_s
 \ge (e/4)^{\nu(N-1)/2}=:c_\varphi>0.
\end{equation}
This lower bound may depend on $N$.

Using the spectral gap inequality for this cutoff weight, proved in
Appendix~\ref{cutoff-gap}, we obtain the following result. Fix $T>0$ and write $p_0:=p$.

\begin{theorem}[Application to heat kernel bounds]
\label{thm1}
Let $d\geq3$ and $N\geq2$.
\begin{enumerate}[label={\rm (\roman*)}]
\item If $0<\nu<2(d-2)$, then there are positive constants
$c_1,c_2$ depending only on $d,N,\nu$ such that for all $\varepsilon \geq 0$
\begin{equation}
p_\varepsilon(t,x,y)
 \leq c_1\Gamma(c_2t,x-y)\varphi_{t,\varepsilon}(y)
 \label{heat-upper}
\end{equation}
for all $0<t\le T$, and all $x,y\in\mathbb R^{dN}$ (a.e.\,if $\varepsilon=0$).

\medskip

\item If $0<\nu<\frac{(d-2)^2}{5d}$, then there are positive constants
$c_3,c_4$ depending only on $d,N,\nu$ such that for all $\varepsilon \geq 0$
\begin{equation}
c_3\Gamma(c_4t,x-y)\varphi_{t,\varepsilon}(y) \leq p_\varepsilon(t,x,y)
 \label{heat-lower}
\end{equation}
for all $0<t\le T$, and all $x,y\in\mathbb R^{dN}$ (a.e.\,if $\varepsilon=0$).
\end{enumerate}
\end{theorem}

Such bounds provide a quantitative description of the competition between Brownian diffusion and singular interactions in a particle system. In particular, the upper bound provides occupation-time estimates for \eqref{part1}.

Alternatively we could appeal to the general weighted Gaussian estimates of Sturm \cite{St} and the weighted Poincar\'e inequality of Fabes--Kenig--Serapioni \cite{FKS}; see also \cite{PR}. At the level of the weighted Dirichlet-form realization, this approach can cover
the larger range $0<\nu<2d$ for both bounds.

The weighted spectral gap inequality of \cite{FKS} (cf. also \cite{PR}) does not imply, in our context, an $N$-independent estimate on the spectral gap. Indeed, our analysis decouples dimension $d$ from the number of particles $N$, while a general result like the one in \cite{FKS, PR}, being applied in $\mathbb R^{dN}$ with a weight, naturally does not.

One difference between our heat kernel bounds and those obtained via \cite{FKS, St} is at the a posteriori level. Namely, in Theorem \ref{thm1} the semigroup and the infinitesimal generator of \eqref{part1}  are defined via \eqref{lim_exists} on the physical space rather than associated to a Dirichlet form with singular weight.

\medskip

We obtain with some additional effort the following non-reversible extension. We state and prove it at the a priori level, i.e.\,for smooth coefficients, but with constants independent of the regularization and of derivatives of the perturbation.

\begin{corollary}
\label{cor1}
Let $d\ge3$, $N\ge2$, $\varepsilon>0$, and
$C \in (L^\infty \cap C^\infty) (\mathbb{R}^{Nd},\mathbb R^{dN \times dN})$, $C^{\top}=-C$. The upper bound of Theorem~\ref{thm1}(i) holds for
$0<\nu<2(d-2)$, and the lower bound of Theorem~\ref{thm1}(ii) holds for
$0<\nu<(d-2)^2/(5d)$, for the heat kernel $p^C_\varepsilon(t,x,y)$ of the operator
\begin{align*}
\Lambda^C_\varepsilon & =- \nabla \cdot(I +C)\cdot \nabla +b_\varepsilon \cdot(I+C)\cdot \nabla \\
& = - \Delta + \big(b_\varepsilon + r^C_\varepsilon\big)\cdot \nabla,
\qquad r^C_\varepsilon:=-\nabla C + b_\varepsilon C
\end{align*}
where $(\nabla C)_j:=\sum_{i=1}^{dN}\partial_i C_{ij}$,
with $\partial_i$ denoting differentiation in the $i$th scalar coordinate.
The constants may additionally depend on $\|C\|_\infty$, but not on
derivatives of $C$ or on $\varepsilon$.
\end{corollary}

This operator is a non-reversible skew-symmetric perturbation of $\Lambda_\varepsilon=-\Delta + b_\varepsilon \cdot \nabla$.
Namely, writing $b_\varepsilon=-\nabla \log \psi_\varepsilon$, $$\psi_\varepsilon(x)=\prod_{1 \leq i<j \leq N} |x^i-x^j|_\varepsilon^{-\frac{\nu}{N}} \quad (\text{formal invariant density of \eqref{part1_e}, i.e.\,$\Lambda_\varepsilon^\ast \psi_\varepsilon=0$}),$$ the additional drift  $r^C_\varepsilon$ is divergence-free with respect to $\psi_\varepsilon dx$, i.e.
$$
 {\rm div\,}\psi_\varepsilon r^C_\varepsilon=0.
$$

\begin{example}
Let $d=3$. Consider a constant skew-symmetric block-diagonal matrix $C={\rm diag}(J_\omega,\dots,J_\omega) \in \mathbb R^{3N \times 3N}$, where 
$$
J_\omega=\left(
\begin{array}{lll}
0 & -\omega_3 & \omega_2 \\
\omega_3 & 0 & -\omega_1 \\
-\omega_2 & \omega_1 & 0
\end{array}
\right),
$$
and the vector $\omega=(\omega_1,\omega_2,\omega_3)$ is fixed.
The action of each diagonal block $J_\omega$ on $\mathbb R^3$ is
the cross product with $\omega$: $J_\omega v=\omega\times v$.
Since $C$ is constant,
\[
 r_\varepsilon^C=b_\varepsilon C
 =-\bigl(\omega\times b_\varepsilon^i\bigr)_{i=1}^N.
\]
Thus, the particle system of Corollary \ref{cor1} is 
$$
dX_t^i = - \frac{\nu}{N}\sum_{j=1, j \neq i}^N \frac{X_t^i-X_t^j}{|X_t^i-X_t^j|_\varepsilon^2}dt + \frac{\nu}{N} \sum_{j=1, j \neq i}^N \frac{\omega \times (X_t^i-X_t^j)}{|X_t^i-X_t^j|_\varepsilon^2}dt + \sqrt{2}dB_t^i.
$$
The second interaction is perpendicular to the corresponding radial interaction, i.e.\,$(x^i-x^j) \cdot \big(\omega \times (x^i-x^j)\big)=0$. The second interaction thus introduces circulation around the axis of $\omega$ without affecting the strength of the radial attraction.
\end{example}

This non-reversible extension also follows from the weighted Gaussian
bounds of Baadi \cite[Theorem~1(2)]{Baa}, which allow real
nonsymmetric coefficients.

If $\nu=0$, i.e.\,there is no attraction between the particles, then the two-sided heat kernel bounds for $\Lambda_\varepsilon^C$ become two-sided Gaussian bounds -- this special case was handled earlier by Osada \cite{O}. Later Qian-Xi \cite{QX} extended Osada's result to skew-symmetric $C$ with entries in ${\rm BMO}$. So, there remains the question whether Corollary \ref{cor1} can be extended to $C=-C^{\top}$ with entries in ${\rm BMO}$.

\medskip

\textbf{6.}~Pointwise estimates on transition densities have
been studied by several different methods. Osada's results apply to finite stochastic vortex systems
\cite{O}. The singular interactions in that example are
divergence-free, and the bounds have the ordinary Gaussian form.
Hao, R\"ockner and Zhang \cite{HRZ} obtain two-sided Gaussian
estimates for singular diffusions in mixed integrability classes and
apply their framework to finite particle systems. Their usual
pair-interaction condition does not include the critical
inverse-distance singularity in \eqref{part1}.
Landim \cite{L} proves upper estimates for the joint transition
probabilities of finitely many exclusion particles. These differ
from estimates for a tagged particle in an infinite system, such as
those in \cite{GGM}.

Graczyk--Sawyer \cite{GS1,GS2} obtain explicit formulas and sharp
estimates for Dyson and radial Dunkl kernels, while
Dziuba\'nski--Hejna \cite{DH} treat the full Dunkl heat kernel.
These results concern repulsive interactions and root-system
geometry. Ren--Zhang \cite{RZ} obtain heat kernel bounds for kinetic
equations with singular drifts and include a second-order particle
application.

Weighted methods also occur in attractive particle models.
Andres--von Renesse \cite{AvR} use Muckenhoupt weights to establish uniqueness and regularity for
interacting Bessel particles on an interval, including an attractive
regime.

In the two-particle case $N=2$, the two-sided heat kernel bounds are known in the literature, see \cite{MS, MSS}.
The many-particle problem is substantially different because its singular set is the union of intersecting collision hyperplanes and the weight contains all pair interactions.  In particular, transferring the argument of  \cite{MS, MSS} more or less directly would lead to a condition on $\nu$ that degenerates to $\nu=0$ as $N$ goes to infinity. The approach of running Nash's and Moser-Davis' methods in the weighted setting with frozen time scale -- the one that we use in Appendicies \ref{direct-heat} and \ref{upper_bound_app} -- appeared in \cite{MS}.

\bigskip

\section{Notation} \label{notations_sect}

\begin{itemize}
\item Let $\mathcal B(X,Y)$ denote the space of bounded linear operators between Banach spaces $X \rightarrow Y$,  endowed with the operator norm $||\cdot ||_{X\rightarrow Y}$, and set $\mathcal B(X):=\mathcal B(X,X)$.
\item We write $T=s\mbox{-} Y \mbox{-}\lim_n T_n$, for $T, T_n \in \mathcal B(X,Y)$ if
$$
\lim_n\|Tf- T_nf\|_Y=0 \quad \text{ for every $f \in X$}.
$$

\item Put
$$
\langle f,g\rangle:=\int_{\mathbb R^{dN}}f(x)\overline{g(x)}\,dx,\qquad \langle h\rangle:=\int_{\mathbb R^{dN}}h(x)\,dx.
$$
We occasionally switch to the usual integral notations when it makes the calculations easier to follow.

\item
The Hardy, Bochner and spectral gap inequalities are first proved for real-valued functions.
The complex-valued functions are used in the theory of
self-adjoint operators and their sesquilinear forms, and in the
proof based on H\"ormander's estimate in Section~\ref{planar}.

\item We write for brevity $\Gamma_s(x)=\Gamma(s,x)$, the Gaussian density.

\item
Given a weight $\rho$ on $\mathbb R^{dN}$, for instance, $\rho=\Gamma_{s,\xi}\psi_\varepsilon$, we write
$$
\langle f,g\rangle_\rho:=\int_{\mathbb R^{dN}}f(x)\overline{g(x)}\rho(x)\,dx,\qquad \langle h\rangle_\rho:=\int_{\mathbb R^{dN}}h(x)\rho(x)\,dx.
$$

\item Let $L^p = L^p(\mathbb{R}^{dN}, dx)$, $W^{1,p} = W^{1,p}(\mathbb{R}^{dN}, dx)$ denote the usual Lebesgue and Sobolev spaces, respectively.

\item Denote $L^p_\rho=L^p(\rho)=L^p(\mathbb R^{dN},\rho(x)dx)$.
For the weights used below, $W^{1,2}(\mathbb R^{dN},\rho\,dx)$
consists of $f\in W^{1,2}_{\mathrm{loc}}$ such that
$f$ and $\nabla f$ belong to $L^2_\rho$.

\item Put $a \wedge b:=\min\{a,b\}$, $a \vee b:=\max\{a,b\}$.

\item If $x=(x^1,\dots,x^N) \in \mathbb R^{dN}$, then $\nabla$ denotes the full gradient on $\mathbb R^{dN}$ and $\nabla_i=\nabla_{x^i}$ the gradient in the $d$-dimensional variable $x^i$.

\item We denote by $(\nabla_i f)_\alpha$ the $\alpha$-th component of
the gradient $\nabla_i f=\nabla_{y^i}f$, that is,
$$
(\nabla_i f)_\alpha = \frac{\partial f}{\partial (y^i)_\alpha}, \qquad 1\le i \le N,\quad 1\le\alpha\le d,
$$
where $(y^i)_\alpha$ denotes the $\alpha$-th coordinate of $y^i \in \mathbb{R}^d$.

\item $C_\infty:=\{f\in C_b(\mathbb R^{dN})\mid\lim_{|x|\to\infty}f(x)=0\}$ endowed with the $\sup$-norm.

\item For any $p>1$, we use $p^\prime$ to denote its conjugate $p/(p- 1)$; if $p=1$, then $p'=\infty$.

\item We denote by $C_{\mathrm{loc}}^{1,1}(\mathbb{R}^{dN})$
the class of continuously differentiable functions whose gradients
are locally Lipschitz continuous. More precisely,
$u\in C_{\mathrm{loc}}^{1,1}(\mathbb{R}^{dN})$ if $u\in C^1(\mathbb{R}^{dN})$
and, for every compact set $K\subset\mathbb{R}^{dN}$, there exists
a constant $C_K<\infty$ such that
$$
  |\nabla u(x)-\nabla u(y)|
  \le C_K|x-y|,
  \qquad x,y\in K.
$$

\end{itemize}

\medskip

\bigskip

\section{Auxiliary estimates used in the proof of Theorem \ref{full-gap}}
\label{auxiliary}

\begin{lemma}
\label{psi-interaction}
Let $\varepsilon>0$. For every $z\ne0$,
\begin{equation} \label{full-pair-derivatives}
\nabla^2(-\log|z|_\varepsilon) = -\frac{I}{|z|_\varepsilon^2}  + \frac{2z\otimes z}{|z|_\varepsilon^4}.
\end{equation}
Moreover,
\begin{equation}\label{full-pair-hessian}
-\frac{I}{|z|^2} \le \nabla^2(-\log|z|_\varepsilon) \le \frac{I}{|z|^2},
\qquad
\bigl|\nabla(-\log|z|_\varepsilon)\bigr| \le \frac{1}{|z|}.
\end{equation}
These bounds hold uniformly in $\varepsilon>0$.
\end{lemma}

\begin{proof}
Differentiating
$-\log|z|_\varepsilon=-\tfrac12\log(|z|^2+\varepsilon)$ gives $ \nabla(-\log|z|_\varepsilon) = -\frac{z}{|z|^2+\varepsilon}$.
It follows that
$$
\bigl|\nabla(-\log|z|_\varepsilon)\bigr| = \frac{|z|}{|z|^2+\varepsilon} \le \frac{1}{|z|}.
$$
Differentiating once more gives \eqref{full-pair-derivatives}.
Since $z\otimes z\ge0$, we have
$$
\nabla^2(-\log|z|_\varepsilon) \ge -\frac{I}{|z|^2+\varepsilon} \ge -\frac{I}{|z|^2}, 
$$
and the Cauchy--Schwarz inequality gives $z\otimes z\le |z|^2I$, so
$$
\begin{aligned}
\nabla^2(-\log|z|_\varepsilon) &\le -\frac{I}{|z|^2+\varepsilon} +\frac{2|z|^2I}{(|z|^2+\varepsilon)^2} = \frac{|z|^2-\varepsilon}{(|z|^2+\varepsilon)^2}\,I  \le \frac{I}{|z|^2+\varepsilon}
\le \frac{I}{|z|^2}.
\end{aligned}
$$
\end{proof}

\begin{lemma}[Many-particle Hardy inequality]
\label{hardy}
Let $d\ge3$, $N\ge2$, and $h\in C_c^\infty(\mathbb R^{dN})$. For every
$1\le i<j\le N$, one has
\begin{equation}
  (d-2)^2\int_{\mathbb{R}^{dN}}
  \frac{|h|^2}{|y^i-y^j|^2} d y
  \le \int_{\mathbb{R}^{dN}}| (\nabla_i-\nabla_j)h|^2 d y,
  \label{hardy-pair}
\end{equation}
and
\begin{equation}
  \frac{(d-2)^2}{N}\ipair
  \int_{\mathbb{R}^{dN}}\frac{|h|^2}{|y^i-y^j|^2} d y
  \le \int_{\mathbb{R}^{dN}}|\nabla h|^2 d y.
  \label{hardy-many}
\end{equation}
\end{lemma}

As discussed in the introduction, the refinement in \cite{HHLT} improves this constant for some values of $d$ and $N$.

\begin{proof}
To obtain the pair estimate \eqref{hardy-pair}, fix $1\le i<j\le N$, write $ \nabla_i:=\nabla_{y^i}$, and   $\nabla_j:=\nabla_{y^j}$, and keep $(y^i+y^j)/2$ and all remaining variables fixed.  Then with $z=y^i-y^j$, one has $ \nabla_i-\nabla_j=2\nabla_z$.
Applying the classical Hardy inequality in $z$ and integrating in the fixed variables
gives \eqref{hardy-pair}.

Regarding the many-particle Hardy inequality \eqref{hardy-many}, we note that for any $v_1,\ldots,v_N\in\mathbb{R}^d$ the elementary identity
\begin{align*}
  \sum_{i<j}|v_i-v_j|^2   &=   \frac12\sum_{i=1}^N\sum_{j=1}^N|v_i-v_j|^2\\
 &=  \frac12\sum_{i=1}^N\sum_{j=1}^N  \left(   |v_i|^2+|v_j|^2 -2(v_i\cdot v_j)  \right)\\
  &= N\sum_{i=1}^N|v_i|^2  -\left|\sum_{i=1}^Nv_i\right|^2
\end{align*}
implies
$$
  \sum_{i<j}|v_i-v_j|^2
  \le N\sum_{i=1}^N|v_i|^2.
$$
We apply this inequality pointwise with $v_i=\nabla_i h(y) \in \mathbb{R}^d$. Since $  |\nabla h(y)|^2  =\sum_{i=1}^N|\nabla_i h(y)|^2$,
we obtain
$$
  \sum_{i<j}|(\nabla_i-\nabla_j)h(y)|^2
  \le N|\nabla h(y)|^2.
$$
Now, summing \eqref{hardy-pair} over all pairs $i<j$ and using the previous
pointwise estimate gives
\begin{align*}
  (d-2)^2\ipair
  \int_{\mathbb{R}^{dN}}
  \frac{|h|^2}{|y^i-y^j|^2}\,dy
  &\le
  \int_{\mathbb{R}^{dN}}
  \sum_{i<j}|(\nabla_i-\nabla_j)h|^2\,dy\\
  &\le
  N\int_{\mathbb{R}^{dN}}|\nabla h|^2\,dy.
\end{align*}
This proves \eqref{hardy-many}.
\end{proof}

\begin{lemma}[Gaussian mass and second moment]
\label{psi-mass}
Let $0<\nu<d$. There is a constant $C$, depending only on $d,N,\nu$,
such that, for every $s>0$, $\varepsilon\ge0$, and
$\xi\in\mathbb R^{dN}$,
\begin{equation}
 0<\br{\Gamma_{s,\xi}}{\psi_\varepsilon}
 \le\br{\Gamma_{s,\xi}
            (1+|\,\cdot-\xi|^2/s)}{\psi_\varepsilon}
 \le C s^{-\nu(N-1)/4}.
 \label{psi-mass2}
\end{equation}
Thus, the measure with density
$\Gamma_{s,\xi}\psi_\varepsilon$ has finite positive mass.
\end{lemma}

\begin{proof}
We first take $s=1$. Since $\psi_\varepsilon\le\psi$ almost
everywhere, it is enough to prove the upper bounds for $\varepsilon=0$.
For $0<p<d$, uniformly in $a,z\in\mathbb R^d$,
\[
 \int_{\mathbb R^d}(4\pi)^{-d/2}e^{-|u-a|^2/4}
                  |u-z|^{-p}\,du\le C_p.
\]
Indeed, on $|u-z|<1$ we bound the Gaussian by its supremum and
obtain a uniform bound from local integrability of $|u-z|^{-p}$.
On the complement, $|u-z|^{-p}\le1$, and the Gaussian has integral one.

Fix $2\le m\le N$ and the variables $x^1,\ldots,x^{m-1}$.
We will apply H\"older's inequality $$
  \int\prod_{i=1}^{m-1} g_i
  \le\prod_{i=1}^{m-1}\left(\int g_i^{m-1}\right)^{\frac{1}{m-1}}
$$
for the Gaussian
probability measure in $x^m$:
\[
 \int_{\mathbb R^d}(4\pi)^{-d/2}e^{-|x^m-\xi^m|^2/4}
       \prod_{i<m}|x^m-x^i|^{-\nu/N}\,dx^m\le C_m,
\]
where we have used $\nu(m-1)/N<d$. Now, successive integration in
$x^N,x^{N-1},\ldots,x^2$ proves the mass bound, uniformly in $\xi$.
In turn,
\[
 (1+|x-\xi|^2)\Gamma_{1,\xi}(x)\le C\Gamma_{2,\xi}(x)
\]
and the same argument prove the second-moment bound at unit scale $s=1$.

Finally, with $y=\sqrt{s}\,x$,
\[
 \Gamma_{s,\xi}(\sqrt{s}\,x)s^{dN/2}
   =\Gamma_{1,\xi/\sqrt{s}}(x),\qquad
 \psi_\varepsilon(\sqrt{s}\,x)
   =s^{-\nu(N-1)/4}\psi_{\varepsilon/s}(x).
\]
Therefore, the previous estimates for $s=1$ give \eqref{psi-mass2}.
\end{proof}

\begin{lemma}
\label{logpsi}
Let $\varepsilon>0$. For every smooth real-valued function $f$ and every $x$
such that $x^i\ne x^j$ whenever $i\ne j$, we have
\begin{equation}\label{full-hessian}
\begin{aligned}
(\nabla^2\log\psi_\varepsilon\,\nabla f)\cdot\nabla f &\le \frac{\nu}{N} \sum_{1\le i<j\le N} \frac{|\nabla_i f-\nabla_j f|^2}{|x^i-x^j|^2}\\
&\le \frac{2\nu}{N} \sum_{i=1}^N\sum_{\substack{j=1\\j\ne i}}^N \frac{|\nabla_i f|^2}{|x^i-x^j|^2}.
\end{aligned}
\end{equation}
Moreover,
\begin{equation}\label{logpsi-laplacian}
-\Delta\log\psi_\varepsilon \le \frac{2d\nu}{N} \sum_{1\le i<j\le N}\frac{1}{|x^i-x^j|^2},
\end{equation}
and, for every $i\ne j$,
\begin{equation} \label{logpsi-directional}
\begin{aligned}
&-(\nabla_i-\nabla_j)\cdot(\nabla_i-\nabla_j) \log\psi_\varepsilon\\
&\qquad\le \frac{d\nu}{N} \left[ \frac{4}{|x^i-x^j|^2} +\sum_{\substack{k=1\\k\ne i,j}}^N
\left( \frac{1}{|x^i-x^k|^2} +\frac{1}{|x^j-x^k|^2} \right) \right].
\end{aligned}
\end{equation}
Here, recall, $\nabla_i$ is the gradient with respect to $x^i \in \mathbb R^d$,
while $\nabla$ and $\Delta$ act on all variables.
\end{lemma}

\begin{proof}
We have $ \log\psi_\varepsilon =-\frac{\nu}{N} \sum_{1\le i<j\le N}\log|x^i-x^j|_\varepsilon$.
Differentiating the pair interactions gives 
$$
\begin{aligned}
&(\nabla^2\log\psi_\varepsilon\,\nabla f)\cdot\nabla f\\
&\quad= \frac{\nu}{N}\sum_{1\le i<j\le N} (\nabla_i f-\nabla_j f)\cdot \left[
\left.\nabla_z^2(-\log|z|_\varepsilon) \right|_{z=x^i-x^j} (\nabla_i f-\nabla_j f) \right].
\end{aligned}
$$
Applying the upper bound in  \eqref{full-pair-hessian} yields
$$
(\nabla^2\log\psi_\varepsilon\,\nabla f)\cdot\nabla f \le \frac{\nu}{N} \sum_{1\le i<j\le N} \frac{|\nabla_i f-\nabla_j f|^2}{|x^i-x^j|^2}.
$$
Then, using $ |\nabla_i f-\nabla_j f|^2 \le 2|\nabla_i f|^2+2|\nabla_j f|^2 $
and
$$
\sum_{1\le i<j\le N} \frac{|\nabla_i f|^2+|\nabla_j f|^2}{|x^i-x^j|^2} = \sum_{i=1}^N\sum_{\substack{j=1\\j\ne i}}^N
\frac{|\nabla_i f|^2}{|x^i-x^j|^2},
$$
we obtain the second inequality in \eqref{full-hessian}.

Taking the trace in the lower bound in \eqref{full-pair-hessian} gives
$$
\nabla_i\cdot\nabla_i\log|x^i-x^j|_\varepsilon =\nabla_j\cdot\nabla_j \log|x^i-x^j|_\varepsilon \le \frac{d}{|x^i-x^j|^2}.
$$
Therefore,
$$
\begin{aligned}
-\Delta\log\psi_\varepsilon &=\frac{2\nu}{N} \sum_{1\le i<j\le N} \nabla_i\cdot\nabla_i\log|x^i-x^j|_\varepsilon\\
&\le \frac{2d\nu}{N} \sum_{1\le i<j\le N}\frac{1}{|x^i-x^j|^2}.
\end{aligned}
$$
This proves \eqref{logpsi-laplacian}.

Finally, fix $i\ne j$. It suffices to consider the pair $(i,j)$ and, for each $k\ne i,j$, the pairs $(i,k)$ and $(j,k)$, since all other terms are independent of $x^i$ and $x^j$ and therefore vanish after differentiation.

For the pair $(i,j)$, we have $ \nabla_j\log|x^i-x^j|_\varepsilon =-\nabla_i\log|x^i-x^j|_\varepsilon$.
Differentiating,  
$$
\begin{aligned}
\nabla_i\cdot\nabla_j\log|x^i-x^j|_\varepsilon &=-\nabla_i\cdot\nabla_i \log|x^i-x^j|_\varepsilon,\\
\nabla_j\cdot\nabla_j\log|x^i-x^j|_\varepsilon
&=\nabla_i\cdot\nabla_i \log|x^i-x^j|_\varepsilon.
\end{aligned}
$$
Therefore,
$$
\begin{aligned}
&(\nabla_i-\nabla_j)\cdot(\nabla_i-\nabla_j) \log|x^i-x^j|_\varepsilon\\
&\quad= \bigl( \nabla_i\cdot\nabla_i -2\nabla_i\cdot\nabla_j +\nabla_j\cdot\nabla_j \bigr)\log|x^i-x^j|_\varepsilon\\
&\quad= 4\nabla_i\cdot\nabla_i\log|x^i-x^j|_\varepsilon\\
&\quad\le \frac{4d}{|x^i-x^j|^2}.
\end{aligned}
$$
Now let $k\ne i,j$. Since $ \nabla_j\log|x^i-x^k|_\varepsilon=0 $, and  $\nabla_i\log|x^j-x^k|_\varepsilon=0$, we have
$$
\begin{aligned}
&(\nabla_i-\nabla_j)\cdot(\nabla_i-\nabla_j) \log|x^i-x^k|_\varepsilon\\
&\qquad= \nabla_i\cdot\nabla_i\log|x^i-x^k|_\varepsilon \le \frac{d}{|x^i-x^k|^2},
\end{aligned}
$$
and
$$
\begin{aligned}
&(\nabla_i-\nabla_j)\cdot(\nabla_i-\nabla_j) \log|x^j-x^k|_\varepsilon\\
&\qquad= \nabla_j\cdot\nabla_j\log|x^j-x^k|_\varepsilon \le \frac{d}{|x^j-x^k|^2}.
\end{aligned}
$$
Summing up these inequalities, we obtain 
$$
\begin{aligned}
&-(\nabla_i-\nabla_j)\cdot(\nabla_i-\nabla_j) \log\psi_\varepsilon\\
&\qquad\le \frac{d\nu}{N} \left[ \frac{4}{|x^i-x^j|^2} +\sum_{\substack{k=1\\k\ne i,j}}^N
\left( \frac{1}{|x^i-x^k|^2} +\frac{1}{|x^j-x^k|^2} \right) \right].
\end{aligned}
$$
This proves \eqref{logpsi-directional}.
\end{proof}

\begin{lemma} \label{ground-state}
Let $\varepsilon>0$, $\xi\in\mathbb R^{dN}$, and let $h\in C_c^\infty(\mathbb{R}^{dN})$ be real-valued. Then
\begin{equation}
\begin{aligned}
  \int_{\mathbb{R}^{dN}}
  \left|\nabla\!\left(h\sqrt{\psi_\varepsilon\Gamma_{1,\xi}}\right)\right|^2 d y
  &=\br{\Gamma_{1,\xi}|\nabla h|^2}{\psi_\varepsilon}
    -\frac12\br{
      \Gamma_{1,\xi}h^2\Delta\log(\psi_\varepsilon\Gamma_{1,\xi})
    }{\psi_\varepsilon}\\
  &\quad-\frac14\br{
      \Gamma_{1,\xi}h^2\left|\nabla\log(\psi_\varepsilon\Gamma_{1,\xi})\right|^2
    }{\psi_\varepsilon}.
\end{aligned}
  \label{ground-state2}
\end{equation}
Moreover, for every $i\ne j$,
\begin{equation}
\begin{aligned}
  &\int_{\mathbb{R}^{dN}}
    \left|(\nabla_i-\nabla_j)
    \left(h\sqrt{\psi_\varepsilon\Gamma_{1,\xi}}\right)\right|^2 d y\\
  &=\br{\Gamma_{1,\xi}|(\nabla_i-\nabla_j)h|^2}{\psi_\varepsilon}
   -\frac12\br{
     \Gamma_{1,\xi}h^2(\nabla_i-\nabla_j)\cdot(\nabla_i-\nabla_j)
     \log(\psi_\varepsilon\Gamma_{1,\xi})}{\psi_\varepsilon}\\
  &\quad-
  \frac14\br{
    \Gamma_{1,\xi}h^2\left|(\nabla_i-\nabla_j)
    \log(\psi_\varepsilon\Gamma_{1,\xi})\right|^2}{\psi_\varepsilon}.
\end{aligned}
  \label{ground-state-pair}
\end{equation}
\end{lemma}

\begin{proof}
We have
\[
  \nabla\!\left(h\sqrt{\psi_\varepsilon\Gamma_{1,\xi}}\right)
  =\sqrt{\psi_\varepsilon\Gamma_{1,\xi}}
  \left(\nabla h+\frac h2\nabla\log(\psi_\varepsilon\Gamma_{1,\xi})\right).
\]
Squaring this identity and integrating over $\mathbb{R}^{dN}$ gives
\begin{align*}
  \int_{\mathbb{R}^{dN}}
  \left|\nabla\!\left(h\sqrt{\psi_\varepsilon\Gamma_{1,\xi}}\right)\right|^2 d y
  &=\br{\Gamma_{1,\xi}|\nabla h|^2}{\psi_\varepsilon}
  +\int_{\mathbb{R}^{dN}}\psi_\varepsilon\Gamma_{1,\xi}h\nabla h\cdot
    \nabla\log(\psi_\varepsilon\Gamma_{1,\xi}) d y\\
  &\quad+\frac14\br{
    \Gamma_{1,\xi}h^2\left|\nabla\log(\psi_\varepsilon\Gamma_{1,\xi})\right|^2
  }{\psi_\varepsilon}.
\end{align*}
The mixed term is handled by integration by parts.  Since $h$ is compactly
supported, we have
\begin{equation*}
  \int_{\mathbb{R}^{dN}}\psi_\varepsilon\Gamma_{1,\xi}h\nabla h\cdot
  \nabla\log(\psi_\varepsilon\Gamma_{1,\xi}) dy =\frac12\int_{\mathbb{R}^{dN}}\nabla(h^2)\cdot
  \nabla(\psi_\varepsilon\Gamma_{1,\xi}) dy
  =-\frac12\int_{\mathbb{R}^{dN}}h^2
  \Delta(\psi_\varepsilon\Gamma_{1,\xi}) dy
\end{equation*}
and
\[
  \frac{\Delta(\psi_\varepsilon\Gamma_{1,\xi})}
  {\psi_\varepsilon\Gamma_{1,\xi}}
  =\Delta\log(\psi_\varepsilon\Gamma_{1,\xi})
  +\left|\nabla\log(\psi_\varepsilon\Gamma_{1,\xi})\right|^2.
\]
This yields \eqref{ground-state2}.

We now prove the pair identity \eqref{ground-state-pair}.  The same product and chain rules yield
$$
\begin{aligned}
  &(\nabla_i-\nabla_j)
  \left(h\sqrt{\psi_\varepsilon\Gamma_{1,\xi}}\right) =\sqrt{\psi_\varepsilon\Gamma_{1,\xi}}
  \left((\nabla_i-\nabla_j)h
  +\frac h2(\nabla_i-\nabla_j)
  \log(\psi_\varepsilon\Gamma_{1,\xi})\right).
\end{aligned}
$$
After squaring and integrating, the corresponding mixed term satisfies
$$
\begin{aligned}
  &\int_{\mathbb{R}^{dN}}\psi_\varepsilon\Gamma_{1,\xi}h
  (\nabla_i-\nabla_j)h\cdot(\nabla_i-\nabla_j)
  \log(\psi_\varepsilon\Gamma_{1,\xi}) d y\\
  &\quad= \frac12\int_{\mathbb{R}^{dN}}(\nabla_i-\nabla_j)(h^2)\cdot
  (\nabla_i-\nabla_j)(\psi_\varepsilon\Gamma_{1,\xi}) d y\\
  &\quad=-\frac12\int_{\mathbb{R}^{dN}}h^2
  (\nabla_i-\nabla_j)\cdot(\nabla_i-\nabla_j)
  (\psi_\varepsilon\Gamma_{1,\xi}) d y.
\end{aligned}
$$
Finally,
$$
\begin{aligned}
  &\frac{(\nabla_i-\nabla_j)\cdot(\nabla_i-\nabla_j)
  (\psi_\varepsilon\Gamma_{1,\xi})}{\psi_\varepsilon\Gamma_{1,\xi}} =(\nabla_i-\nabla_j)\cdot(\nabla_i-\nabla_j)
  \log(\psi_\varepsilon\Gamma_{1,\xi})
  +\left|(\nabla_i-\nabla_j)
  \log(\psi_\varepsilon\Gamma_{1,\xi})\right|^2.
\end{aligned}
$$
Substituting this into the expanded square proves
\eqref{ground-state-pair}.
\end{proof}

\bigskip

\section{Proof of Theorem~\ref{full-gap}}
\label{gap-proof}

\subsection{Many-particle weighted Hardy-type inequalities}

\begin{proposition} \label{first-singular}
For $\varepsilon>0$, $\xi\in\mathbb R^{dN}$, and $f\in C_c^\infty(\mathbb{R}^{dN})$,
\begin{equation}
  \bigl((d-2)^2-d\nu\bigr)\ipair
  \br{\Gamma_{1,\xi}\dfrac{|\nabla f|^2}{|y^i-y^j|^2}}{\psi_\varepsilon}
  \le N\br{\Gamma_{1,\xi}|\nabla^2f|^2}{\psi_\varepsilon}
  +\frac{dN^2}{4}\br{\Gamma_{1,\xi}|\nabla f|^2}{\psi_\varepsilon}.
  \label{first-singular2}
\end{equation}
\end{proposition}

\begin{proof}

\noindent
1. Fix $1\le m\le N$ and $1\le\alpha\le d$.  By a standard density argument, the compactly supported function
$
 (\nabla_m f)_\alpha   \sqrt{\psi_\varepsilon\Gamma_{1,\xi}}
$
belongs to $W^{1,2}(\mathbb{R}^{dN})$\footnote{Here, $(\nabla_m \cdot)_\alpha=\frac{\partial \cdot}{\partial(y^m)_\alpha}$
denotes differentiation with respect to the $\alpha$-th coordinate
of $y^m$; see the notation section.}, so that the many-particle Hardy inequality \eqref{hardy-many} applies.  We obtain
$$
\begin{aligned}
  &\frac{(d-2)^2}{N}\ipair
  \int_{\mathbb{R}^{dN}}
  \frac{\psi_\varepsilon\Gamma_{1,\xi}}
       {|y^i-y^j|^2}
  \left|(\nabla_m f)_\alpha \right|^2 d y \le
  \int_{\mathbb{R}^{dN}}
  \left| \nabla\!\left(
  (\nabla_m f)_\alpha
  \sqrt{\psi_\varepsilon\Gamma_{1,\xi}}
  \right)\right|^2  d y.
\end{aligned}
$$
We now sum over $m$ and $\alpha$ and apply Lemma \ref{ground-state}, i.e.\,\eqref{ground-state2}, to each derivative of $f$.  Since
$$
  \sum_{m=1}^N\sum_{\alpha=1}^d
  \left|(\nabla_m f)_\alpha\right|^2
  =|\nabla f|^2,
  \qquad
  \sum_{m=1}^N\sum_{\alpha=1}^d
  \left|\nabla(\nabla_m f)_\alpha \right|^2
  =|\nabla^2f|^2,
$$
we find
\begin{align*}
  \frac{(d-2)^2}{N}\ipair
  \br{\Gamma_{1,\xi}\frac{|\nabla f|^2}{|y^i-y^j|^2}}{\psi_\varepsilon}
  &\le \br{\Gamma_{1,\xi}|\nabla^2f|^2}{\psi_\varepsilon}
  -\frac12\br{
    \Gamma_{1,\xi}|\nabla f|^2\Delta\log(\psi_\varepsilon\Gamma_{1,\xi})
  }{\psi_\varepsilon}\\
  &\quad-\frac14\br{
    \Gamma_{1,\xi}|\nabla f|^2
    |\nabla\log(\psi_\varepsilon\Gamma_{1,\xi})|^2
  }{\psi_\varepsilon}.
\end{align*}

\smallskip

\noindent
2. Since
\[
  \Delta\log(\psi_\varepsilon\Gamma_{1,\xi})
  =\Delta\log\psi_\varepsilon+\Delta\log\Gamma_{1,\xi}
\]
and
\[
  \log\Gamma_{1,\xi}(y)
  =-\frac{dN}{2}\log(4\pi)-\frac{|y-\xi|^2}{4},
  \qquad
  \Delta\log\Gamma_{1,\xi}=-\frac{dN}{2},
\]
we obtain
\begin{align} \label{equat}
  \frac{(d-2)^2}{N}\ipair
  \br{\Gamma_{1,\xi}\frac{|\nabla f|^2}{|y^i-y^j|^2}}{\psi_\varepsilon}
  &\le \br{\Gamma_{1,\xi}|\nabla^2f|^2}{\psi_\varepsilon}
  -\frac12\br{\Gamma_{1,\xi}|\nabla f|^2\Delta\log\psi_\varepsilon}{\psi_\varepsilon}  \nonumber \\
  &\quad+\frac{dN}{4}\br{\Gamma_{1,\xi}|\nabla f|^2}{\psi_\varepsilon}
  -\frac14\br{\Gamma_{1,\xi}|\nabla f|^2
  |\nabla\log(\psi_\varepsilon\Gamma_{1,\xi})|^2}{\psi_\varepsilon}.
\end{align}

\smallskip

\noindent
3. The last term in \eqref{equat}  is nonpositive and may therefore be
discarded.  Moreover, \eqref{logpsi-laplacian} states that
$$
  -\Delta\log\psi_\varepsilon
  \le\frac{2d\nu}{N}\ipair\frac1{|y^i-y^j|^2}.
$$
Multiplying this pointwise inequality by the nonnegative function
$\frac12\psi_\varepsilon\Gamma_{1,\xi}|\nabla f|^2$ and integrating gives
$$
  -\frac12\br{\Gamma_{1,\xi}|\nabla f|^2\Delta\log\psi_\varepsilon}{\psi_\varepsilon}
  \le\frac{d\nu}{N}\ipair
  \br{\Gamma_{1,\xi}\frac{|\nabla f|^2}{|y^i-y^j|^2}}{\psi_\varepsilon}.
$$
Therefore,
\begin{align*}
  \frac{(d-2)^2}{N}\ipair
  \br{\Gamma_{1,\xi}\frac{|\nabla f|^2}{|y^i-y^j|^2}}{\psi_\varepsilon}
  &\le \br{\Gamma_{1,\xi}|\nabla^2f|^2}{\psi_\varepsilon}
  +\frac{d\nu}{N}\ipair
  \br{\Gamma_{1,\xi}\frac{|\nabla f|^2}{|y^i-y^j|^2}}{\psi_\varepsilon}
  \\
  &\quad+\frac{dN}{4}
  \br{\Gamma_{1,\xi}|\nabla f|^2}{\psi_\varepsilon}.
\end{align*}
Moving the middle term in the right-hand side to the left-hand side yields
\[
  \frac{(d-2)^2-d\nu}{N}\ipair
  \br{\Gamma_{1,\xi}\frac{|\nabla f|^2}{|y^i-y^j|^2}}{\psi_\varepsilon}
  \le \br{\Gamma_{1,\xi}|\nabla^2f|^2}{\psi_\varepsilon}
  +\frac{dN}{4}\br{\Gamma_{1,\xi}|\nabla f|^2}{\psi_\varepsilon}.
\]
It remains to multiply this inequality by $N$ to arrive at \eqref{first-singular2}.
\end{proof}

\begin{proposition}\label{componentwise}
For $\varepsilon>0$, $\xi\in\mathbb R^{dN}$, and $f\in C_c^\infty(\mathbb{R}^{dN})$,
\begin{equation}
\begin{aligned}
  &
  \left((d-2)^2-\frac{d\nu(N+2)}{2N}\right)
  \opair\sum_{\alpha=1}^d
  \br{\Gamma_{1,\xi}
  \frac{\left| (\nabla_i f)_\alpha \right|^2}
  {|y^i-y^j|^2}}{\psi_\varepsilon}\\
  &\quad\le 2(N-1)\br{\Gamma_{1,\xi}|\nabla^2f|^2}{\psi_\varepsilon}
  +\frac d2(N-1)\br{\Gamma_{1,\xi}|\nabla f|^2}{\psi_\varepsilon} +\frac{d\nu}{N}\ipair
  \br{\Gamma_{1,\xi}\frac{|\nabla f|^2}{|y^i-y^j|^2}}{\psi_\varepsilon}.
\end{aligned}
  \label{component-singular}
\end{equation}
Therefore,
\begin{equation}
\begin{aligned}
  & \left((d-2)^2-\frac{d\nu(N+2)}{2N}\right)
   \opair\sum_{\alpha=1}^d
  \br{\Gamma_{1,\xi}
  \frac{\left|(\nabla_i f)_\alpha\right|^2}
  {|y^i-y^j|^2}}{\psi_\varepsilon} \\
  &\le\left(2(N-1)+\frac{d\nu}{(d-2)^2-d\nu}\right)
  \br{\Gamma_{1,\xi}|\nabla^2f|^2}{\psi_\varepsilon}\\
  &\quad+\left(
    \frac d2(N-1)+\frac{d^2\nu N}{4\bigl((d-2)^2-d\nu\bigr)}
  \right)\br{\Gamma_{1,\xi}|\nabla f|^2}{\psi_\varepsilon}.
\end{aligned}
  \label{component-explicit}
\end{equation}
\end{proposition}

\begin{proof}

\smallskip
\noindent
1. Fix $i\ne j$ and
$1\le\alpha\le d$. By a simple approximation argument, the inequality  \eqref{hardy-pair} holds for the function
$
  \frac{\partial f}{\partial(y^i)_\alpha}
  \sqrt{\psi_\varepsilon\Gamma_{1,\xi}}.
$
Applying that inequality, then
summing over all ordered pairs $i\ne j$ and all $\alpha$, and using
\eqref{ground-state-pair}, gives
\begin{align*}
  &
  (d-2)^2\opair\sum_{\alpha=1}^d
  \br{\Gamma_{1,\xi}
  \frac{\left|(\nabla_i f)_\alpha\right|^2}
  {|y^i-y^j|^2}}{\psi_\varepsilon}\\
  &\quad\le
  \opair\sum_{\alpha=1}^d
  \br{\Gamma_{1,\xi}\left|(\nabla_i-\nabla_j)
  (\nabla_i f)_\alpha \right|^2}{\psi_\varepsilon}\\
  &\quad -\frac12\opair\sum_{\alpha=1}^d
  \br{\Gamma_{1,\xi}
  \left|(\nabla_i f)_\alpha \right|^2
  (\nabla_i-\nabla_j)\cdot(\nabla_i-\nabla_j)
  \log(\psi_\varepsilon\Gamma_{1,\xi})}{\psi_\varepsilon}\\
  &\quad -\frac14\opair\sum_{\alpha=1}^d
  \br{\Gamma_{1,\xi}
  \left|(\nabla_i f)_\alpha \right|^2
  \left|(\nabla_i-\nabla_j)
  \log(\psi_\varepsilon\Gamma_{1,\xi})\right|^2}{\psi_\varepsilon}.
\end{align*}
The last term is non-positive and may therefore be discarded when estimating
the right-hand side from above.

\smallskip
\noindent

2. The elementary inequality $|\cdot -\cdot|^2\le2|\cdot|^2+2|\cdot|^2$ gives
\begin{align*}
  &\opair\sum_{\alpha=1}^d
  \left|(\nabla_i-\nabla_j)
  (\nabla_i f)_\alpha \right|^2\\
  &\quad\le
  2(N-1)\sum_{i=1}^N\sum_{\alpha,\beta=1}^d
  \left|\frac{\partial^2f}
  {\partial(y^i)_\beta\partial(y^i)_\alpha}\right|^2
  +2\opair\sum_{\alpha,\beta=1}^d
  \left|\frac{\partial^2f}
  {\partial(y^j)_\beta\partial(y^i)_\alpha}\right|^2\\
  &\quad\le2(N-1)|\nabla^2f|^2.
\end{align*}
Indeed, the diagonal blocks of the matrix of second derivatives occur
$N-1$ times in the first sum, whereas each off-diagonal block occurs once
in the second sum.

Furthermore,
$$
  (\nabla_i-\nabla_j)\cdot(\nabla_i-\nabla_j)\log\Gamma_{1,\xi}=-d,
$$
because each of $\Delta_i\log\Gamma_{1,\xi}$ and $\Delta_j\log\Gamma_{1,\xi}$ equals
$-d/2$, while the mixed derivatives vanish.  Hence the Gaussian part of
the second term above (upon writing $\log \psi_\varepsilon \Gamma_{1,\xi} = \log \psi_\varepsilon + \log \Gamma_{1,\xi}$) is
$$
  \frac d2\opair\sum_{\alpha=1}^d
  \br{\Gamma_{1,\xi}
  \left|(\nabla_i f)_\alpha \right|^2
  }{\psi_\varepsilon}
  =\frac d2(N-1)\br{\Gamma_{1,\xi}|\nabla f|^2}{\psi_\varepsilon},
$$
and, by \eqref{logpsi-directional},
\begin{align*}
  &-\frac12\opair\sum_{\alpha=1}^d
  \br{\Gamma_{1,\xi}
  \left|(\nabla_i f)_\alpha \right|^2
  (\nabla_i-\nabla_j)\cdot(\nabla_i-\nabla_j)
  \log\psi_\varepsilon}{\psi_\varepsilon}\\
  &\quad\le\frac{d\nu}{2N}
  \br{\Gamma_{1,\xi}\opair|\nabla_i f|^2\left[
  \frac4{|y^i-y^j|^2}
  +\sum_{k\ne i,j}\left(
  \frac1{|y^i-y^k|^2}+\frac1{|y^j-y^k|^2}
  \right)\right]}{\psi_\varepsilon}.
\end{align*}
We now count the terms in the square brackets.  First,
\begin{align*}
  \opair |\nabla_i f|^2\sum_{k\ne i,j}\frac1{|y^i-y^k|^2} & = \sum_{i=1}^N\sum_{k\ne i} \frac{|\nabla_i f|^2}{|y^i-y^k|^2}
\sum_{j\notin\{i,k\}}1 \\
& = (N-2) \sum_{i\neq k} \frac{|\nabla_i f|^2}{|y^i-y^k|^2}\\
  & = (N-2)\opair\frac{|\nabla_i f|^2}{|y^i-y^j|^2},
\end{align*}
since, for each fixed ordered pair $(i,k)$, there are exactly $N-2$
admissible choices of $j$.  Second,
\[
\begin{aligned}
  \opair |\nabla_i f|^2\sum_{k\ne i,j}\frac1{|y^j-y^k|^2}
  &=2\sum_{j<k}\frac{\sum_{i\notin\{j,k\}}|\nabla_i f|^2}
  {|y^j-y^k|^2} \le 2 \sum_{j<k}\frac{|\nabla f|^2}{|y^j-y^k|^2}.
\end{aligned}
\]
Combining the two estimates gives
\begin{align*}
  &\opair |\nabla_i f|^2\left[
  \frac4{|y^i-y^j|^2}
  +\sum_{k\ne i,j}\left(
  \frac1{|y^i-y^k|^2}+\frac1{|y^j-y^k|^2}
  \right)\right]\\
  &\qquad\le(N+2)\opair
  \frac{|\nabla_i f|^2}{|y^i-y^j|^2}
  +2\ipair\frac{|\nabla f|^2}{|y^i-y^j|^2}.
\end{align*}
After integration with
the weight, the interaction contribution is bounded above by
\begin{align*}
  &\frac{d\nu(N+2)}{2N}\opair
  \br{\Gamma_{1,\xi}\frac{|\nabla_i f|^2}{|y^i-y^j|^2}}
  {\psi_\varepsilon} +\frac{d\nu}{N}\ipair
  \br{\Gamma_{1,\xi}\frac{|\nabla f|^2}{|y^i-y^j|^2}}
  {\psi_\varepsilon}.
\end{align*}

3. We obtain
\begin{align*}
  (d-2)^2\opair\sum_{\alpha=1}^d
  \br{\Gamma_{1,\xi}
  \frac{\left|(\nabla_i f)_\alpha\right|^2}
  {|y^i-y^j|^2}}{\psi_\varepsilon} & \le 2(N-1)\br{\Gamma_{1,\xi}|\nabla^2f|^2}{\psi_\varepsilon}
  +\frac d2(N-1)\br{\Gamma_{1,\xi}|\nabla f|^2}{\psi_\varepsilon}\\
  &\qquad+\frac{d\nu(N+2)}{2N}\opair
  \br{\Gamma_{1,\xi}\frac{|\nabla_i f|^2}{|y^i-y^j|^2}}
  {\psi_\varepsilon}\\
  &\qquad+\frac{d\nu}{N}\ipair
  \br{\Gamma_{1,\xi}\frac{|\nabla f|^2}{|y^i-y^j|^2}}
  {\psi_\varepsilon}.
\end{align*}
The ordered singular term in the second line is the same as the sum on the
left after summing over $\alpha$, i.e.\,writing $$
  \sum_{\alpha=1}^d
  \left|\frac{\partial f}{\partial(y^i)_\alpha}\right|^2 =  \sum_{\alpha=1}^d |(\nabla_i f)_\alpha |^2
  =|\nabla_i f|^2.$$
Moving it to the left gives  \eqref{component-singular}.

4. Under \eqref{full-nu-condition}, one has $(d-2)^2-d\nu>0$.  Therefore,
\eqref{first-singular2} implies
\begin{align*}
  &\frac{d\nu}{N}\ipair
  \br{\Gamma_{1,\xi}\frac{|\nabla f|^2}{|y^i-y^j|^2}}
  {\psi_\varepsilon}\\
  &\quad\le\frac{d\nu}{(d-2)^2-d\nu}
  \br{\Gamma_{1,\xi}|\nabla^2f|^2}{\psi_\varepsilon}
  +\frac{d^2\nu N}{4\bigl((d-2)^2-d\nu\bigr)}
  \br{\Gamma_{1,\xi}|\nabla f|^2}{\psi_\varepsilon}.
\end{align*}
Substituting this bound into \eqref{component-singular} gives
\eqref{component-explicit}.
\end{proof}

\subsection{Bochner identity}

\begin{proposition} \label{bochner}
In $L^2(\mathbb{R}^{dN},\psi_\varepsilon\Gamma_{1,\xi}d y)$, for $\varepsilon>0$,
define
\begin{equation}
  \mathcal Af:=-\frac1{\psi_\varepsilon\Gamma_{1,\xi}}
  \operatorname{div}(\psi_\varepsilon\Gamma_{1,\xi}\nabla f),
  \qquad f\in C_c^\infty(\mathbb{R}^{dN}).
  \label{def-A}
\end{equation}
Then $\mathcal A$ is symmetric and nonnegative. For real-valued
$f\in C_c^\infty(\mathbb R^{dN})$, one has
\begin{equation}
  \br{\Gamma_{1,\xi}(\mathcal Af)f}{\psi_\varepsilon}
  =\br{\Gamma_{1,\xi}|\nabla f|^2}{\psi_\varepsilon}.
  \label{A-form}
\end{equation}
Moreover,
\begin{equation}
\begin{aligned}
  \br{\Gamma_{1,\xi}|\mathcal Af|^2}{\psi_\varepsilon}
  &=\br{\Gamma_{1,\xi}|\nabla^2f|^2}{\psi_\varepsilon}
  +\frac12\br{\Gamma_{1,\xi}|\nabla f|^2}{\psi_\varepsilon}\\
  &\quad-\br{\Gamma_{1,\xi}(\nabla^2\log\psi_\varepsilon\,\nabla f)
    \cdot\nabla f}{\psi_\varepsilon},
\end{aligned}
  \label{bochner2}
\end{equation}
and therefore
\begin{equation}
  \br{\Gamma_{1,\xi}|\mathcal Af|^2}{\psi_\varepsilon}
  \ge M_\psi\br{\Gamma_{1,\xi}|\nabla f|^2}{\psi_\varepsilon}
  =M_\psi\br{\Gamma_{1,\xi}(\mathcal Af)f}{\psi_\varepsilon}.
  \label{A-coercive}
\end{equation}
\end{proposition}

\begin{proof}
1.~The integration-by-parts identity \eqref{A-form} is immediate.

2.~Two further integrations by parts give the weighted Bochner identity
\eqref{bochner2}. In detail, first observe that
$$
 \mathcal Af  =  -\Delta f-\nabla\log(\psi_\varepsilon\Gamma_{1,\xi})\cdot\nabla f.
$$
All indices below range from $1$ to $dN$. Using the fact that $f$ has compact support, we obtain
\begin{align*}
 \br{\Gamma_{1,\xi}|\mathcal Af|^2}{\psi_\varepsilon}   &= -\int_{\mathbb R^{dN}}     \operatorname{div}(\psi_\varepsilon\Gamma_{1,\xi}\nabla f)\,\mathcal Af\,dy = \sum_i\int_{\mathbb R^{dN}}   \psi_\varepsilon\Gamma_{1,\xi}\, \partial_i f\,\partial_i(\mathcal Af)\,dy.
\end{align*}
Furthermore,
\begin{equation}
\label{Af_id}
 \partial_i(\mathcal Af)
 =
 -\partial_i\Delta f
 -\sum_j
    \partial_{ij}\log(\psi_\varepsilon\Gamma_{1,\xi})\,\partial_jf
 -\sum_j
    \partial_j\log(\psi_\varepsilon\Gamma_{1,\xi})\,\partial_{ij}f.
\end{equation}
For the term containing \(\partial_i\Delta f\), integration by parts
in the \(j\)-th coordinate gives
\begin{align*}
 -\sum_i\int_{\mathbb R^{dN}}
     \psi_\varepsilon\Gamma_{1,\xi}\,
     \partial_i f\,\partial_i\Delta f\,dy
 &=
 -\sum_{i,j}\int_{\mathbb R^{dN}}
     \psi_\varepsilon\Gamma_{1,\xi}\,
     \partial_i f\,\partial_{jji}f\,dy \\
 &=
 \sum_{i,j}\int_{\mathbb R^{dN}}
     \partial_j\bigl(
       \psi_\varepsilon\Gamma_{1,\xi}\,\partial_i f
     \bigr)\partial_{ij}f\,dy \\
 &=
 \int_{\mathbb R^{dN}}
     \psi_\varepsilon\Gamma_{1,\xi}|\nabla^2f|^2\,dy \\
 &\quad
 +\sum_{i,j}\int_{\mathbb R^{dN}}
     \psi_\varepsilon\Gamma_{1,\xi}\,
     \partial_j\log(\psi_\varepsilon\Gamma_{1,\xi})\,
     \partial_i f\,\partial_{ij}f\,dy.
\end{align*}
The last term cancels the term arising from the last summand in \eqref{Af_id}. Therefore,
$$
\left\langle\Gamma_{1,\xi}|\mathcal Af|^2\right\rangle_{\psi_\varepsilon}  = \left\langle\Gamma_{1,\xi}|\nabla^2f|^2\right\rangle_{\psi_\varepsilon}  - \left\langle \Gamma_{1,\xi} \bigl(  \nabla^2\log(\psi_\varepsilon\Gamma_{1,\xi})\nabla f
   \bigr)\cdot\nabla f \right\rangle_{\psi_\varepsilon}. 
$$
Since
\[
 \nabla^2\log(\psi_\varepsilon\Gamma_{1,\xi})
 =
 \nabla^2\log\psi_\varepsilon
 +
 \nabla^2\log\Gamma_{1,\xi}
 =
 \nabla^2\log\psi_\varepsilon-\frac12 I,
\]
we conclude that
\[
 \left\langle\Gamma_{1,\xi}|\mathcal Af|^2\right\rangle_{\psi_\varepsilon}
 =
 \left\langle\Gamma_{1,\xi}|\nabla^2f|^2\right\rangle_{\psi_\varepsilon}
 +\frac12
  \left\langle\Gamma_{1,\xi}|\nabla f|^2\right\rangle_{\psi_\varepsilon}
 -
 \left\langle
   \Gamma_{1,\xi}
   (\nabla^2\log\psi_\varepsilon\nabla f)\cdot\nabla f
 \right\rangle_{\psi_\varepsilon},
\]
which is \eqref{bochner2}.

For $\varepsilon>0$, the density is smooth and positive,
so all derivatives and integrations by parts above are classical.

3.~We now prove \eqref{A-coercive}. To this end, we need to estimate the interaction term in \eqref{bochner2}. By
\eqref{full-nu-condition},
\[
 (d-2)^2-\frac{d\nu}{2}\left(1+\frac2N\right)
 \ge (d-2)^2-d\nu>0.
\]
Multiplying \eqref{component-explicit} by $2\nu/N$ and
dividing by its positive left-hand coefficient gives
\begin{align*}
 &\frac{2\nu}{N}\opair
   \br{\Gamma_{1,\xi}\frac{|\nabla_i f|^2}{|y^i-y^j|^2}}
      {\psi_\varepsilon}\\
 &\quad\le
 \frac{2\nu\left(2\frac{N-1}{N}
       +\frac{d\nu}{N((d-2)^2-d\nu)}\right)}
      {(d-2)^2-\frac{d\nu}{2}(1+2/N)}
       \br{\Gamma_{1,\xi}|\nabla^2f|^2}{\psi_\varepsilon}\\
 &\qquad+
 \frac{\frac{d\nu}{2}\left(2\frac{N-1}{N}
       +\frac{d\nu}{(d-2)^2-d\nu}\right)}
      {(d-2)^2-\frac{d\nu}{2}(1+2/N)}
       \br{\Gamma_{1,\xi}|\nabla f|^2}{\psi_\varepsilon}.
\end{align*}
We bound the two coefficients independently of $N$.
Since $d\nu/((d-2)^2-d\nu)<1/2$, the parenthesis in the
coefficient of the second-order derivatives integral is at most $2$.
Regarding the first derivatives integral, we have
\[
 \frac12\left(2\frac{N-1}{N}
        +\frac{d\nu}{(d-2)^2-d\nu}\right)
        \bigl((d-2)^2-d\nu\bigr)
 \le (d-2)^2-\frac{d\nu}{2}\left(1+\frac2N\right).
\]
The right-hand side minus the left-hand side is
$((d-2)^2-2d\nu)/N>0$. We have thus proved
\begin{equation}
 \begin{aligned}
 \frac{2\nu}{N}\opair
   \br{\Gamma_{1,\xi}\frac{|\nabla_i f|^2}{|y^i-y^j|^2}}
      {\psi_\varepsilon}
 &\le\frac{4\nu}{(d-2)^2-d\nu}
       \br{\Gamma_{1,\xi}|\nabla^2f|^2}{\psi_\varepsilon}\\
 &\quad+\frac{d\nu}{(d-2)^2-d\nu}
       \br{\Gamma_{1,\xi}|\nabla f|^2}{\psi_\varepsilon}.
 \end{aligned}
 \label{ordered-hardy-control}
\end{equation}
Now, using \eqref{full-hessian}, we bound the interaction
term in the Bochner identity. Hence
\begin{equation}
 \begin{aligned}
 \br{\Gamma_{1,\xi}|\mathcal Af|^2}{\psi_\varepsilon}
 &\ge\left(1-\frac{4\nu}{(d-2)^2-d\nu}\right)
       \br{\Gamma_{1,\xi}|\nabla^2f|^2}{\psi_\varepsilon}\\
 &\quad+M_\psi\br{\Gamma_{1,\xi}|\nabla f|^2}{\psi_\varepsilon}.
 \end{aligned}
 \label{full-bochner-coercivity}
\end{equation}
Here
\[
 \frac{4\nu}{(d-2)^2-d\nu}<\frac2d<1,
 \qquad
 \frac12-\frac{d\nu}{(d-2)^2-d\nu}=:M_\psi>0.
\]
Discarding the nonnegative second-derivatives integral yields the sought inequality
\eqref{A-coercive}.
\end{proof}

\begin{proposition}
\label{spectral}
For $\varepsilon>0$, the operator $\mathcal A$ is essentially self-adjoint on
$C_c^\infty(\mathbb{R}^{dN})$.  Its closure, still denoted by $\mathcal A$, has spectrum
$$
  \sigma(\mathcal A)\subset\{0\}\cup[M_\psi,\infty).
$$
\end{proposition}

\begin{proof}
1.~We first prove essential self-adjointness. We use below some standard operator-theoretic results, see, for example,
\cite[Chapter~VIII]{ReedSimonI}. Since $\mathcal A$ is densely
defined, symmetric, and nonnegative, the standard range criterion reduces
the problem to proving
$$
  \ker(\mathcal A^*+1)=\{0\}.
$$
Indeed,
\[
 \overline{\operatorname{Ran}(\mathcal A+1)}
 =
 \ker(\mathcal A^*+1)^\perp,
\]
while nonnegativity gives
\[
        \|(\mathcal A+1)f\|\geq \|f\|,
        \qquad f\in C_c^\infty(\mathbb R^{dN}).
\]
The same estimate holds for the closure of \(\mathcal A+1\), and hence the
range of this closure is closed. Thus, if
\(\ker(\mathcal A^*+1)=\{0\}\), then \(\operatorname{Ran}(\mathcal A+1)\) is dense,
and consequently the closure of \(\mathcal A+1\) is onto. The range criterion
therefore implies that \(\mathcal A\) is essentially self-adjoint.

Let us prove that $\ker(\mathcal A^*+1)=\{0\}$.
Let $u\in\ker(\mathcal A^*+1)$. Thus \(u\in D(\mathcal A^*)\) and
\(\mathcal A^*u=-u\). By the definition of the adjoint, for every
\(\zeta\in C_c^\infty(\mathbb R^{dN})\),
\[
 \bigl\langle \mathcal A\zeta,u\bigr\rangle_{L^2(\psi_\varepsilon\Gamma_{1,\xi}dy)}
 =
 \bigl\langle \zeta, \mathcal A^*u\bigr\rangle_{L^2(\psi_\varepsilon\Gamma_{1,\xi}dy)}
 =
 -\bigl\langle \zeta,u\bigr\rangle_{L^2(\psi_\varepsilon\Gamma_{1,\xi}dy)}.
\]
By the definition of \(\mathcal A\), this becomes
\[
 -\int_{\mathbb R^{dN}}
      \operatorname{div}
      \bigl(\psi_\varepsilon\Gamma_{1,\xi}\nabla\zeta\bigr)
      \,\overline{u}\,dy
 +
 \int_{\mathbb R^{dN}}
      \psi_\varepsilon\Gamma_{1,\xi}\zeta\,\overline{u}\,dy
 =0.
\]
Since the coefficients are real, taking complex conjugates and replacing
\(\overline{\zeta}\) by \(\zeta\) gives
\[
 -\int_{\mathbb R^{dN}}
      u\,\operatorname{div}
      \bigl(\psi_\varepsilon\Gamma_{1,\xi}\nabla\zeta\bigr)\,dy
 +
 \int_{\mathbb R^{dN}}
      \psi_\varepsilon\Gamma_{1,\xi}u\zeta\,dy
 =0,
 \qquad
 \zeta\in C_c^\infty(\mathbb R^{dN}).
\]
Hence we have, in the sense of distributions,
\begin{equation}
\label{u_eq}
  -\operatorname{div}(\psi_\varepsilon\Gamma_{1,\xi}\nabla u)
  +\psi_\varepsilon\Gamma_{1,\xi}u=0.
\end{equation}
Furthermore, on every compact subset of $\mathbb{R}^{dN}$, the function
$\psi_\varepsilon\Gamma_{1,\xi}$ belongs to $C^{1,1}$ and is bounded above and
below by strictly positive constants.  The equation is therefore locally
uniformly elliptic, and so
$u\in W^{1,2}_{\mathrm{loc}}(\mathbb{R}^{dN})$.
Now, choose $\chi_R\in C_c^\infty(\mathbb{R}^{dN})$ such that
$$
  0\le\chi_R\le1,\qquad
  \chi_R=1\ \text{on }B_R,\qquad
  \operatorname{supp}\chi_R\subset B_{2R},\qquad
  |\nabla\chi_R|\le\frac{C}{R}.
$$
Multiplying equation \eqref{u_eq} by the test function $\chi_R^2\overline u$, integrating, and taking real parts yields
$$
 \bigl\langle  \Gamma_{1,\xi}   \chi_R^2|\nabla u|^2 \bigr\rangle_{\psi_\varepsilon} + \bigl\langle   \Gamma_{1,\xi}
    \chi_R^2|u|^2\bigr\rangle_{\psi_\varepsilon} = -2\operatorname{Re} \bigl\langle  \Gamma_{1,\xi}\chi_R\overline u\,
    \nabla u\cdot\nabla\chi_R \bigr\rangle_{\psi_\varepsilon}
$$
The Cauchy--Schwarz inequality therefore yields
$$
 \frac{1}{2} \bigl\langle  \Gamma_{1,\xi}   \chi_R^2|\nabla u|^2 \bigr\rangle_{\psi_\varepsilon} + \bigl\langle   \Gamma_{1,\xi}   \chi_R^2|u|^2\bigr\rangle_{\psi_\varepsilon}  \le 2 \bigl\langle  \Gamma_{1,\xi}
    |u|^2|\nabla\chi_R|^2  \bigr\rangle_{\psi_\varepsilon}
$$
Since $\chi_R=1$ on $B_R$, it follows that
\[
  \int_{B_R}\psi_\varepsilon\Gamma_{1,\xi}|u|^2d y
  \le\frac{C}{R^2}
  \|u\|_{L^2(\psi_\varepsilon\Gamma_{1,\xi}d y)}^2.
\]
Letting $R\rightarrow\infty$ and using monotone convergence gives $u=0$.
Consequently, $\mathcal A$ is essentially self-adjoint.

2.~We next extend \eqref{A-coercive} to the domain of the closure which we
continue to denote by $\mathcal A$, a self-adjoint operator.  Let $f\in D(\mathcal A)$.  By the
definition of the closure, there exists
$f_n\in C_c^\infty(\mathbb{R}^{dN})$ such that
\[
  f_n\rightarrow f
  \quad \text{ and } \quad
  \mathcal Af_n\rightarrow\mathcal Af
  \quad\text{in }L^2(\psi_\varepsilon\Gamma_{1,\xi}d y).
\]
We apply \eqref{A-coercive} separately to the real and imaginary parts
of $f_n$, and then pass to the limit:
\begin{equation}
\label{contr_ineq}
  \br{\Gamma_{1,\xi}|\mathcal Af|^2}{\psi_\varepsilon}
  \ge M_\psi\br{\Gamma_{1,\xi}(\mathcal Af)\overline f}{\psi_\varepsilon}
  \quad f\in D(\mathcal A).
\end{equation}

Now, using the previous inequality, we describe the spectrum of $\mathcal A$. Let \(E_{\mathcal A}\) denote the spectral
resolution of the self-adjoint operator \(\mathcal A\), so that
\[
        E_{\mathcal A}(B)=\mathbf 1_B(\mathcal A)
\]
for every Borel set \(B\subset\mathbb R\). Fix an interval
\([a,b]\subset(0,M_\psi)\). We claim that
\[
        E_{\mathcal A}([a,b])=0.
\]
Suppose, to the contrary, that \(E_{\mathcal A}([a,b])\neq0\), and choose
\[
        0\neq f\in\operatorname{Ran}E_{\mathcal A}([a,b]).
\]
The spectral measure of \(f\) is then supported in \([a,b]\). The spectral characterization of the operator domain
therefore gives (we omit the weight indices)
\[
 \int_{\mathbb R}\lambda^2\,
     d\langle E_{\mathcal A}(\lambda)f,f\rangle
 =
 \int_{[a,b]}\lambda^2\,
     d\langle E_{\mathcal A}(\lambda)f,f\rangle
 \leq b^2\|f\|^2<\infty.
\]
Hence \(f\in D(\mathcal A)\).
By the Spectral theorem,
\begin{align*}
 \left\langle\Gamma_{1,\xi}|\mathcal Af|^2\right\rangle_{\psi_\varepsilon}
 -
 M_\psi\left\langle\Gamma_{1,\xi}(\mathcal Af)\overline f\right\rangle_{\psi_\varepsilon}
 &=
 \int_{[a,b]}\lambda(\lambda-M_\psi)\,
     d\left\langle E_{\mathcal A}(\lambda)f,f\right\rangle.
\end{align*}
The function $\lambda \mapsto \left\langle E_{\mathcal A}(\lambda)f,f\right\rangle$ is non-decreasing and, by our hypothesis, $d\left\langle E_{\mathcal A}(\lambda)f,f\right\rangle$ has non-zero total mass. Since the integrand $\lambda(\lambda-M_\psi)$ is strictly negative on $[a,b]$, the Stieltjes integral is strictly negative.
This contradicts \eqref{contr_ineq}.
Consequently,
\[
        E_{\mathcal A}([a,b])=0
        \qquad\text{for every }[a,b]\subset(0,M_\psi).
\]

If \(\lambda_0\in\sigma(\mathcal A)\cap(0,M_\psi)\), we could choose
\(0<a<\lambda_0<b<M_\psi\). Every open neighborhood of a point of the
spectrum of a self-adjoint operator has a nonzero spectral projection.
Hence \(E_{\mathcal A}([a,b])\neq0\), contradicting the preceding conclusion.
Therefore,
\[
        \sigma(\mathcal A)\cap(0,M_\psi)=\varnothing.
\]
\end{proof}

\subsection{Proof of  the spectral gap inequality completed}
\label{gap-completion}

We complete the proof of Theorem~\ref{full-gap} in four steps.
We first fix $s=1$, $\varepsilon>0$, and an arbitrary Gaussian centre
$\xi\in\mathbb R^{dN}$.

\medskip
\noindent\emph{Step 1}. Let us first prove the spectral gap inequality for $\varepsilon>0$, and on the form-domain.
By Proposition~\ref{spectral}, the operator $\mathcal A$, initially
defined on $C_c^\infty(\mathbb R^{dN})$, is essentially self-adjoint
and nonnegative. Its closure, still denoted by $\mathcal A$, is therefore
a nonnegative self-adjoint operator. The closure of the gradient form
in \eqref{A-form} defines its Friedrichs extension. Since the
self-adjoint extension is unique, this is precisely $\mathcal A$.
Consequently, its closed form has domain $D(\mathcal A^{1/2})$, and
\[
 \|\mathcal A^{1/2}f\|_{L^2(\Gamma_{1,\xi}\psi_\varepsilon dy)}^2
 =\br{\Gamma_{1,\xi}|\nabla f|^2}{\psi_\varepsilon},
 \qquad f\in D(\mathcal A^{1/2}).
\]
Here $C_c^\infty(\mathbb R^{dN})$ is a form core: every
$f\in D(\mathcal A^{1/2})$ has an approximating sequence
$f_n\in C_c^\infty(\mathbb R^{dN})$ for which
\[
 \br{\Gamma_{1,\xi}|f_n-f|^2}{\psi_\varepsilon}
 +\br{\Gamma_{1,\xi}|\nabla f_n-\nabla f|^2}{\psi_\varepsilon}
 \rightarrow0.
\]
The square root of the sum of the weighted $L^2$ norm squared and
the form energy is the form norm. The gradient appearing here is
the distributional gradient. Indeed, a sequence that is Cauchy in
the form norm is Cauchy in ordinary $W^{1,2}(B_R)$ for every $R>0$:
the density $\Gamma_{1,\xi}\psi_\varepsilon$ is continuous and bounded
below by a positive constant on each compact ball. Its local Sobolev
limit identifies the limit of the gradients with the distributional
gradient of $f$.

We next identify the kernel of $\mathcal A$. Lemma~\ref{psi-mass} gives
\[
 0<\br{\Gamma_{1,\xi}}{\psi_\varepsilon}<\infty,
\]
so constant functions belong to the weighted $L^2$ space. To see
that they also belong to the form domain, choose
$\chi_R\in C_c^\infty(\mathbb R^{dN})$ with
\[
 0\le\chi_R\le1,\qquad
 \chi_R=1\ \text{on }B_R,\qquad
 \operatorname{supp}\chi_R\subset B_{2R},\qquad
 |\nabla\chi_R|\le C/R.
\]
Then
\[
 \br{\Gamma_{1,\xi}|1-\chi_R|^2}{\psi_\varepsilon}
 \le\int_{|y|>R}\Gamma_{1,\xi}\psi_\varepsilon\,dy
 \rightarrow0,
\]
because the density is integrable, and
\[
 \br{\Gamma_{1,\xi}|\nabla\chi_R|^2}{\psi_\varepsilon}
 \le\frac{C^2}{R^2}\br{\Gamma_{1,\xi}}{\psi_\varepsilon}
 \rightarrow0.
\]
Thus the cutoffs converge to $1$ in the closed form domain, and
\[
 1\in D(\mathcal A^{1/2}),\qquad \mathcal A^{1/2}1=0.
\]
The same holds for every constant function.

Conversely, let $h\in D(\mathcal A^{1/2})$ satisfy
$\mathcal A^{1/2}h=0$, and choose a sequence $h_n$
converging to $h$ in the form-norm. For every $R>0$, local comparability of the
weighted and unweighted norms gives
\[
 h_n\rightarrow h\quad\text{in }L^2(B_R),\qquad
 \nabla h_n\rightarrow0\quad\text{in }L^2(B_R).
\]
For $\zeta\in C_c^\infty(B_R)$ and each coordinate $i$, integration
by parts therefore yields
\[
 \int_{B_R}h\,\partial_i\zeta\,dy
 =\lim_{n\to\infty}\int_{B_R}h_n\,\partial_i\zeta\,dy
 =-\lim_{n\to\infty}\int_{B_R}\partial_i h_n\,\zeta\,dy=0.
\]
Hence $\nabla h=0$ in the sense of distributions on every ball.
Since $\mathbb R^{dN}$ is connected, $h$ is constant almost
everywhere. For a nonnegative self-adjoint operator the kernels
of the operator and its square root agree, so
\[
 \ker\mathcal A=\ker\mathcal A^{1/2}
 =\{\text{constant functions}\}.
\]

For real-valued $f\in D(\mathcal A^{1/2})$, put
\[
 \bar f:=\frac{\br{\Gamma_{1,\xi}f}{\psi_\varepsilon}}
                  {\br{\Gamma_{1,\xi}}{\psi_\varepsilon}}.
\]
The numerator is finite by the Cauchy--Schwarz inequality and the
finite mass of the weight. Moreover,
\[
 \br{\Gamma_{1,\xi}(f-\bar f)}{\psi_\varepsilon}=0.
\]
Thus the constant function $\bar f$ is the orthogonal projection of
$f$ onto $\ker\mathcal A$, and $f-\bar f$ is orthogonal to that kernel.
In the spectral notation used above, this says
\[
 E_{\mathcal A}(\{0\})f=\bar f,\qquad
 E_{\mathcal A}(\{0\})(f-\bar f)=0.
\]
Proposition~\ref{spectral} gives
$\sigma(\mathcal A)\subset\{0\}\cup[M_\psi,\infty)$.
The spectral measure of $h=f-\bar f$ is therefore supported in
$[M_\psi,\infty)$. By the Spectral theorem, omitting the weight
indices,
\begin{align*}
 \|\mathcal A^{1/2}h\|^2
 &=\int_{[M_\psi,\infty)}\lambda\,
       d\langle E_{\mathcal A}(\lambda)h,h\rangle\\
 &\ge M_\psi\int_{[M_\psi,\infty)}
       d\langle E_{\mathcal A}(\lambda)h,h\rangle
 =M_\psi\|h\|^2.
\end{align*}
Finally, expanding the square and using the zero weighted mean of
$f-\bar f$ gives, for every $c\in\mathbb R$,
\[
 \br{\Gamma_{1,\xi}|f-c|^2}{\psi_\varepsilon}
 =\br{\Gamma_{1,\xi}|f-\bar f|^2}{\psi_\varepsilon}
  +|c-\bar f|^2\br{\Gamma_{1,\xi}}{\psi_\varepsilon}.
\]
The weighted mean is therefore the unique minimizing constant.
Since subtracting a constant does not change the form energy, we
have proved
\begin{equation}
 \begin{aligned}
 \br{\Gamma_{1,\xi}|\nabla f|^2}{\psi_\varepsilon}
 &=\|\mathcal A^{1/2}(f-\bar f)\|_{L^2(\Gamma_{1,\xi}\psi_\varepsilon dy)}^2\\
 &\ge M_\psi\br{\Gamma_{1,\xi}|f-\bar f|^2}{\psi_\varepsilon}
 =M_\psi\inf_{c\in\mathbb R}
             \br{\Gamma_{1,\xi}|f-c|^2}{\psi_\varepsilon}.
 \end{aligned}
 \label{full-gap-unit}
\end{equation}

\medskip
\noindent\emph{Step 2}. We now extend the previous spectral gap inequality, still for $\varepsilon>0$, to locally Sobolev functions.
Let
\[
 f\in L^2_{\Gamma_{1,\xi}\psi_\varepsilon}(\mathbb R^{dN})
       \cap W^{1,2}_{\mathrm{loc}}(\mathbb R^{dN})
\]
be real-valued and have finite weighted energy. We show 
that $f\in D(\mathcal A^{1/2})$, so that \eqref{full-gap-unit}
applies. (If the weighted energy is infinite, the inequality already
holds.)

Using the cutoffs from Step~1, we first obtain
\[
 \br{\Gamma_{1,\xi}|\chi_Rf-f|^2}{\psi_\varepsilon}
 \le\int_{|y|>R}\Gamma_{1,\xi}\psi_\varepsilon|f|^2\,dy
 \rightarrow0.
\]
We have
\[
 \nabla(\chi_Rf)-\nabla f
 =(\chi_R-1)\nabla f+f\nabla\chi_R.
\]
It follows that
\begin{align*}
 &\br{\Gamma_{1,\xi}|\nabla(\chi_Rf)-\nabla f|^2}{\psi_\varepsilon}\\
 &\quad\le2\int_{|y|>R}\Gamma_{1,\xi}\psi_\varepsilon
                                      |\nabla f|^2\,dy
       +\frac{2C^2}{R^2}
          \int_{B_{2R}\setminus B_R}
             \Gamma_{1,\xi}\psi_\varepsilon|f|^2\,dy\\
 &\quad\le2\int_{|y|>R}\Gamma_{1,\xi}\psi_\varepsilon
                                      |\nabla f|^2\,dy
       +\frac{2C^2}{R^2}\br{\Gamma_{1,\xi}|f|^2}{\psi_\varepsilon}
       \rightarrow0.
\end{align*}
The first term tends to zero because the weighted energy is finite, while
the second tends to zero because $f$ belongs to the weighted $L^2$
space. Thus cutoff approximation converges both in weighted $L^2$ norm
and in weighted gradient norm.

For fixed $R$, the function $\chi_Rf$ belongs to ordinary
$W^{1,2}(\mathbb R^{dN})$ and is supported in $B_{2R}$. On the
slightly larger ball $B_{2R+1}$ the density is bounded above and
below by positive constants, which may depend on $R$, $\varepsilon$,
and $\xi$. Mollifying $\chi_Rf$ with smooth kernels of radius less
than $1$ therefore gives functions
$f_{R,n}\in C_c^\infty(B_{2R+1})$ converging in ordinary $W^{1,2}$,
and hence satisfying
\[
 \br{\Gamma_{1,\xi}|f_{R,n}-\chi_Rf|^2}{\psi_\varepsilon}
 +\br{\Gamma_{1,\xi}|\nabla f_{R,n}-\nabla(\chi_Rf)|^2}
       {\psi_\varepsilon}
 \rightarrow0\quad\text{as }n\to\infty.
\]
Taking $R\to\infty$ and choosing $n$ large enough at each stage
produces a single sequence in $C_c^\infty(\mathbb R^{dN})$
converging to $f$ in weighted $L^2$ norm and weighted gradient norm.
By the definition of the closed form, $f\in D(\mathcal A^{1/2})$.
This proves \eqref{full-gap-unit} for the full locally Sobolev
class at unit scale and every $\varepsilon>0$.

\medskip
\noindent\emph{Step 3}. We now pass to the limit $\varepsilon \downarrow 0$.
Recall that $\psi=\psi_0$. Since 
$|y^i-y^j|_\varepsilon$ decreases to $|y^i-y^j|$ and its exponent
in $\psi_\varepsilon$ is negative,
\[
 0<\psi_\varepsilon\le\psi,\qquad
 \psi_\varepsilon\uparrow\psi\quad\text{a.e. as }\varepsilon\downarrow0,
\]
where a.e.\,excludes only the collision sets. Let
\[
 f\in L^2_{\Gamma_{1,\xi}\psi}(\mathbb R^{dN})
       \cap W^{1,2}_{\mathrm{loc}}(\mathbb R^{dN})
\]
be real-valued. As before, we may assume that its weighted energy
is finite, otherwise the sought inequality holds by definition. The inequality $\psi_\varepsilon\le\psi$ gives
\[
 \br{\Gamma_{1,\xi}|f|^2}{\psi_\varepsilon}
 \le\br{\Gamma_{1,\xi}|f|^2}{\psi}<\infty,\qquad
 \br{\Gamma_{1,\xi}|\nabla f|^2}{\psi_\varepsilon}
 \le\br{\Gamma_{1,\xi}|\nabla f|^2}{\psi}<\infty.
\]
Thus Step~2 applies to this same function $f$ for every
$\varepsilon>0$.

We now verify the convergence of all terms in the inequality.
By the monotone convergence theorem,
\begin{align*}
 \br{\Gamma_{1,\xi}}{\psi_\varepsilon}
 &\rightarrow\br{\Gamma_{1,\xi}}{\psi},\\
 \br{\Gamma_{1,\xi}|f|^2}{\psi_\varepsilon}
 &\rightarrow\br{\Gamma_{1,\xi}|f|^2}{\psi},\\
 \br{\Gamma_{1,\xi}|\nabla f|^2}{\psi_\varepsilon}
 &\rightarrow\br{\Gamma_{1,\xi}|\nabla f|^2}{\psi}.
\end{align*}
Lemma~\ref{psi-mass} ensures that the limiting mass is finite
and strictly positive. It also gives, by Cauchy--Schwarz,
\[
 \br{\Gamma_{1,\xi}|f|}{\psi}
 \le\br{\Gamma_{1,\xi}}{\psi}^{1/2}
      \br{\Gamma_{1,\xi}|f|^2}{\psi}^{1/2}<\infty.
\]
Since $\Gamma_{1,\xi}\psi_\varepsilon|f|
\le\Gamma_{1,\xi}\psi|f|$, dominated convergence therefore gives
\[
 \br{\Gamma_{1,\xi}f}{\psi_\varepsilon}
 \rightarrow\br{\Gamma_{1,\xi}f}{\psi}.
\]
In particular, the weighted means converge:
\[
 \frac{\br{\Gamma_{1,\xi}f}{\psi_\varepsilon}}
      {\br{\Gamma_{1,\xi}}{\psi_\varepsilon}}
 \rightarrow
 \frac{\br{\Gamma_{1,\xi}f}{\psi}}
      {\br{\Gamma_{1,\xi}}{\psi}}.
\]
To pass to the limit in the minimum over constants, use the
variance identity from Step~1,
\begin{equation}
 \inf_{c\in\mathbb R}\br{\Gamma_{1,\xi}|f-c|^2}{\psi_\varepsilon}
 =\br{\Gamma_{1,\xi}|f|^2}{\psi_\varepsilon}
  -\frac{\br{\Gamma_{1,\xi}f}{\psi_\varepsilon}^{\,2}}
         {\br{\Gamma_{1,\xi}}{\psi_\varepsilon}}.
 \label{variance-limit}
\end{equation}
The preceding convergence results show that its right-hand side
converges to
\[
 \br{\Gamma_{1,\xi}|f|^2}{\psi}
 -\frac{\br{\Gamma_{1,\xi}f}{\psi}^{\,2}}
        {\br{\Gamma_{1,\xi}}{\psi}}
 =\inf_{c\in\mathbb R}\br{\Gamma_{1,\xi}|f-c|^2}{\psi}.
\]
Letting $\varepsilon\downarrow0$ in \eqref{full-gap-unit} now yields
\[
 \br{\Gamma_{1,\xi}|\nabla f|^2}{\psi}
 \ge M_\psi\inf_{c\in\mathbb R}
               \br{\Gamma_{1,\xi}|f-c|^2}{\psi}.
\]
This proves the unit scale ($s=1$) spectral gap inequality also for the singular weight ($\varepsilon=0$).

\medskip
\noindent\emph{Step 4}. Finally, we pass to an arbitrary scale.
Fix $s>0$, $\varepsilon\ge0$, and $\xi\in\mathbb R^{dN}$, and let
$f\in L^2_{\Gamma_{s,\xi}\psi_\varepsilon}
\cap W^{1,2}_{\mathrm{loc}}$ be real-valued with finite weighted
energy. Set $$g(x)=f(\sqrt{s}\,x).$$ For every pair $i<j$,
\[
 |\sqrt{s}\,x^i-\sqrt{s}\,x^j|_\varepsilon
 =\bigl(s|x^i-x^j|^2+\varepsilon\bigr)^{1/2}
 =\sqrt{s}\,|x^i-x^j|_{\varepsilon/s}.
\]
There are $N(N-1)/2$ pair factors, each with exponent $-\nu/N$.
Hence the weight, the Gaussian density and the gradient scale as follows:
\[
 \psi_\varepsilon(\sqrt{s}\,x)
    =s^{-\nu(N-1)/4}\psi_{\varepsilon/s}(x),\qquad
 \Gamma_{s,\xi}(\sqrt{s}\,x)s^{dN/2}
    =\Gamma_{1,\xi/\sqrt{s}}(x),
\]
\[
 \nabla g(x)=\sqrt{s}\,\nabla f(\sqrt{s}\,x).
\]
These identities also hold a.e.\,when $\varepsilon=0$.
The change of variables $y=\sqrt{s}\,x$ gives
\begin{align*}
 \br{\Gamma_{s,\xi}}{\psi_\varepsilon}
 &=s^{-\nu(N-1)/4}
       \br{\Gamma_{1,\xi/\sqrt{s}}}{\psi_{\varepsilon/s}},\\
 \br{\Gamma_{s,\xi}f}{\psi_\varepsilon}
 &=s^{-\nu(N-1)/4}
       \br{\Gamma_{1,\xi/\sqrt{s}}g}{\psi_{\varepsilon/s}}.
\end{align*}
Thus the corresponding weighted means agree: their common value is
\[
 \bar f_{s,\xi,\varepsilon}
 :=\frac{\br{\Gamma_{s,\xi}f}{\psi_\varepsilon}}
          {\br{\Gamma_{s,\xi}}{\psi_\varepsilon}}
 =\frac{\br{\Gamma_{1,\xi/\sqrt{s}}g}{\psi_{\varepsilon/s}}}
        {\br{\Gamma_{1,\xi/\sqrt{s}}}{\psi_{\varepsilon/s}}}.
\]
More explicitly, for every $c\in\mathbb R$,
\[
 \br{\Gamma_{s,\xi}|f-c|^2}{\psi_\varepsilon}
 =s^{-\nu(N-1)/4}
       \br{\Gamma_{1,\xi/\sqrt{s}}|g-c|^2}{\psi_{\varepsilon/s}},
\]
and the energy satisfies
\[
 \br{\Gamma_{s,\xi}|\nabla f|^2}{\psi_\varepsilon}
 =s^{-1-\nu(N-1)/4}
       \br{\Gamma_{1,\xi/\sqrt{s}}|\nabla g|^2}{\psi_{\varepsilon/s}}.
\]
The first identity with $c=0$ and the second identity verify the
weighted integrability required to apply the unit-scale inequality
to $g$; the linear change of variables also preserves the local
Sobolev property. Taking the infimum over $c$ in the first identity
displays the variance scaling:
\[
 \inf_{c\in\mathbb R}\br{\Gamma_{s,\xi}|f-c|^2}{\psi_\varepsilon}
 =s^{-\nu(N-1)/4}\inf_{c\in\mathbb R}
       \br{\Gamma_{1,\xi/\sqrt{s}}|g-c|^2}{\psi_{\varepsilon/s}}.
\]
Apply Step~2 if $\varepsilon>0$, and Step~3 if $\varepsilon=0$,
with centre $\xi/\sqrt{s}$ and regularization $\varepsilon/s$.
Using the two scaling identities, we obtain
\begin{align*}
 \br{\Gamma_{s,\xi}|\nabla f|^2}{\psi_\varepsilon}
 &=s^{-1-\nu(N-1)/4}
       \br{\Gamma_{1,\xi/\sqrt{s}}|\nabla g|^2}{\psi_{\varepsilon/s}}\\
 &\ge M_\psi s^{-1-\nu(N-1)/4}\inf_{c\in\mathbb R}
       \br{\Gamma_{1,\xi/\sqrt{s}}|g-c|^2}{\psi_{\varepsilon/s}}\\
 &=\frac{M_\psi}{s}\inf_{c\in\mathbb R}
       \br{\Gamma_{s,\xi}|f-c|^2}{\psi_\varepsilon}.
\end{align*}
This is
\eqref{full-gap2}. The proof of
Theorem~\ref{full-gap} is completed. \hfill \qed

\bigskip

\section{Proof of Corollary \ref{confined-consequences}}
Fix $s>0$, $\xi\in\mathbb R^{dN}$, and $\varepsilon\ge0$. Let $\mathcal A$ be the nonnegative self-adjoint operator defined
before the corollary. 

\medskip
\noindent\emph{Step 1.}
The closed form domain is $D(\mathcal A^{1/2})$, and
\[
 \|\mathcal A^{1/2}v\|_{L^2(\mu_{s,\xi,\varepsilon})}^2
 =\int|\nabla v|^2\,d\mu_{s,\xi,\varepsilon},
 \qquad v\in D(\mathcal A^{1/2}).
\]
Here $\nabla v$ is the distributional gradient. (Indeed, an
approximating sequence from $C_c^\infty$ in the form norm is
Cauchy in ordinary $W^{1,2}$ on every compact set, because the
density of $\mu_{s,\xi,\varepsilon}$ has a positive lower bound
there. This identifies the limit of its gradients with $\nabla v$.
This argument applies also when $\varepsilon=0$. Therefore,
\eqref{confined-poincare} applies to every function in the
closed form domain. For complex-valued functions, apply that
inequality separately to the real and imaginary parts and add up the results.)

Let us also comment on why constants are preserved (here we allow for some repetitions with the previous proof). Choose
$\chi_R\in C_c^\infty$ with $0\le\chi_R\le1$, $\chi_R=1$ on
$B_R$, support in $B_{2R}$, and $|\nabla\chi_R|\le C/R$.
Since the measure has mass one,
\[
 \int|\chi_R-1|^2\,d\mu_{s,\xi,\varepsilon}\rightarrow0,
 \qquad
 \int|\nabla\chi_R|^2\,d\mu_{s,\xi,\varepsilon}
 \le\frac{C^2}{R^2}\rightarrow0.
\]
Thus $1$ belongs to the closed form domain and has zero gradient.
The identity
\[
 \int\nabla1\cdot\nabla\overline v\,d\mu_{s,\xi,\varepsilon}=0,
 \qquad v\in D(\mathcal A^{1/2}),
\]
and the definition of the operator associated with the form give
$1\in D(\mathcal A)$ and $\mathcal A1=0$. Hence
$e^{-t\mathcal A}1=1$. By self-adjointness, for every
$f\in L^2(\mu_{s,\xi,\varepsilon})$,
\[
 \int e^{-t\mathcal A}f\,d\mu_{s,\xi,\varepsilon}
 =\int f\,\bigl(e^{-t\mathcal A}1\bigr)
                  \,d\mu_{s,\xi,\varepsilon}
 =\int f\,d\mu_{s,\xi,\varepsilon}.
\]
Thus the semigroup preserves the mean appearing in
\eqref{confined-relaxation}.

\medskip
\noindent\emph{Step 2.}
First let $f\in D(\mathcal A)$, and put
\[
 u(t):=e^{-t\mathcal A}f
          -\int f\,d\mu_{s,\xi,\varepsilon}.
\]
Then $u(t)\in D(\mathcal A)$,
$\partial_tu(t)=-\mathcal A u(t)$ in $L^2$, and Step~1 gives
$\int u(t)\,d\mu_{s,\xi,\varepsilon}=0$. Differentiating its
squared $L^2$ norm, we obtain
\begin{align*}
 \frac d{dt}\int|u(t)|^2\,d\mu_{s,\xi,\varepsilon}
 &=-2\operatorname{Re}\int
       (\mathcal A u(t))\overline{u(t)}\,d\mu_{s,\xi,\varepsilon}\\
 &=-2\int|\nabla u(t)|^2\,d\mu_{s,\xi,\varepsilon}\\
 &\le-\frac{2M_\psi}{s}
             \int|u(t)|^2\,d\mu_{s,\xi,\varepsilon}.
\end{align*}
The second equality is the defining identity of the closed form,
so it remains valid for the singular weight. The last inequality
is \eqref{confined-poincare}, applied to the function $u(t)$
whose mean is zero. Gronwall's inequality yields
\[
 \int|u(t)|^2\,d\mu_{s,\xi,\varepsilon}
 \le e^{-2M_\psi t/s}
       \int|u(0)|^2\,d\mu_{s,\xi,\varepsilon}.
\]
Taking square roots gives \eqref{confined-relaxation} for
$f\in D(\mathcal A)$.

For general $f\in L^2(\mu_{s,\xi,\varepsilon})$, since $\mathcal A$ is densely defined, we can choose
$f_n\in D(\mathcal A)$ converging to $f$ in that space. The semigroup of $\mathcal A$
is an $L^2$ contraction, so
\[
 \|e^{-t\mathcal A}(f_n-f)\|_{L^2(\mu_{s,\xi,\varepsilon})}
 \le\|f_n-f\|_{L^2(\mu_{s,\xi,\varepsilon})}
 \rightarrow0.
\]
Moreover, Cauchy--Schwarz and the fact that the measure has mass
one give
\[
 \left|\int(f_n-f)\,d\mu_{s,\xi,\varepsilon}\right|
 \le\|f_n-f\|_{L^2(\mu_{s,\xi,\varepsilon})}
 \rightarrow0.
\]
We may therefore pass to the limit $n \rightarrow \infty$ in
\eqref{confined-relaxation} applied to $f_n$. This proves the first
assertion for every $f\in L^2(\mu_{s,\xi,\varepsilon})$.

\medskip
\noindent\emph{Step 3.}
Let $g:\mathbb R^d\to\mathbb R$ be Lipschitz. We first check that
the function in \eqref{empirical-variance} satisfies the
integrability requirements of \eqref{confined-poincare}.
The inequality
$|g(z)|\le|g(0)|+\operatorname{Lip}(g)|z|$ and
Cauchy--Schwarz give
\[
 \left|\frac1N\sum_{i=1}^Ng(x^i)\right|^2
 \le\frac1N\sum_{i=1}^N|g(x^i)|^2
 \le2|g(0)|^2+\frac{2\operatorname{Lip}(g)^2}{N}|x|^2.
\]
Lemma~\ref{psi-mass}, together with
$|x|^2\le2|x-\xi|^2+2|\xi|^2$, shows that the last expression is
integrable with respect to $\mu_{s,\xi,\varepsilon}$. Thus the
average belongs to $L^2(\mu_{s,\xi,\varepsilon})$.

A Lipschitz function $g$ satisfies
$|\nabla g|\le\operatorname{Lip}(g)$ a.e. In
particular, the  average belongs to
$W^{1,2}_{\mathrm{loc}}(\mathbb R^{dN})$, and its derivative in
the $i$-th particle coordinate is
\[
 \nabla_i\left(\frac1N\sum_{j=1}^Ng(x^j)\right)
 =\frac1N\nabla g(x^i).
\]
The full gradient consists of these $N$ blocks of $d$ coordinates.
Its squared length is therefore the sum of the squared lengths
of the blocks:
\[
 \left|\nabla\left(\frac1N\sum_{i=1}^Ng(x^i)\right)\right|^2
 =\frac1{N^2}\sum_{i=1}^N|\nabla g(x^i)|^2
 \le\frac{\operatorname{Lip}(g)^2}{N}
 \quad\text{a.e.}
\]
Integrating this inequality and applying
\eqref{confined-poincare} gives
\begin{align*}
 \operatorname{Var}_{\mu_{s,\xi,\varepsilon}}
       \left(\frac1N\sum_{i=1}^Ng(x^i)\right)
 &\le\frac{s}{M_\psi}\int
       \left|\nabla\left(\frac1N\sum_{i=1}^Ng(x^i)\right)\right|^2
                         \,d\mu_{s,\xi,\varepsilon}\\
 &\le\frac{s\,\operatorname{Lip}(g)^2}{M_\psi N}.
\end{align*}
This is \eqref{empirical-variance}. The proof of Corollary \ref{confined-consequences} is completed. \hfill\qed

\bigskip

\section{Spectral gap inequality in dimension two}
\label{planar}

In dimension two, the Hardy inequality used in
Section~\ref{gap-proof} does not exist (or, rather, is trivial since the corresponding Hardy constant becomes zero). Instead, we use methods of several complex variables. We identify $\mathbb R^{2N}$ with $\mathbb C^N$ and use
H\"ormander's estimate for $\bar\partial$, together with
simultaneous rotations of the particles. Related arguments for in the repulsive case were employed earlier in
\cite{dim2_Chafai,dim2_Suzuki,R26}.

Throughout this section, $d=2$, $N\geq2$, and the Gaussian is centred
at the origin. Thus
\[
 \Gamma_s(x)=(4\pi s)^{-N}e^{-|x|^2/(4s)},\qquad
 \psi_\varepsilon(x)
 =\prod_{1\leq i<j\leq N}|x^i-x^j|_\varepsilon^{-\nu/N},
 \qquad \psi=\psi_0.
\]

\begin{theorem}
\label{dim2_gap}
Let $N\geq2$, $s>0$, $\varepsilon \geq0$ and
$\nu<4$. For every real-valued
\[
 f\in L^2_{\Gamma_s\psi_\varepsilon}(\mathbb R^{2N})
       \cap W^{1,2}_{\mathrm{loc}}(\mathbb R^{2N}),
\]
one has
\begin{equation}\label{dim2_ineq}
 \br{\Gamma_s|\nabla f|^2}{\psi_\varepsilon}
 \geq\frac1{2s}\inf_{c\in\mathbb R}
                 \br{\Gamma_s|f-c|^2}{\psi_\varepsilon}.
\end{equation}
The constant $1/2s$ is sharp: the minimizing constant is
\[
 c=\frac{\br{\Gamma_s f}{\psi_\varepsilon}}
         {\br{\Gamma_s}{\psi_\varepsilon}}.
\]
In particular, for $0\leq\nu<4$ the estimate is uniform in $N$ and
$\varepsilon\geq0$.
\end{theorem}

We first verify integrability of the weight. 

\begin{lemma}[Gaussian mass in dimension two]\label{dim2_mass}
For $N\geq2$, $s>0$, and $\nu\geq0$,
\[
 \br{\Gamma_s}{\psi}<\infty
 \quad\Leftrightarrow\quad \nu<4.
\]
If $0\leq\nu<4$, then, for every $k\geq0$ and
$\varepsilon\geq0$,
\begin{equation}\label{dim2_moments}
 0<\br{\Gamma_s}{\psi_\varepsilon}
 \leq\br{\Gamma_s(1+|x|^k)}{\psi_\varepsilon}
 \leq\br{\Gamma_s(1+|x|^k)}{\psi}<\infty.
\end{equation}
\end{lemma}

The proof is a Gaussian version of the bounded-domain argument in \cite[Lemma~2]{dim2_Kiessling}.

\begin{proof}
Suppose first that $0\leq\nu<4$. The arithmetic-geometric mean
inequality gives a.e.
\begin{equation}\label{dim2_star}
 \psi(x)
 =\left(\prod_{i=1}^N
                 \prod_{j\ne i}|x^i-x^j|^{-\nu/2}\right)^{1/N}
 \leq\frac1N\sum_{i=1}^N
                 \prod_{j\ne i}|x^i-x^j|^{-\nu/2}.
\end{equation}
We have, uniformly in $a\in\mathbb R^2$, using that the centred Gaussian density only becomes smaller as $a$ goes away from the origin,
\[
 \int_{\mathbb R^2}\frac{e^{-|u|^2/(4s)}}{4\pi s}|u-a|^{-\nu/2} du
 \leq1+\frac1{4\pi s}\int_{|v|<1}|v|^{-\nu/2} dv
 =1+\frac1{s(4-\nu)}.
\]
For the $i$-th term of \eqref{dim2_star}, integrate first in the
variables $x^j$, $j\ne i$, and then in $x^i$. This yields
\[
 0<\br{\Gamma_s}{\psi}
 \leq\left(1+\frac1{s(4-\nu)}\right)^{N-1}<\infty.
\]
The estimate
$(1+|x|^k)\Gamma_s(x)\leq C_{k,s,N}\Gamma_{2s}(x)$ and
$\psi_\varepsilon\leq\psi$ imply \eqref{dim2_moments}.

For the proof of the necessity of $\nu<4$, let us introduce the orthogonal decomposition
\[
 q=\frac1{\sqrt N}\sum_{i=1}^Nx^i\in\mathbb R^2,
 \qquad
 x^i=\frac q{\sqrt N}+y^i,\qquad
 y\in H:=\left\{y\in(\mathbb R^2)^N 
 \mid \sum_{i=1}^Ny^i=0\right\}.
\]
Here $\dim H=2(N-1)$, $|x|^2=|q|^2+|y|^2$, and
$\psi(x)=\psi(y)$. In polar coordinates $y=r\omega$ on $H$,
\[
 \psi(r\omega)=r^{-\nu(N-1)/2}\psi(\omega).
\]
Put $r=|y|$. Since
\[
 |y^i-y^j|\leq |y^i|+|y^j|\leq 2r,
 \qquad 1\leq i<j\leq N,
\]
we have a.e.
\[
 \psi(y)
 =\prod_{1\leq i<j\leq N}|y^i-y^j|^{-\nu/N}
 \geq (2r)^{-\nu(N-1)/2}.
\]
On the region $|q|<1$, $0<r<1$, the Gaussian factor is bounded
below by a positive constant. Integration in $q$ and polar
integration in the $2(N-1)$-dimensional space $H$ therefore give,
for some $C>0$,
\[
 \br{\Gamma_s}{\psi}
 \geq C\int_0^1 r^{(N-1)(2-\nu/2)-1} dr.
\]
If $\nu\geq4$, the exponent is at most $-1$, and so the integral
diverges.
\end{proof}

We identify $\mathbb R^{2N}$ with $\mathbb C^N$ by setting
$z_i=(x^i)_1+\mathrm i(x^i)_2$ and write
\[
 \partial_{z_i}
 =\frac12\left(\partial_{(x^i)_1}
                    -\mathrm i\partial_{(x^i)_2}\right),\qquad
 \partial_{\bar z_i}
 =\frac12\left(\partial_{(x^i)_1}
                    +\mathrm i\partial_{(x^i)_2}\right).
\]
Then every real-valued
$f\in W^{1,2}_{\mathrm{loc}}(\mathbb R^{2N})$ satisfies
\begin{equation}\label{dim2_realcomplex}
 \sum_{i=1}^N|\partial_{\bar z_i}f|^2
 =\frac14|\nabla f|^2.
\end{equation}
We write $V_{z_j\bar z_k}:=\partial_{z_j}\partial_{\bar z_k}V$
for the entries of the complex Hessian.

\begin{lemma}
\label{dim2_hormander}
Let real-valued $V\in C^2(\mathbb C^N)$ and $\kappa>0$ satisfy
\begin{equation}\label{dim2_strictlevi}
 \sum_{j,k=1}^N V_{z_j\bar z_k}(z)v_j\overline{v_k}
 \geq\kappa|v|^2,\qquad z,v\in\mathbb C^N.
\end{equation}
If $f\in L^2_{e^{-V}}(\mathbb C^N)\cap W^{1,2}_{\mathrm{loc}}(\mathbb C^N)$
and $\sum_{j=1}^N\int_{\mathbb C^N}|\partial_{\bar z_j}f|^2e^{-V} dx<\infty$, then
there is an entire holomorphic function $h\in L^2_{e^{-V}}(\mathbb C^N)$ such
that
\begin{equation}\label{dim2_hormanderbound}
 \int_{\mathbb C^N}|f-h|^2e^{-V} dx
 \leq\frac1\kappa\sum_{j=1}^N
               \int_{\mathbb C^N}|\partial_{\bar z_j}f|^2e^{-V} dx.
\end{equation}
\end{lemma}

\begin{proof}
The $(0,1)$-form
$g=\sum_{j=1}^N(\partial_{\bar z_j}f) d\bar z_j$ belongs to the weighted
$L^2$ space and satisfies $\bar\partial g=0$.
H\"ormander's theorem \cite[Theorem~2.2.1$'$]{dim2_H}, applied with
\eqref{dim2_strictlevi}, gives a solution of $\bar\partial u=g$ such
that
\[
 \int_{\mathbb C^N}|u|^2e^{-V} dx
 \leq\frac1\kappa\sum_{j=1}^N\int_{\mathbb C^N}|\partial_{\bar z_j}f|^2e^{-V} dx.
\]
The function $f-u$ is locally square integrable and
$\bar\partial(f-u)=0$, i.e.\,this is an entire function, say, $h$. Since $f,u\in L^2_{e^{-V}}(\mathbb C^N)$, the
same holds for $h$, and the estimate \eqref{dim2_hormanderbound} follows.
\end{proof}

The following observation is the simultaneous rotational argument used in
\cite[Lemma~2.7]{dim2_Chafai} and
\cite[Lemma~2.2]{dim2_Suzuki}.

\begin{lemma}\label{dim2_rotation}
Let $\mu$ be a probability measure on $\mathbb C^N$, invariant under
$z\mapsto e^{\mathrm i\theta}z$. If $f\in L_{\mu}^2$ is real-valued
and $\int_{\mathbb C^N} f d\mu=0$, then, for every entire holomorphic
$h\in L_\mu^2$,
\begin{equation}\label{dim2_rotationbound}
 \int_{\mathbb C^N} f^2 d\mu\leq2\int_{\mathbb C^N}|f-h|^2 d\mu.
\end{equation}
\end{lemma}

\begin{proof}
For every entire holomorphic $g\in L_\mu^1$, the mean-value formula
on the complex line $\lambda\mapsto\lambda z$, followed by Fubini's
theorem, gives
\begin{equation}\label{dim2_meanvalue}
 \int_{\mathbb C^N} g d\mu
 =\int_{\mathbb C^N}\frac1{2\pi}\int_0^{2\pi}
                  g(e^{\mathrm i\theta}z)\dd\theta d\mu(z)
 =g(0).
\end{equation}
Fubini's theorem applies because rotation invariance makes the
integral of $|g(e^{\mathrm i\theta}z)|$ over the product measure equal
to $\int_{\mathbb C^N}|g| d\mu$.

By \eqref{dim2_meanvalue}, subtracting $h(0)$ from $h$ makes its
mean zero and cannot increase its distance from $f$. We may therefore
assume that $h(0)=0$. Applying \eqref{dim2_meanvalue} to $h^2$ gives
$\int_{\mathbb C^N} h^2 d\mu=0$, and hence
\[
 \int_{\mathbb C^N}(\Real h)^2 d\mu=\int_{\mathbb C^N}(\Imag h)^2 d\mu.
\]
Since $f=\Real f$,
\begin{align*}
 \int_{\mathbb C^N} f^2 d\mu
 &\leq2\int_{\mathbb C^N}(f-\Real h)^2 d\mu+2\int_{\mathbb C^N}(\Real h)^2 d\mu\\
 &=2\int_{\mathbb C^N}(f-\Real h)^2 d\mu+2\int_{\mathbb C^N}(\Imag h)^2 d\mu\\
 &=2\int_{\mathbb C^N}|f-h|^2 d\mu.
\end{align*}
\end{proof}

\begin{proof}[Proof of Theorem~\ref{dim2_gap}]

Step 1. Let $\varepsilon>0$ and $\nu\geq0$. The measure in
\eqref{confined-measure}, with $d=2$ and $\xi=0$, is
\[
 d\mu_{s,0,\varepsilon}(x)
 =\frac{\Gamma_s\psi_\varepsilon}
        {\br{\Gamma_s}{\psi_\varepsilon}}d x.
\]
Put
\[
 V_\varepsilon(z)=\frac{|z|^2}{4s}
      +\frac\nu{2N}\sum_{1\leq i<j\leq N}
                  \log(|z_i-z_j|^2+\varepsilon).
\]
The probability measure $\mu_{s,0,\varepsilon}$ is a normalization of
$e^{-V_\varepsilon}\dd x$. For every $v\in\mathbb C^N$,
\begin{equation}\label{dim2_levi}
 \sum_{j,k=1}^N(V_\varepsilon)_{z_j\bar z_k}v_j\overline{v_k}
 =\frac{|v|^2}{4s}
  +\frac{\nu\varepsilon}{2N}\sum_{1\leq i<j\leq N}
       \frac{|v_i-v_j|^2}{(|z_i-z_j|^2+\varepsilon)^2}
 \geq\frac{|v|^2}{4s}.
\end{equation}
If the left-hand side of \eqref{dim2_ineq} is infinite, we have nothing to do. Otherwise, we subtract from $f$ the $\mu_{s,0,\varepsilon}$-mean of $f$ so that we can apply Lemma~\ref{dim2_rotation} in a moment. 
Now, for thus modified $f$, we apply Lemma~\ref{dim2_hormander}, \eqref{dim2_levi} and
\eqref{dim2_realcomplex} to obtain an entire function $h\in L^2(\mu_{s,0,\varepsilon})$
such that
\begin{align*}
 \int_{\mathbb R^{2N}}|f-h|^2 d\mu_{s,0,\varepsilon}
 &\leq4s\sum_{j=1}^N\int_{\mathbb R^{2N}}|\partial_{\bar z_j}f|^2 d\mu_{s,0,\varepsilon}\\
 &=s\int_{\mathbb R^{2N}}|\nabla f|^2 d\mu_{s,0,\varepsilon}.
\end{align*}
The measure $\mu_{s,0,\varepsilon}$ is invariant under simultaneous
rotations. Lemma~\ref{dim2_rotation} therefore yields
\[
 \int_{\mathbb R^{2N}} f^2 d\mu_{s,0,\varepsilon}
 \leq2s\int_{\mathbb R^{2N}}|\nabla f|^2 d\mu_{s,0,\varepsilon},
\]
which yields \eqref{dim2_ineq} for $\varepsilon>0$. 

\smallskip

Step 2. Let $0\leq\nu<4$ and suppose that $f$ has finite Dirichlet energy for the
singular weight (i.e.\,$\varepsilon=0$). Since $\psi_\varepsilon\leq\psi$, the preceding
estimate applies to $f$ for every $\varepsilon>0$. That is, recalling that at the previous step we have subtracted the weighted mean of $f$, we have
\begin{equation}\label{dim2_limitvariance}
 \br{\Gamma_s|\nabla f|^2}{\psi_\varepsilon}
 \geq\frac1{2s}\left(\br{\Gamma_s f^2}{\psi_\varepsilon}
-\frac{\br{\Gamma_s f}{\psi_\varepsilon}^{\,2}}
{\br{\Gamma_s}{\psi_\varepsilon}}\right).
\end{equation}
We have
$\br{\Gamma_s|f|}{\psi}<\infty$ by Cauchy-Schwarz and
Lemma~\ref{dim2_mass}, so the dominated convergence theorem applies to
$\br{\Gamma_s f}{\psi_\varepsilon}$. Further, as $\varepsilon \downarrow 0$, the monotone convergence applies to the other terms. Therefore, we can pass to the limit in \eqref{dim2_limitvariance} to arrive at the sought spectral gap inequality for $\varepsilon=0$.

\smallskip

Step 3. To prove sharpness of spectral gap $1/2s$, let us recall the decomposition from
Lemma~\ref{dim2_mass}:
\[
 q=\frac1{\sqrt N}\sum_{i=1}^Nx^i,\qquad
 y^i=x^i-\frac q{\sqrt N},\qquad
 H=\left\{y\in(\mathbb R^2)^N \mid \sum_{i=1}^Ny^i=0\right\}.
\]
The map
\[
 (q,y)\longmapsto
 \left(\frac q{\sqrt N}+y^1,\ldots,
       \frac q{\sqrt N}+y^N\right)
\]
is an orthogonal identification of $\mathbb R^2\times H$ with
$\mathbb R^{2N}$. Indeed, $\sum_{i=1}^Ny^i=0$ gives
\[
 |x|^2=\sum_{i=1}^N\left|\frac q{\sqrt N}+y^i\right|^2
      =|q|^2+|y|^2.
\]
Writing $d y_H$ for the induced Lebesgue measure on $H$, we
therefore have $d x=d q\,d y_H$. Moreover,
$x^i-x^j=y^i-y^j$, so $\psi_\varepsilon(x)=\psi_\varepsilon(y)$.

Put
\[
 Z_{s,\varepsilon}^H
 :=\int_H e^{-|y|^2/(4s)}\psi_\varepsilon(y) d y_H.
\]
Then
\[
 \br{\Gamma_s}{\psi_\varepsilon}
 =(4\pi s)^{-N}
   \left(\int_{\mathbb R^2}e^{-|q|^2/(4s)} d q\right)
   Z_{s,\varepsilon}^H
 =(4\pi s)^{1-N}Z_{s,\varepsilon}^H.
\]
Thus $0<Z_{s,\varepsilon}^H<\infty$ by Lemma~\ref{dim2_mass}.
In these coordinates, the normalized measure is
\[
 d\mu_{s,0,\varepsilon}(q,y)
 =\left((4\pi s)^{-1}e^{-|q|^2/(4s)} d q\right)
  \left((Z_{s,\varepsilon}^H)^{-1}
        e^{-|y|^2/(4s)}\psi_\varepsilon(y) d y_H\right).
\]
Both factors are probability measures. Hence, under
$\mu_{s,0,\varepsilon}$, the variables $q$ and $y$ are independent,
and $q$ is a centred Gaussian with covariance $2sI_2$.

Now, take in the just established spectral gap inequality the  function
\[
 f(x)=q_1=\frac1{\sqrt N}\sum_{i=1}^N(x^i)_1.
\]
Then
\[
 \int_{\mathbb R^{2N}}f d\mu_{s,0,\varepsilon}=0,\qquad
 \int_{\mathbb R^{2N}}f^2 d\mu_{s,0,\varepsilon}
 =\int_{\mathbb R^2}q_1^2(4\pi s)^{-1}
                 e^{-|q|^2/(4s)}d q=2s.
\]
Also, $\nabla_i f=(N^{-1/2},0)$ for every $i$, and hence
$|\nabla f|^2=\sum_{i=1}^N N^{-1}=1$. In particular, $f$ is an admissible function in the theorem. Since its mean is zero, the minimizing
constant in \eqref{dim2_ineq} is $c=0$, and
\[
 \br{\Gamma_s|\nabla f|^2}{\psi_\varepsilon}
 =\br{\Gamma_s}{\psi_\varepsilon},\qquad
 \inf_{c\in\mathbb R}\br{\Gamma_s|f-c|^2}{\psi_\varepsilon}
 =2s\br{\Gamma_s}{\psi_\varepsilon}.
\]
Thus equality holds in \eqref{dim2_ineq}, and so no larger constant
than $1/2s$ is possible in Theorem \ref{dim2_gap}.
\end{proof}

\begin{remark} Let us compare Theorem \ref{dim2_gap} with Theorem \ref{full-gap}. With a few additional efforts, we can prove the analogous spectral gap inequality for the cut-off (i.e.\,bounded away from zero) weight $\varphi$. However, including an arbitrary centre $\xi$ in the Gaussian density is problematic for the complex-analytic proof of Theorem \ref{dim2_gap} since this breaks the simultaneous rotation argument (unless of course $\xi_1=\dots=\xi_N$, i.e.\,all centres coincide, in which case we can simply make a change of variables) .
\end{remark}

\bigskip

\appendix

\section{Direct proof of the lower heat kernel bound}
\label{direct-heat}

\subsection{The spectral gap inequality for the cutoff weight}
\label{cutoff-gap}

We use the cutoff weight defined in \eqref{eta_def}--\eqref{phi_s_def}.

\begin{theorem}
\label{thm_spectral_gap}
For all $d\ge3$, $N \ge 2$, if the strength of attraction between the particles $\nu>0$ satisfies
\begin{equation}
  0<\nu<\frac{(d-2)^2}{5d}.
  \label{nu-condition}
\end{equation}
For $\xi\in\mathbb R^{dN}$ put
\[
 m_{s,\varepsilon}(\xi)
 :=\br{\Gamma_{s,\xi}}{\varphi_{s,\varepsilon}},
 \qquad
 \bar f^{\,\xi}_{s,\varepsilon}
 :=m_{s,\varepsilon}(\xi)^{-1}
   \br{\Gamma_{s,\xi}f}{\varphi_{s,\varepsilon}}.
\]
Then
\begin{equation}
 \br{\Gamma_{s,\xi}|\nabla f|^2}{\varphi_{s,\varepsilon}}
 \geq\frac Ms
 \br{\Gamma_{s,\xi}
       |f-\bar f^{\,\xi}_{s,\varepsilon}|^2}
      {\varphi_{s,\varepsilon}},
 \label{translated-gap}
\end{equation}
uniformly in $\xi,s$, and $\varepsilon$,
where the constant
$$
  M:=\frac{(d-2)^2-5d\nu}{2((d-2)^2-d\nu)}
$$
and also in the limiting case $\varepsilon=0$
\begin{equation}
  \br{\Gamma_{s,\xi}|\nabla f|^2}{\varphi_{s}}   \ge \frac{M}{s}  \br{\Gamma_{s,\xi}|f-\bar{f}_{s}|^2}{\varphi_{s}}.
  \label{main-gap2}
\end{equation}
\end{theorem}

The proof of Theorem \ref{thm_spectral_gap} follows the proof of Theorem \ref{full-gap} with a few modifications needed to account for the fact that the weight $\varphi$ is defined in terms of functions $a$ and $\eta$, that is, it is bounded from below by  the constant $c_\varphi>0$.

We prove in the same way as before:

\begin{lemma}
\label{mass2}
There is a constant $C<\infty$ such that
for every $s>0$, $\varepsilon\ge0$, and
$\xi\in\mathbb R^{dN}$,
\begin{equation}
 c_\varphi
 \leq\br{\Gamma_{s,\xi}}{\varphi_{s,\varepsilon}}
 \leq C,
 \qquad
 \br{\Gamma_{s,\xi}
       \left(1+\frac{|\,\cdot-\xi|^2}{s}\right)}
      {\varphi_{s,\varepsilon}}
 \leq C.
 \label{mass-bound}
\end{equation}
\end{lemma}

\begin{lemma}\label{interaction}
The following assertions hold.
\begin{enumerate}
\item At every point such that
$|z|_\varepsilon\in(0,\infty)\setminus\{1,2\}$, 
$a(|z|_\varepsilon)$ is twice differentiable and
\begin{equation}  \label{chain-rule}
  \nabla^2 a(|z|_\varepsilon)
  =\frac{a'(|z|_\varepsilon)}{|z|_\varepsilon} I
  +\left(
    \frac{a''(|z|_\varepsilon)}{|z|_\varepsilon^2}
    -\frac{a'(|z|_\varepsilon)}{|z|_\varepsilon^3}
  \right)z\otimes z. 
\end{equation}

\item For almost every $z\ne0$ and every $\xi\in\mathbb{R}^d$,
\begin{equation}   \label{second-derivative-a}
  -\frac{|\xi|^2}{|z|^2}
  \le \xi\cdot\nabla^2 a(|z|_\varepsilon) \xi
  \le \frac{2|\xi|^2}{|z|^2}.
\end{equation}
The same inequalities hold in the sense of distributions.

\item For every $z\ne0$,
\begin{equation}
  |\nabla a(|z|_\varepsilon)|\le\frac1{|z|}.
  \label{gradient-a}
\end{equation}
\end{enumerate}
\end{lemma}

\begin{proof}
Set $r:=|z|_\varepsilon$.  Since
$$
  \nabla r=\frac{z}{r},  \qquad \nabla^2 r=\frac{1}{r} I -\frac{1}{r^3}z\otimes z,
$$
the chain rule gives \eqref{chain-rule}. We next prove the second-order estimate.  For $z\ne0$, write
$$
  \xi=\xi_{\mathrm T}+\xi_{\mathrm R},
  \qquad
  \xi_{\mathrm R}:=\frac{\xi\cdot z}{|z|^2}z,
  \qquad
  \xi_{\mathrm T}:=\xi-\xi_{\mathrm R}.
$$
Thus $\xi_{\mathrm T}$ is orthogonal to $z$ and $\xi_{\mathrm R}$ is parallel to $z$.
We consider three cases:
\begin{enumerate}
\item If $0<r<1$, then
$$
  a'(r)=-\frac1r,
  \qquad
  a''(r)=\frac1{r^2},
$$
and hence
\begin{equation}
  \nabla^2 a(|z|_\varepsilon)
  =-\frac1{r^2} I +\frac2{r^4}z\otimes z.
  \label{second-derivative-first}
\end{equation}
The coefficients in the tangential and radial directions are, respectively,
$$
  \lambda_{\mathrm T}=-\frac1{r^2},
  \qquad
  \lambda_{\mathrm R}
  =\frac{|z|^2-\varepsilon}{r^4}
  =\frac{2|z|^2}{r^4}-\frac1{r^2}.
$$
Since $r^2=|z|^2+\varepsilon$,
$$
  -\frac1{|z|^2}\le\lambda_{\mathrm T}\le0,
  \qquad
  -\frac1{|z|^2}
  \le-\frac1{r^2}
  \le\lambda_{\mathrm R}
  \le\frac1{r^2}
  \le\frac1{|z|^2}.
$$

\item If $1<r<2$, then
$$
  a'(r)=1-\frac2r,
  \qquad
  a''(r)=\frac2{r^2},
$$
so that
\begin{equation}
  \nabla^2 a(|z|_\varepsilon)
  =\frac{r-2}{r^2} I +\frac{4-r}{r^4}z\otimes z.
  \label{second-derivative-second}
\end{equation}
Here
$$
  -\frac1{r^2}\le\lambda_{\mathrm T}
  =\frac{r-2}{r^2}\le0.
$$
The rank-one term in \eqref{second-derivative-second} is nonnegative; therefore $\lambda_{\mathrm R}\ge\lambda_{\mathrm T}$ and both coefficients are bounded below by $-|z|^{-2}$.  Moreover,
$$
  \lambda_{\mathrm R}
  =\frac{2|z|^2+(r-2)\varepsilon}{r^4}
  \le\frac{2|z|^2}{r^4}
  \le\frac2{|z|^2}.
$$
\item If $r>2$,  $a$ is constant, so both coefficients vanish.
\end{enumerate}

The preceding bounds, together with the orthogonal decomposition of $\xi$, prove \eqref{second-derivative-a}.

It remains to estimate the first derivative.  From
$
  \nabla a(|z|_\varepsilon) =a'(r)\frac{z}{r}
$
and $|a'(r)|\le r^{-1}$ for $0<r<2$, while $a'(r)=0$ for $r>2$, we obtain
\[
  |\nabla a(|z|_\varepsilon)|
  \le\frac{|z|}{r^2}
  \le\frac1{|z|},
\]
which is \eqref{gradient-a}.

Finally,
\[
  a'(1^-)=a'(1^+)=-1,
  \qquad
  a'(2^-)=a'(2^+)=0.
\]
For $\varepsilon>0$, the continuity of $a'$ at $1$ and $2$ implies that
$a \in C^{1,1}_{\mathrm{loc}}(\mathbb{R}^d)$. Therefore, no surface term is supported on $\{r=1\}$ or
$\{r=2\}$.  
\end{proof}

The previous lemma yields:

\begin{lemma}
For every smooth $f$, we have, everywhere outside of the collision set:
\begin{equation}
  (\nabla^2\log\varphi_\varepsilon\,\nabla f)\cdot\nabla f
  \le\frac{2\nu}{N}\ipair
      \frac{|\nabla_i f-\nabla_j f|^2}{|y^i-y^j|^2}
  \le\frac{4\nu}{N}\opair
      \frac{|\nabla_i f|^2}{|y^i-y^j|^2},
  \label{logphi-second-order}
\end{equation}
and
\begin{equation}
  -\Delta\log\varphi_\varepsilon
  \le\frac{2d\nu}{N}\ipair\frac1{|y^i-y^j|^2}.
  \label{logphi-laplacian}
\end{equation}
For each $i\ne j$,
\begin{equation}
\begin{aligned}
  -&(\nabla_i-\nabla_j)\cdot(\nabla_i-\nabla_j)\log\varphi_\varepsilon\\
  &\le\frac{d\nu}{N}\left[
    \frac4{|y^i-y^j|^2}
    +\sum_{k\ne i,j}\left(
      \frac1{|y^i-y^k|^2}+\frac1{|y^j-y^k|^2}
    \right)
  \right].
\end{aligned}
  \label{logphi-directional}
\end{equation}
\end{lemma}

The rest of the proof repeats the proof of Theorem \ref{full-gap}.

\subsection{The auxiliary heat kernel}

 Our goal is to estimate the approximating heat kernels $p_\varepsilon(t,x,y)$ from below  with constants independent of $\varepsilon>0$.
The bounds on the limiting heat kernel $p(t,x,y)$ will follow using the semigroup convergence \eqref{lim_exists} and the Lebesgue differentiation theorem. 

It will be convenient to introduce drifts corresponding to weights $\varphi_{s}$, $\varphi_{s,\varepsilon}$:
\begin{equation}
\label{b_var}
\bar{b}(y)=-\frac{\nabla \varphi_{s}}{\varphi_{s}}, \quad
\bar{b}_\varepsilon(y)=-\frac{\nabla \varphi_{s,\varepsilon}}{\varphi_{s,\varepsilon}},
\end{equation}
which will appear in the intermediate calculations. These drifts depend on $s$.

Put $\varphi=\varphi_{s,\varepsilon}$.
Let
$A$ be the self-adjoint operator in $L^2_\varphi := L^2(\mathbb R^{dN},\varphi(y)\,dy)$ associated with
\begin{equation} 
 a[u,w]:=\langle\nabla u,\nabla w\rangle_\varphi,
 \qquad D(a)=W^{1,2}(\mathbb R^{dN}).
 \label{form}
\end{equation}
On $C_c^\infty(\mathbb R^{dN})$,
\begin{equation}
 A=-\varphi^{-1}\operatorname{div}(\varphi\nabla)
   =(-\nabla+\bar b_\varepsilon)\cdot\nabla.
 \label{A_def}
\end{equation}
We denote the kernel of $e^{-tA}$ with respect to
$\varphi(y)\,dy$ by $q_{s,\varepsilon}(t,x,y)$. Let us note that $q_{s,\varepsilon}(t,x,y)\varphi_{s,\varepsilon}(y)$ is the integral kernel of $e^{-tA}$ with respect to the Lebesgue measure.

We now compare $q_\varepsilon \varphi$ to $p_\varepsilon$, the heat kernel of $e^{-t\Lambda_\varepsilon}$, $\Lambda_\varepsilon=(-\nabla + b_\varepsilon) \cdot \nabla$. Notice that $A$ and $\Lambda_\varepsilon$ differ by $h_\varepsilon \cdot \nabla$, where 
$$
 h_\varepsilon:=b_\varepsilon-\bar b_\varepsilon
$$
is a bounded drift, as the next lemma shows.

\begin{lemma}[Comparing heat kernels $q$ and $p$]
\label{gaussian-transfer}
We have, uniformly in $s,\varepsilon>0$,
\begin{equation}
 \|h_\varepsilon\|_\infty\leq \frac{C}{\sqrt s}.
 \label{bb_est}
\end{equation}
Let $M_0>0$ and $0<t\leq M_0s$. The following are true:

\begin{enumerate}[label=(\alph*)]
\item If, at time $t$,
\[
 q_{s,\varepsilon}(t,x,y)
 \leq K_0t^{-dN/2}
 \exp\left(-\frac{|x-y|^2}{K_1t}\right),
\]
then, after changing the constants in the Gaussian,
\begin{equation}
 p_\varepsilon(t,x,y)
 \leq K_2t^{-dN/2}
 \exp\left(-\frac{|x-y|^2}{K_3t}\right)
 \varphi_{s,\varepsilon}(y).
 \label{gaussian-transfer-upper}
\end{equation}
Here $K_2,K_3$ depend on $M_0,K_0,K_1,d,N,$ and $\nu$, but not on
$s$ or $\varepsilon$.

\item If, for some $\alpha,c_0>0$,
\[
 q_{s,\varepsilon}(t,x,y)\geq c_0t^{-dN/2},
 \qquad |x-y|\leq \alpha\sqrt t,
\]
then
\begin{equation}
 p_\varepsilon(t,x,y)
 \geq c_1t^{-dN/2}\varphi_{s,\varepsilon}(y),
 \qquad |x-y|\leq \alpha\sqrt t,
 \label{gaussian-transfer-lower}
\end{equation}
where $c_1>0$ depends on $M_0,\alpha,c_0,d,N,$ and $\nu$, but not on
$s$ or $\varepsilon$.
\end{enumerate}
\end{lemma}
\begin{proof}

The operator-theoretic argument given below is a bit of overkill since we need this lemma at the a priori level, so all heat kernels that appear in the proof are smooth, and we can appeal instead to the classical theory and the Feynman-Kac formula. We opt to work in the more general setting to give an argument that, in principle, also works directly for $\varepsilon=0$.

Recall that
$$
\varphi(x) \equiv \varphi_{s,\varepsilon}=\prod_{1\le i<j\le N}\eta(s^{-\frac{1}{2}}|x^i-x^j|_\varepsilon) =e^{\frac{\nu}{N}\sum_{1\le i<j\le N}a(s^{-\frac{1}{2}}|x^i-x^j|_\varepsilon)}.
$$
In turn,
$$
\psi(x) \equiv \psi_{s,\varepsilon}=\prod_{1\le i<j\le N}(s^{-\frac{1}{2}}|x^i-x^j|_\varepsilon)^{-\frac{\nu}{N}}=e^{-\frac{\nu}{N}\sum_{i<j}\log(s^{-\frac{1}{2}}|x^i-x^j|_\varepsilon)}.
$$
Below the constants denoted by $C$ may change from line to line and depend on
$d,N$, and $\nu$, but not on $s$ or $\varepsilon$.

We divide the proof into three steps.

\smallskip
\noindent \emph{Step 1} (Proof of the estimate \eqref{bb_est} on the difference of the two drifts).
We have
\[
 b_\varepsilon=-\nabla\log\psi,
 \qquad
 \bar b_\varepsilon=-\nabla\log\varphi,
\]
and therefore
\[
 h_\varepsilon
 =b_\varepsilon-\bar b_\varepsilon
 =\nabla\log(\varphi/\psi)
 =:\nabla R.
\]
Thus, to prove \eqref{bb_est}, we need to estimate $\|\nabla R\|_\infty$.

It follows from the definitions of $\varphi$ and $\psi$ that $ R=\log\frac{\varphi}{\psi}$ admits the representation
\begin{equation}
 R(x)=\frac{\nu}{N}\sum_{i<j}
 g\left(s^{-1/2}|x^i-x^j|_\varepsilon\right), \qquad  g(r):=a(r)+\log r.
 \label{R-pair-sum}
\end{equation}
By the definition of $a$ in \eqref{eta_def},
\[
 g(r)=
 \begin{cases}
 0, &0<r<1,\\
 r-1-\log r, &1\leq r\leq2,\\
 1-2\log2+\log r, &r>2.
 \end{cases}
\]
Therefore,
\[
 g'(r)=
 \begin{cases}
 0, &0<r<1,\\
 1-r^{-1}, &1<r<2,\\
 r^{-1}, &r>2,
 \end{cases}
\]
so $g'$ is continuous and globally Lipschitz, 
and
\begin{equation}
 \|g'\|_\infty\leq\frac12,
 \qquad
 |g''(r)|\leq1
 \quad\text{for a.e. }r>0.
 \label{g-derivative-bounds}
\end{equation}

For $z\in\mathbb R^d$, put
\[
 r:=s^{-1/2}|z|_\varepsilon, \qquad
 F_{s,\varepsilon}(z):=g(r).
\]
Then
\[
 \nabla_z r=\frac{1}{\sqrt s}\frac{z}{|z|_\varepsilon},
 \qquad
 \nabla_z^2r
 =\frac{1}{\sqrt s\,|z|_\varepsilon}
   \bigl(I-\frac{z \otimes z}{|z|^2_\varepsilon}\bigr).
\]
Thus, almost everywhere,
\[
 \nabla_z^2F_{s,\varepsilon}
 =
 \frac1s g''(r)\frac{z \otimes z}{|z|^2_\varepsilon}
 +
 \frac{g'(r)}{\sqrt s\,|z|_\varepsilon}
 \bigl(I-\frac{z \otimes z}{|z|^2_\varepsilon}\bigr).
\]
If $r<1$, then $g'(r)=g''(r)=0$. If $r\geq1$, then
$|z|_\varepsilon\geq\sqrt s$. It follows from
\eqref{g-derivative-bounds} that
\[
 |\nabla F_{s,\varepsilon}(z)|
 \leq\frac{C}{\sqrt s},
 \qquad
 |\nabla^2F_{s,\varepsilon}(z)|
 \leq\frac Cs.
\]
Since $g'$ is continuous at $1$ and $2$, these estimates also hold
for the weak derivatives. Summing the finitely many pair contributions in
\eqref{R-pair-sum}, we obtain
\begin{equation}
 \|\nabla R\|_\infty\leq\frac{C}{\sqrt s},
 \qquad
 \|\nabla^2R\|_\infty\leq\frac Cs.
 \label{R-derivative-bounds}
\end{equation}
We do not need the estimate on the second derivatives of $R$ in the proof of \eqref{bb_est}, but we will use  it later in the proof.

\begin{remark}
Alternatively, we can obtain the first estimate as follows. Put
$
 r_{ij}:=s^{-1/2}|x^i-x^j|_\varepsilon.
$
Then
\[
 \nabla_iR
 =
 \frac{\nu}{N\sqrt s}
 \sum_{j\ne i}
 g'(r_{ij})
 \frac{x^i-x^j}{|x^i-x^j|_\varepsilon}.
\]
Since $\|g'\|_\infty\leq1/2$ and
$|x^i-x^j|\leq|x^i-x^j|_\varepsilon$, we have
$
 |\nabla_iR|
 \leq\frac{\nu(N-1)}{2N\sqrt s}.
$
Therefore,
\[
 \|\nabla R\|_\infty
 \leq
 \left(\sum_{i=1}^N\|\nabla_iR\|_\infty^2\right)^{1/2}
 \leq\frac{\nu(N-1)}{2\sqrt N\,\sqrt s}.
\]
\end{remark}

The first estimate in \eqref{R-derivative-bounds} proves
\eqref{bb_est}.

\smallskip
\noindent
\emph{Step 2} (Preliminary heat kernel comparison).
This is the principal step. Our goal is to obtain a two-sided estimate
\begin{equation}
\begin{aligned}
 &\exp\left\{-C\left(\frac ts+
              \frac{|x-y|}{\sqrt s}\right)\right\}
 q_{s,\varepsilon}(t,x,y)\varphi_{s,\varepsilon}(y)
 \leq p_\varepsilon(t,x,y)\\
 &\qquad\leq
 \exp\left\{C\left(\frac ts+
              \frac{|x-y|}{\sqrt s}\right)\right\}
 q_{s,\varepsilon}(t,x,y)\varphi_{s,\varepsilon}(y).
\end{aligned}
 \label{basic-gaussian-transfer}
\end{equation}
The exponential factors will be absorbed in the Gaussian densities in the final step (Step 3).

To establish \eqref{basic-gaussian-transfer}, we will use the Lie-Trotter product formula and the comparison between the generators of these heat kernels, i.e.\,$$A=A_\varepsilon=-\Delta-\nabla\log\varphi\cdot\nabla$$ and $\Lambda_\varepsilon$, respectively. 

1.~We will need to estimate $\|AR\|_\infty$, where, recall,
$\nabla R=h_\varepsilon=b_\varepsilon-\bar b_\varepsilon$, i.e.\,exactly the difference between the drifts in $\Lambda_\varepsilon$ and $A$.
In turn, to estimate $AR$, we must control
$\nabla\log\varphi\cdot\nabla R$. This requires taking into account some cancellation
since $\nabla\log\varphi$ is not bounded uniformly in $\varepsilon$:

For $i\ne j$, denote the contribution of the pair $(i,j)$ to
$\nabla_i\log\varphi$ by
\[
 G_{ij}
 :=
 \frac{\nu}{N\sqrt s}\,
 a'(r_{ij})
 \frac{x^i-x^j}{|x^i-x^j|_\varepsilon}.
\]
Then
 $G_{ji}=-G_{ij}$,
$\nabla_i\log\varphi=\sum_{j\ne i}G_{ij}$.
The definition of $a$ gives
$
 |r a'(r)|\leq 1$.
It follows that, for $r_{ij}:=s^{-1/2}|x^i-x^j|_\varepsilon$,
\begin{align}
 |G_{ij}|\,|x^i-x^j|
 &=
 \frac{\nu}{N}|r_{ij}a'(r_{ij})|
 \frac{|x^i-x^j|^2}{|x^i-x^j|_\varepsilon^2}
 \notag\\
 &\leq\frac{\nu}{N}.
 \label{Gij-bound}
\end{align}

Next, put
\[
 H_i:=\nabla_iR
\]
and denote by $\widehat x$ the configuration obtained by replacing in $x$ both
$x^i$ and $x^j$ by their average
\[
 \widehat x^i=\widehat x^j:=\frac{x^i+x^j}{2}
\]
and leaving all other particle coordinates unchanged. Since $R$ is
invariant under permutations of the particle labels and
$\widehat x^i=\widehat x^j$, we have
\[
 H_i(\widehat x)=H_j(\widehat x).
\]
The second derivatives bound in \eqref{R-derivative-bounds} and the mean value theorem therefore give
\begin{align}
 |H_i(x)-H_j(x)|
 &\leq
 |H_i(x)-H_i(\widehat x)|
 +|H_j(\widehat x)-H_j(x)|
 \notag\\
 &\leq\frac Cs|x-\widehat x|
 \leq\frac Cs|x^i-x^j|.
 \label{Hij-bound}
\end{align}

We have
\begin{align*}
 \nabla\log\varphi\cdot\nabla R
 &=
 \sum_{i=1}^N\sum_{j\ne i}G_{ij}\cdot H_i\\
 &=
 \sum_{i<j}G_{ij}\cdot(H_i-H_j).
\end{align*}
Combining \eqref{Gij-bound} and \eqref{Hij-bound} yields
\[
 \left\|
 \nabla\log\varphi\cdot\nabla R
 \right\|_\infty
 \leq\frac Cs,
\]
where we have used the fact that $|x^i-x^j|$ coming from the bound on $H_i-H_j$ is what is needed to cancel the singularity in $G_{ij}$, see the previous estimate.

Since $AR=-\Delta R-\nabla\log\varphi\cdot\nabla R$,
the second derivatives estimate in \eqref{R-derivative-bounds} now gives
\[
 \|AR\|_\infty\leq\frac Cs.
\]

2.~Set
\[
 V:=\frac14|\nabla R|^2+\frac12AR.
\]
By the previous estimates, we have
\begin{equation}
 \|V\|_\infty\leq\frac Cs.
 \label{gauge-potential-bound}
\end{equation}

3.~Define  the unitary map
\[
 U:L^2(\mathbb R^{dN},\psi\,dx)
   \rightarrow L^2(\mathbb R^{dN},\varphi\,dx),
 \qquad
Uf=\left(\frac{\psi}{\varphi}\right)^{1/2} f.
\]
The operator $\Lambda_\varepsilon$ is a self-adjoint operator in $L^2(\mathbb R^{dN},\psi\,dx)$, the operator $A$ is a self-adjoint operator in $L^2(\mathbb R^{dN},\varphi\,dx)$. They have sesquilinear forms $\ell [u,v]=\langle \nabla u,\nabla v\rangle_\psi$ and $a[f,g]=\langle \nabla f,\nabla g\rangle_\varphi$ acting in different spaces. Let us compare the quadratic forms of $U\Lambda_\varepsilon U^{-1}$ and $A+V$ that are defined on the same space: for real-valued
$f\in C_c^\infty(\mathbb R^{dN})$,
\begin{align*}
 \ell[U^{-1}f,U^{-1}f]= 
 \int_{\mathbb R^{dN}}
 \left|\nabla(\left(\frac{\psi}{\varphi}\right)^{-1/2}f)\right|^2\psi\,dx 
 & =
 \int_{\mathbb R^{dN}}
 \left|\nabla f+\frac12f\nabla R\right|^2\varphi\,dx\\
 &=
 \int_{\mathbb R^{dN}}|\nabla f|^2\varphi\,dx
 +\int_{\mathbb R^{dN}}
   f\nabla f\cdot\nabla R\,\varphi\,dx\\
 &\quad
 +\frac14\int_{\mathbb R^{dN}}
   f^2|\nabla R|^2\varphi\,dx.
\end{align*}
By the definition of $A$ and the integration by parts,
\[
 \int_{\mathbb R^{dN}}
 f\nabla f\cdot\nabla R\,\varphi\,dx
 =
 \frac12\int_{\mathbb R^{dN}}f^2AR\,\varphi\,dx,
\]
so
$
 \ell[U^{-1}f,U^{-1}f]=a[f,f]+\langle Vf,f\rangle_\varphi.
$
By the polarization identity, the corresponding sesquilinear forms coincide:
$$
 \ell[U^{-1}f,U^{-1}g]=a[f,g]+\langle Vf,g\rangle_\varphi.
 $$
Invoking the definition of $U$, one sees that $\ell[U^{-1}f,U^{-1}g]=\langle U\Lambda_\varepsilon U^{-1}f,g\rangle_\varphi$.
This identity extends to the domains of the operators involved, which yields
\begin{equation}
 U\Lambda_\varepsilon U^{-1}=A+V.
 \label{gauge-conjugation}
\end{equation}

4.~Since $e^{-tA}$ is positivity preserving, the Lie-Trotter product formula
and \eqref{gauge-potential-bound} give, for every nonnegative $f$,
\begin{equation}
 e^{-Ct/s}e^{-tA}f
 \leq e^{-t(A+V)}f
 \leq e^{Ct/s}e^{-tA}f.
 \label{bounded-potential-domination}
\end{equation}

5.~Let $k_V(t,x,y)$ denote the kernel of $e^{-t(A+V)}$ with respect to
$\varphi(y)\,dy$. Since $q_{s,\varepsilon}$
is the kernel of $e^{-tA}$ with respect to the same measure,
\eqref{bounded-potential-domination} implies
\begin{equation}
 e^{-Ct/s}q_{s,\varepsilon}(t,x,y)
 \leq k_V(t,x,y)
 \leq e^{Ct/s}q_{s,\varepsilon}(t,x,y).
 \label{kV-q-comparison}
\end{equation}
On the other hand, \eqref{gauge-conjugation} gives
\begin{align*}
 e^{-t(A+V)}f(x)
 &=
 \rho(x)e^{-t\Lambda_\varepsilon}(\rho^{-1}f)(x), \qquad \quad \rho(x):=\left(\frac{\psi(x)}{\varphi(x)}\right)^{1/2}\\
 &=
 \rho(x)\int_{\mathbb R^{dN}}
 p_\varepsilon(t,x,y)\rho(y)^{-1}f(y)\,dy.
\end{align*}
Hence
\begin{equation}
 k_V(t,x,y)
 =
 \frac{\rho(x)}{\rho(y)\varphi(y)}
 p_\varepsilon(t,x,y),
 \qquad
 p_\varepsilon(t,x,y)
 =
 \frac{\rho(y)}{\rho(x)}
 \varphi(y)k_V(t,x,y).
 \label{kV-p-relation}
\end{equation}
Furthermore, \eqref{R-derivative-bounds} gives
\[
 |\log\rho(y)-\log\rho(x)|
 =\frac12|R(y)-R(x)|
 \leq C\frac{|x-y|}{\sqrt s}.
\]
Combining this estimate with
\eqref{kV-q-comparison} and \eqref{kV-p-relation}, we obtain \eqref{basic-gaussian-transfer}.

\smallskip
\noindent
\emph{Step 3} (Exponential factors -- to Gaussians).
Suppose first that the upper bound in part~\textup{(a)} holds. The
upper estimate in \eqref{basic-gaussian-transfer} gives
\[
 p_\varepsilon(t,x,y)
 \leq K_0t^{-dN/2}\varphi_{s,\varepsilon}(y)
 \exp\left\{
 -\frac{|x-y|^2}{K_1t}
 +C\frac{|x-y|}{\sqrt s}
 +C\frac ts
 \right\}.
\]
For every $\delta>0$, the quadratic inequality gives
\begin{equation}
 C\frac{|x-y|}{\sqrt s}
 \leq
 \delta\frac{|x-y|^2}{t}
 +\frac{C^2}{4\delta}\frac ts.
 \label{linear-gaussian-absorption}
\end{equation}
Choose $\delta=(2K_1)^{-1}$. Since \(t/s\leq M_0\),
\[
 -\frac{|x-y|^2}{K_1t}
 +\delta\frac{|x-y|^2}{t}
 =
 -\frac{|x-y|^2}{2K_1t},
\]
and all remaining exponential terms are bounded by a constant
depending only on $M_0,K_1,d,N$, and $\nu$. Thus
\[
 p_\varepsilon(t,x,y)
 \leq K_2t^{-dN/2}
 \exp\left(-\frac{|x-y|^2}{2K_1t}\right)
 \varphi_{s,\varepsilon}(y).
\]
This proves part~\textup{(a)}, for example with $K_3=2K_1$.

Finally, suppose that the near-diagonal lower bound in
part~\textup{(b)} holds. If
\[
 |x-y|\leq\alpha\sqrt t,
 \qquad
 t\leq M_0s,
\]
then
\[
 \frac ts\leq M_0,
 \qquad
 \frac{|x-y|}{\sqrt s}
 \leq\alpha\sqrt{\frac ts}
 \leq\alpha\sqrt{M_0}.
\]
The lower estimate in \eqref{basic-gaussian-transfer} therefore
gives
\[
 p_\varepsilon(t,x,y)
 \geq
 c_0\exp\{-C(M_0+\alpha\sqrt{M_0})\}
 t^{-dN/2}\varphi_{s,\varepsilon}(y).
\]
This is \eqref{gaussian-transfer-lower}, with
\[
 c_1
 :=
 c_0\exp\{-C(M_0+\alpha\sqrt{M_0})\}>0,
\]
and proves part~\textup{(b)}. 
\end{proof}

From now on, we fix $s,\varepsilon>0$ and use the abbreviations
$q=q_{s,\varepsilon}$ and $\varphi=\varphi_{s,\varepsilon}$ introduced
above. We will be considering $t$ that lie in the interval
$\frac{s}{2\tau} \leq t \leq \frac{s}{\tau}$, where $\tau \geq 1$ is a
generic constant that will be specified later.  The constants in the
estimates below are independent of $\varepsilon$.

We have the conservation of mass property:
\begin{equation}
\label{mass}
\langle q(t,x,\cdot)\rangle_{\varphi}=1, \quad \text{ for all $t>0$}. 
\end{equation}
In addition, $q(t,x,y)=q(t,y,x)$ and its reproduction formula is
understood with respect to the measure
$
\varphi(y)dy$.

Our goal is to prove
\begin{enumerate}
\item[--] The weighted entropy estimate for $q$, 
\item[--] The estimate on the weighted Nash $G$-function for $q$.
\end{enumerate}
They will yield a lower bound on every ball of radius comparable to
$\sqrt t$, uniformly in the centre of that ball.  In Step~2 we will
apply this estimate only with $s$ comparable to the short kernel time.
We will then compose the physical kernel $p_\varepsilon$, whose
definition is independent of $s$.

The following elementary estimate is needed to keep all entropy
constants uniform in the starting point.

\begin{lemma}
\label{exp-convolution}
There is $C<\infty$ such that, for every $s,\theta>0$,
$\varepsilon\ge0$, and $x\in\mathbb R^{dN}$,
\begin{equation}
 \langle  e^{-\theta|x-\cdot|} \rangle_{\varphi}  \leq C\bigl(\theta^{-1}+\sqrt s\bigr)^{dN}.
 \label{exp-convolution2}
\end{equation}
\end{lemma}

\begin{proof}
We first deduce a uniform volume growth estimate from Lemma
\ref{mass2}. For every $z\in\mathbb R^{dN}$,
\[
 \Gamma_{s,z}(y)\geq c_{N,d} s^{-Nd/2},
 \qquad y\in B_{\sqrt s}(z).
\]
Therefore, Lemma \ref{mass2} gives
$$
 \langle  \mathbf{1}_{B_{\sqrt s}(z)}  \rangle_{\varphi} \leq Cs^{Nd/2}  \langle \Gamma_{s,z}(\cdot)  \rangle_{\varphi} \leq Cs^{Nd/2}
$$
uniformly in $s,\varepsilon$, and $z$.

Noting that the ball $B_R(x)$ can be covered by at most $C_{d,N}(1+R/\sqrt s)^{dN}$ balls of radius $\sqrt s$, we obtain
\begin{equation}
\langle  \mathbf{1}_{B_R(x)}  \rangle_{\varphi}
\leq C(R+\sqrt s)^{dN},
\qquad R>0.
\label{weighted-volume-growth}
\end{equation}
uniformly in $x,s$, and $\varepsilon$.

Since  $ e^{-\theta|x-y|} =\theta\int_{|x-y|}^{\infty}e^{-\theta R}\,dR $, using Tonelli's theorem,  we obtain
\begin{align*}
 \langle  e^{-\theta|x-\cdot|} \rangle_{\varphi}  &=\theta\int_0^\infty e^{-\theta R}
 \langle  \mathbf{1}_{B_R(x)}  \rangle_{\varphi}
dR\\
&\leq C\theta\int_0^\infty
e^{-\theta R}(R+\sqrt s)^{dN}\,dR\\
&\leq C(\theta^{-1}+\sqrt s)^{dN}.
\end{align*}
\end{proof}

Define 
the weighted Nash moment
$$
M(t,x):=\langle |x-\cdot| q(t,x,\cdot)\rangle_{\varphi} = \int_{\mathbb R^{dN}} |x-y| q(t,x,y) \varphi(y) dy,
$$
and the weighted Nash entropy
$$
Q(t,x):=-\langle q \log q \rangle_{\varphi} =-\int_{\mathbb R^{dN}} q(t,x,y) \log q(t,x,y) \varphi(y) dy.
$$
Put
$$
\tilde{Q}(t):=\frac{dN}{2}\log t.
$$

\begin{lemma}[Entropy estimate]
\label{prop-entropy} 
Let $\tau \geq 1$ be fixed.
There exist constants $C_\pm$, independent of $s$, $\varepsilon$, and
$x$, such that, for all
$s/(4\tau)\leq t\leq s/\tau$,
\begin{equation}
\label{entropy-esti}
-C_- \le Q(t,x) - \tilde{Q}(t) \le C_+.
\end{equation}
\end{lemma}

\begin{proof}
It is enough to keep track of the facts that $t\leq s$ and
$s/t\leq4\tau$; constants below may depend on the fixed $\tau$.
We start with the first inequality in \eqref{entropy-esti}. By the upper Gaussian bound \eqref{q-upper-A}, we obtain
\begin{equation}\label{q-estimate}
q(t,x,y)  \le C_1 t^{-\frac{Nd}{2}} e^{\frac{C_2t}{s}} \leq C t^{-\frac{dN}{2}}
\end{equation}
(where we have used $t \leq s$), 
so
\begin{align*}
- Q(t) & \le \int_{\mathbb R^{dN}} q(t,x,y) \log C  t^{-\frac{Nd}{2}}  \varphi(y) dy \\
& = \int_{\mathbb R^{dN}} q(t,x,y) (\log C) \varphi(y) dy  - \frac{dN}{2}  \log  t \int_{\mathbb R^{dN}} q(t,x,y)   \varphi(y) dy,
\end{align*}
i.e.\,in more compact notations,
$$
- Q(t) \le   \log C \langle q(t,x,\cdot)\rangle_{\varphi} - \frac{dN}{2}  \log  t \langle q(t,x,\cdot)\rangle_{\varphi}.
$$
Hence, taking into account the conservation of mass
property \eqref{mass} of $q$, we obtain
$$
Q(t)  - \frac{dN}{2}  \log  t \ge  - \log C := -C_-,
$$
which gives the first inequality in \eqref{entropy-esti}.  Notice that
the same argument gives this lower entropy estimate for every
$0<t\leq s$; this slightly wider range will be used below when we
integrate from $0$ to $t$.

Next, we prove the second inequality in \eqref{entropy-esti}.  Define the weighted Nash function
$$
\mathcal N(t):=\langle \frac{|\nabla q|^2}{q}\rangle_{\varphi} = \int_{\mathbb R^{dN}} \frac{|\nabla_y q(t,x,y)|^2}{q(t,x,y)}\varphi(y) dy.
$$
In the following calculation of the derivative of Nash's moment we use the symmetry of $q$, so $A_x q(t,x,y)=A_y q(t,x,y)$ (we use subscript $x$ or $y$ to emphasize in what variable operator $A$ acts on the heat kernel): 
\begin{align*}
M'(t) & =\langle |x-\cdot| \partial_t q(t,x,\cdot)\rangle_{\varphi}  = \int_{\mathbb R^{dN}} |x-y| \partial_t q(t,x,y) \varphi(y) dy \\
& = -  \int_{\mathbb R^{dN}} |x-y|A_x q(t,x,y) \varphi(y) dy =  -  \int_{\mathbb R^{dN}} |x-y|A_y q(t,x,y)\varphi(y) dy \\
& (\text{recall the expression \eqref{A_def} for $A$}) \\
  & = \int_{\mathbb R^{dN}} |x-y| \varphi(y) \Delta_y q(t,x,y) dy  -  \int_{\mathbb R^{dN}} |x-y| \varphi(y)  \bar{b}_\varepsilon \cdot \nabla_y q(t,x,y) dy   \\
	& (\text{integrate by parts in the first term}) \\
  & = -  \int_{\mathbb R^{dN}} \nabla_y |x-y|  \nabla_y q(t,x,y) \varphi(y) dy + 0,
    \end{align*}
		where, at the last step, we have used 
		$$0=-  \int_{\mathbb R^{dN}}  |x-y|  \nabla_y q(t,x,y) \frac{\nabla \varphi(y)}{ \varphi(y)}  \varphi(y) dy - \int_{\mathbb R^{dN}} |x-y| \varphi(y)  \bar{b}_\varepsilon \cdot \nabla_y q(t,x,y) dy$$ due to 
$\bar{b}_\varepsilon  = - \frac{\nabla \varphi}{ \varphi}$, i.e.\,\eqref{b_var}. 
Therefore, by Cauchy-Schwarz' inequality,
$$
M'(t) \le \left(\int_{\mathbb R^{dN}} (\nabla  |x-y|)^2 q(t,x,y) \varphi(y) dy\ \right)^{\frac{1}{2}} \left(  \int_{\mathbb R^{dN}} \frac{|\nabla_y q(t,x,y)|^2}{q(t,x,y)}\varphi(y) dy\right)^{\frac{1}{2}},
$$
and so, invoking again the conservation of probability, we arrive at
\begin{equation}
\label{M_N}
M'(t)  \le  \sqrt{\mathcal N(t)},
\end{equation}
where $\mathcal N$ is Nash's function introduced above.

On the other hand, we can evaluate
\begin{align*}
Q'(t) & =  -\int_{\mathbb R^{dN}} (\partial_t q(t,x,y)) \log q(t,x,y) \varphi(y) dy -\int_{\mathbb R^{dN}} q(t,x,y) \partial_t\log q(t,x,y) \varphi(y) dy \\
& = \int_{\mathbb R^{dN}}  A_xq(t,x,y) \log q(t,x,y) \varphi(y) dy -\int_{\mathbb R^{dN}}  \partial_t q(t,x,y) \varphi(y) dy \\
& = \int_{\mathbb R^{dN}}  A_yq(t,x,y)(1+ \log q(t,x,y) )\varphi(y) dy. 
\end{align*}
For fixed $\varepsilon>0$ the coefficients are smooth and bounded, so
the following integration by parts is justified first with standard
cutoffs and then by approximation.  Invoking \eqref{form},
\begin{align*}
Q'(t) = \int_{\mathbb R^{dN}}  \nabla_y q(t,x,y) \cdot \nabla_y (1+ \log q(t,x,y) )\varphi(y) dy  = \int_{\mathbb R^{dN}} \frac{|\nabla_y q(t,x,y)|^2}{q(t,x,y)}\varphi(y) dy,
\end{align*}
that is,
\begin{equation}
\label{Q_N}
Q'(t) = \mathcal{N}(t).
\end{equation}
Combining \eqref{M_N} and \eqref{Q_N} and integrating in time, we obtain the following inequality between the moment and the entropy:
$$
M(t)  \le \int_0^t \sqrt{Q'(\tau)} d\tau.
$$
Here we have used the fact that $M(0)=0$, which is an obvious consequence of the fact that $q(0,x,y)$ is the delta-function concentrated in $x-y$.

We estimate the right-hand side of the previous inequality:
$$
 \int_0^t \sqrt{ Q'(\tau)} d\tau \le \left(  \int_0^t \frac{1}{\sqrt{\tau}}d\tau \right)^{\frac{1}{2}}\left( \int_0^t \sqrt{\tau}Q'(\tau)  d\tau \right)^{\frac{1}{2}} =  \left(  \int_0^t \frac{1}{\sqrt{\tau}} d\tau \right)^{\frac{1}{2}}\left( \int_0^t \sqrt{\tau} dQ(\tau) \right)^{\frac{1}{2}}.
$$
Since $  \int_0^t \tau^{-\frac{1}{2}}d\tau = 2\sqrt{t}$, we thus have 
\begin{equation}\label{Mt}
M(t)  \le  \left( 2\sqrt{t} \int_0^t \sqrt{\tau} dQ(\tau) \right)^{\frac{1}{2}},
\end{equation}
where, integrating by parts, we evaluate
\begin{align*}
 \int_0^t \sqrt{\tau} dQ(\tau) & =   \sqrt{t} Q(t) - \frac{1}{2} \int_0^t \frac{1}{\sqrt{\tau}} Q(\tau) d\tau \\ 
& (\text{we apply the first inequality in \eqref{entropy-esti}}) \\
& \le  \sqrt{t} Q(t) + \frac{1}{2} \int_0^t \frac{1}{\sqrt{\tau}} \left( - \frac{dN}{2}  \log  \tau + C_- \right) d\tau \\
& \leq   \sqrt{t} \left(Q(t) + C_-  -\frac{dN}{2}\log t + dN \right).
\end{align*}
Applying this estimate in \eqref{Mt} yields
\begin{equation}
\label{m_est}
M(t)  \le \sqrt{2t \left(Q(t) + C_-  - \frac{dN}{2}\log t + dN \right)}. 
\end{equation}

Next, inequality $ q\log q \ge -\xi q - e^{-1-\xi}$ ($\xi \in \mathbb R$) with $\xi = \beta + \theta |x-\cdot| $ for $\theta >0$, and the preservation of mass property \eqref{mass}, yield
\begin{equation}
\label{Q_expr}
- Q(t)= \int_{\mathbb R^{dN}} q(t,x,y) \log q(t,x,y) \varphi(y) dy \ge - \beta - \theta M(t) - e^{-1-\beta}\int_{\mathbb R^{dN}}  e^{-\theta |x-y| } \varphi(y) dy.
\end{equation}
Lemma \ref{exp-convolution} estimates the last integral, uniformly
in $x$, by $C(\theta^{-1}+\sqrt s)^{dN}$.
Applying this in \eqref{Q_expr}, we arrive at the following intermediate upper bound on the entropy:
$$
Q(t) \le   \beta + \theta M(t) + e^{-1-\beta}  C (\theta^{-1} + s^{\frac{1}{2}})^{dN}.
$$
Let us fix $t$.
We are free to choose parameters $\theta$ and $\beta$. We define $\theta$ by the identity $ \theta M(t) = 1$ and put $\beta:= -1 + \log C + Nd \log  (\theta^{-1} + s^{\frac{1}{2}})$. We obtain $Q(t) \leq \beta + 2$. Therefore, invoking again the definition of $\beta$,
\begin{align*}
Q(t) & \le 1+\log C + dN \log  (s^{\frac{1}{2}} + \theta^{-1})  \\
& = \log (eC) + dN \log (s^{\frac{1}{2}} + \theta^{-1}) \\
& =  \log (eC)  + dN \log  (s^{\frac{1}{2}} + M(t)).    
\end{align*}
We can now estimate:
\begin{align*}
Q(t)- dN \log \sqrt{t} & \le  \log (eC) + dN \log  (s^{\frac{1}{2}} + M)  - dN \log \sqrt{t} \\
& (\text{use }s/t\leq4\tau) \\
& \leq  \log (eC) + dN\log \left(
       \frac{M(t)}{\sqrt{t} } +2\sqrt\tau \right).
\end{align*}
It follows that
$$
e^{(Q(t)- Nd \log \sqrt{t})/Nd}
 \le C_\tau t^{-1/2}\left(M(t)+\sqrt t\right).
$$
Hence,  using inequality \eqref{m_est}, we arrive at
$$
e^{(Q(t)- Nd \log \sqrt{t})/dN}
 \le C_\tau\left(
 \sqrt{Q(t)-dN\log\sqrt t+C_-+dN}+1\right).
$$
This inequality implies that there is a constant $C_+$, depending
only on $d,N,\nu$ and the fixed $\tau$, such that
$$
Q(t)-  \frac{dN}{2} \log t \le  C_+.
$$
The proof of Lemma \ref{prop-entropy} is completed.
\end{proof}

\begin{remark}
The entropy estimate alone already gives the 
on-diagonal lower bound
$$
 q_{s,\varepsilon}(r,x,x)\ge c r^{-dN/2},  \qquad \frac{s}{2\tau}\le r\le\frac{2s}{\tau}.
$$
Indeed, put $u=r/2$.  Then
$u\in[s/(4\tau),s/\tau]$, and Jensen's inequality, symmetry, and the
reproduction property give
\[
 \exp\!\left(\langle q(u,x,\cdot)\log q(u,x,\cdot)
                    \rangle_\varphi\right)
 \le\langle q(u,x,\cdot)^2\rangle_\varphi
 =q(2u,x,x)=q(r,x,x).
\]
The lower entropy estimate bounds the left-hand side below by
$c u^{-Nd/2}$, which is comparable to $r^{-Nd/2}$.
\end{remark}

\medskip

We now estimate the weighted Nash $G$-function.  Fix a centre
$\xi\in\mathbb R^{dN}$.  In this subsection only, abbreviate
$\Gamma_{s,\xi}$ to $\Gamma_s$ and put
\[
 \chi_s(y):=e^{-|y-\xi|^2/(4s)},
 \qquad
 m_s(\xi):=\langle\Gamma_s\rangle_\varphi.
\]
Thus the centre is moved only in the auxiliary Gaussian; neither the
weight $\varphi_{s,\varepsilon}$ nor the kernel $q_{s,\varepsilon}$ is
translated.  Define
\begin{equation}
\label{G_def}
G_\xi(t,z) := \left\langle \Gamma_s(\cdot)
 \log q(t,z,\cdot)\right\rangle_{\varphi}.
\end{equation}

\begin{lemma}[$G$-function bound]
\label{N-G-bound} 
There exist constants $K>0$ and $\tau\ge2$, independent of
$s,\varepsilon$, and $\xi$, such that
\begin{equation}
\label{G_ineq}
 m_s(\xi)^{-1}G_\xi(t,z)\ge-\tilde Q(t)-K,
\end{equation}
whenever $s/(2\tau)\leq t\leq s/\tau$ and
$z\in B_{\sqrt t}(\xi)$.
\end{lemma}

\begin{proof}
We use Lemma \ref{prop-entropy} and Theorem
\ref{thm_spectral_gap}.  Fix a target time
$t\in[s/(2\tau),s/\tau]$.  All differential inequalities below are
used only for $u\in[t/2,t]$; this interval is contained in
$[s/(4\tau),s/\tau]$, where the entropy estimate is available.
Let us first carry out the proof of \eqref{G_ineq} for
\begin{equation}
\label{G_def_eps}
G_{\varepsilon_1}(t,z) := \left\langle \Gamma_s(\cdot)
 \log \left[\varepsilon_1+q(t,z,\cdot)\right]\right\rangle_{\varphi},
 \qquad \varepsilon_1>0,
\end{equation}
where $\varepsilon_1$ is needed to justify the integration by parts.
We pass to the limit $\varepsilon_1\downarrow0$ at the end of the
proof.  To shorten notation, write for now $G=G_{\varepsilon_1}$.
We have
\begin{align}
\partial_t G& =  \big\langle \Gamma_s(\cdot)  \frac{q^{\prime}(t,z,\cdot)}{ \varepsilon_1+q(t,z,\cdot)}\big\rangle_{\varphi} =  - \big\langle \frac{\Gamma_s(\cdot)}{ \varepsilon_1+q(t,z,\cdot)}  A_zq^{}(t,z,\cdot)\big\rangle_{\varphi} \notag \\
&  = - \big\langle \frac{\Gamma_s(\cdot)}{ \varepsilon_1+q(t,z,\cdot)}  A_{\cdot}\, q^{}(t,z,\cdot)\big\rangle_{\varphi}  \notag \\
& = - \big\langle\nabla_{\cdot}\, \frac{\Gamma_s(\cdot)}{ \varepsilon_1+q(t,z,\cdot)},  \nabla_{\cdot}\, q^{}(t,z,\cdot)\big\rangle_{\varphi} \notag \\[2mm]
& =  - \big\langle \nabla_{\cdot}\,  \Gamma_s(\cdot),  \nabla_{\cdot}\,   \log (\varepsilon_1+q(t,z,\cdot)) \big\rangle_{\varphi} +  \big\langle     \Gamma_s(\cdot) \vert \nabla_{\cdot}\,   \log (\varepsilon_1+q(t,z,\cdot))\vert^2 \big\rangle_{\varphi} \notag \\[2mm] 
& =:I_1 + I_2. \label{g_t_est}
\end{align}
Let us estimate $I_1$:  
\begin{align*}
-I_1
&= \int_{\mathbb{R}^{dN}}
   \nabla_x\Gamma_s(x)
   \cdot\nabla_x\log\bigl(\varepsilon_1+q(t,z,x)\bigr)
   \varphi(x)\,dx\\
&= \int_{\mathbb{R}^{dN}}
   \frac{\nabla_x\Gamma_s(x)}{\sqrt{\Gamma_s(x)}}
   \cdot
   \left(
     \sqrt{\Gamma_s(x)}\,
     \nabla_x\log\bigl(\varepsilon_1+q(t,z,x)\bigr)
   \right)
   \varphi(x)\,dx \\
&\le \frac12
   \int_{\mathbb{R}^{dN}}
   \frac{|\nabla_x\Gamma_s(x)|^2}{\Gamma_s(x)}
   \varphi(x)\,dx + \frac12    \int_{\mathbb{R}^{dN}}  \Gamma_s(x)
   \left|
     \nabla_x\log\bigl(\varepsilon_1+q(t,z,x)\bigr)
   \right|^2
   \varphi(x)\,dx  \\
&= \frac12
   \left\langle
     \Gamma_s(\cdot)
     \left|\nabla_{\cdot}\log\Gamma_s(\cdot)\right|^2
   \right\rangle_{\varphi}
   +\frac12 I_2.
\end{align*}
So,
\begin{equation}\label{G-derive}
I_1 + I_2 \ge -\frac{1}{2} \left\langle  \Gamma_s(\cdot) \vert \nabla_\cdot  \log \Gamma_s(\cdot) \vert^2 \right\rangle_{\varphi} + \frac{1}{2} \left\langle     \Gamma_s(\cdot) \vert \nabla_{\cdot}  \log (\varepsilon_1+q(t,z,\cdot))\vert^2 \right\rangle_{\varphi}.
\end{equation}
By Lemma \ref{mass2}, uniformly in the centre $\xi$,
\begin{equation}  \label{centered-Gaussian-energy}
 \left\langle\Gamma_s
 |\nabla\log\Gamma_s|^2\right\rangle_\varphi =\frac1{4s^2}\left\langle\Gamma_s
 |\,\cdot-\xi|^2\right\rangle_\varphi  \leq\frac Cs.
\end{equation}
Combining \eqref{g_t_est}, \eqref{G-derive}, and
\eqref{centered-Gaussian-energy}, we obtain
$$
\partial_t G(t,z)+\frac Cs
 \geq\frac12\left\langle\Gamma_s
 |\nabla\log(\varepsilon_1+q(t,z,\cdot))|^2
 \right\rangle_\varphi.
$$
Write $m=m_s(\xi)$.  By \eqref{uniform-weight-lower},
$m\ge c_\varphi$. 

Let $\tau_0$ be large enough for the Gaussian
comparison \eqref{estim-varphi} below, and choose
\begin{equation}
 \label{tau_choice}
 \tau\ge 2\vee\frac{2C}{dN\, c_\varphi}\vee\tau_0.
\end{equation}
This choice depends only on $d,N,$ and $\nu$, because the lower bound
for $m$ and all constants in the upper estimate
\eqref{q-upper-A} are uniform in $s,\varepsilon,$ and $\xi$.
Since $u\le s/\tau$, the preceding differential inequality and
$\tilde Q'(u)=\frac{Nd}{2u}$ give, for $u\in[t/2,t]$,
\begin{equation}\label{G-deriv}
 \frac d{du}\{G(u,z)+m\tilde Q(u)\}
 \ge \frac{M}{2s}
 \left\langle\Gamma_s
 \left|\log(\varepsilon_1+q(u,z,\cdot))-m^{-1}G(u,z)
 \right|^2\right\rangle_\varphi .
\end{equation}
Here we used the translated spectral gap
\eqref{translated-gap}.

By the Gaussian upper bound \eqref{q-upper-A}, for
$u\in[t/2,t]$ and
$z\in B_{\sqrt t}(\xi)$, that estimate and the elementary inequality
$|y-\xi|^2\le2|y-z|^2+2|z-\xi|^2$ imply, after increasing $\tau$ if
necessary,
\begin{equation}\label{estim-varphi}
 c_0\left(\frac us\right)^{dN/2}q(u,z,y)
 \le \Gamma_{s,\xi}(y),
 \qquad y\in\mathbb R^{dN} .
\end{equation}
Notice that this comparison is centred at $\xi$, not at the origin.
Combining \eqref{G-deriv}, \eqref{estim-varphi}, Cauchy-Schwarz, and
$\langle q(u,z,\cdot)\rangle_\varphi=1$, we obtain
\begin{equation}
 \frac d{du}\{-m\Phi(u)\}
 \ge \frac c{s}\left(\frac us\right)^{dN/2}
       (\Phi(u)-C_+)_{+}^{2},
 \qquad
 \Phi(u):=-m^{-1}G(u,z)-\tilde Q(u).
 \label{Phi-ode}
\end{equation}
Indeed, Lemma \ref{prop-entropy} and
$\log(\varepsilon_1+q)\ge\log q$ show that
\[
 \left\langle q
 \left|\log(\varepsilon_1+q)-m^{-1}G\right|\right\rangle_\varphi
 \ge (\Phi-C_+)_{+}.
\]

It remains to integrate \eqref{Phi-ode}, and only on the interval
$[t/2,t]$.  

If
$\Phi(u_0)<2(C_+\vee1)$ for some $u_0\in[t/2,t]$, then the monotonicity of
$-m\Phi$ gives $\Phi(t)<2(C_+\vee1)$.  Otherwise
$\Phi(u)\ge2 (C_+\vee1)$ throughout that interval, so
$(\Phi-C_+)_{+}\ge\Phi/2$ and
\[
 \frac d{du}\frac1{\Phi(u)}
 \ge\frac c{ms}\left(\frac us\right)^{dN/2}.
\]
Integration from $t/2$ to $t$, followed by Lemma \ref{mass2}, yields
\[
 \Phi(t)\le C m\left(\frac st\right)^{dN/2+1}\le C,
\]
because $s/t\in[\tau,2\tau]$.  This proves \eqref{G_ineq} for
$G_{\varepsilon_1}$.  For fixed $t,z$, the upper bound
\eqref{q-upper-A} supplies a finite constant $M_{t,s}$ with
$q(t,z,\cdot)\le M_{t,s}$.  Apply monotone convergence to
$\log(M_{t,s}+1)-\log(\varepsilon_1+q)$ and then let
$\varepsilon_1\downarrow0$.  This gives the stated bound for $G_\xi$.
    \end{proof}

\subsubsection*{Proof of the lower bound in Theorem \ref{thm1} completed}

\noindent\emph{Step 1.}
Let $r>0$, set $s=\tau r$, and write
$q=q_{\tau r,\varepsilon}$ and
$\varphi=\varphi_{\tau r,\varepsilon}$.  If
$|u-v|\le\sqrt r$, put $\xi=(u+v)/2$.  Then
$u,v\in B_{\sqrt{r/2}}(\xi)$, and the reproduction formula at the
half-time gives
\begin{align}
 q(r,u,v)
 &=\langle q(r/2,u,\cdot)q(r/2,\cdot,v)\rangle_\varphi\notag\\
 &\ge(4\pi s)^{dN/2}
   \langle\Gamma_{s,\xi}\rangle_\varphi
   \exp\!\left\{
    \frac{G_\xi(r/2,u)+G_\xi(r/2,v)}
         {\langle\Gamma_{s,\xi}\rangle_\varphi}
       \right\}.                                      \label{centered-Jensen}
\end{align}
Here we inserted
$e^{-|\cdot-\xi|^2/(4s)}=(4\pi s)^{Nd/2}\Gamma_{s,\xi}\le1$,
then applied Jensen's inequality and symmetry of $q$.  Since
$r/2=s/(2\tau)$, Lemmas \ref{N-G-bound} and \ref{mass2} imply
\begin{equation}
 q_{\tau r,\varepsilon}(r,u,v)\ge c r^{-dN/2},
 \qquad |u-v|\le\sqrt r.                         \label{q_bd}
\end{equation}
We now use Lemma \ref{gaussian-transfer}(b) with
$s=\tau r$ and $M_0=1$.
Here
\[
 r\|h_\varepsilon\|_\infty^2\leq C/\tau,
\]
so the perturbation changes only the constants in the local Gaussian
estimate.  It follows from \eqref{q_bd} that
\begin{equation}
 p_\varepsilon(r,u,v)
 \ge c r^{-dN/2}\varphi_{\tau r,\varepsilon}(v),
 \qquad |u-v|\le\sqrt r.                         \label{p-local}
\end{equation}
Although its proof uses the auxiliary scale $\tau r$, this is an
estimate for the physical kernel $p_\varepsilon$, which is independent
of that choice.

\medskip

\noindent\emph{Step 2.}
We first derive an unweighted Gaussian bound.  Fix $t>0$ and
$x,y\in\mathbb R^{dN}$, and set
$$
 L:=\left\lceil\max\left\{1, 16\frac{|x-y|^2}{t} \right\}\right\rceil,
 \qquad r:=\frac tL,
 \qquad x_m:=x+\frac mL(y-x).
$$
For $1\le m<L$, let
$A_m=B_{\sqrt r/8}(x_m)$, and put $z_0=x,z_L=y$.  If
$z_m\in A_m$ and $z_{m+1}\in A_{m+1}$, with the evident endpoint
convention, then $|z_{m+1}-z_m|\le\sqrt r$; consequently
\eqref{p-local}, with the same matched auxiliary scale $\tau r$
for every factor, applies along the entire chain.  Since
$\varphi_{\tau r,\varepsilon}\ge c_\varphi$, the Lebesgue semigroup
property of $p_\varepsilon$ gives
\begin{align}
 p_\varepsilon(t,x,y)
 &\ge\int_{A_1\times\cdots\times A_{L-1}}
       \prod_{m=0}^{L-1}p_\varepsilon(r,z_m,z_{m+1})
       \,dz_1\cdots dz_{L-1}\notag\\
 &\ge c\,t^{-dN/2}\exp(-CL)
 \ge c\,t^{-dN/2}
       \exp\!\left(-C\frac{|x-y|^2}{t}\right).       \label{p-global-unweighted}
\end{align}
It remains to recover the endpoint weight at scale $t$.  Put
$r_0=t/2$ and use the semigroup formula once more, restricting its
integral to $B_{\sqrt{r_0}/4}(y)$.  On this ball,
\eqref{p-global-unweighted} controls the first factor and
\eqref{p-local}, used with auxiliary scale $\tau r_0$, controls the
second.  Moreover,
$|x-z|^2/r_0\le4|x-y|^2/t+1/8$ there.  Since $\tau\ge2$ and the
cutoff function $\eta$ is nonincreasing,
$\varphi_{\tau r_0,\varepsilon}(y)
\ge\varphi_{t,\varepsilon}(y)$.  Therefore
\begin{align*}
 p_\varepsilon(t,x,y)
 &\ge\int_{B_{\sqrt{r_0}/4}(y)}
       p_\varepsilon(r_0,x,z)p_\varepsilon(r_0,z,y)\,dz\\
 &\ge c r_0^{-dN}
       e^{-C|x-y|^2/t}\varphi_{t,\varepsilon}(y)
       |B_{\sqrt{r_0}/4}(y)|.
\end{align*}
Consequently,
\begin{equation}
 p_\varepsilon(t,x,y)
 \ge c\,t^{-dN/2}
       \exp\!\left(-C\frac{|x-y|^2}{t}\right)
       \varphi_{t,\varepsilon}(y),                  \label{p-final-lower}
\end{equation}
with constants independent of $\varepsilon$ (and, in fact, valid for
all $t>0$).  

\bigskip

\section{Direct proof of the upper heat kernel bound}
\label{upper_bound_app}

In this appendix we improve slightly the proof of the upper bound in \cite{BK}, namely, for all $\nu<2(d-2)$ and all dimensions $d \geq 3$, we regularize the denominator of the interaction kernel as
$|z|_\varepsilon=(|z|^2+\varepsilon)^{1/2}$; in \cite{BK} we used a different regularization when $d=3$. 

We will only deal with $\varepsilon>0$. The passage to the limit $\varepsilon \downarrow 0$ is justified in a standard way using the semigroup convergence \eqref{lim_exists} and the Lebesgue differentiation theorem.

As before, the constants denoted by $C$ below may change from
line to line and depend on $d,N,\nu$, but not on $s$ or
$\varepsilon$.

Fix $s>0$ and $\varepsilon>0$.  We use the notations introduced in the beginning of the proof of Theorem \ref{thm1}:
\[
 \varphi:=\varphi_{s,\varepsilon},
 \qquad
 A=-\varphi^{-1}\operatorname{div}(\varphi\nabla),
\]
and $q=q_{s,\varepsilon}$ denotes the kernel of $e^{-tA}$ with
respect to $\varphi(x)\,dx$. 

For the calculation below we retain the notation
\begin{equation}
 \psi_{s,\varepsilon}(x)
 :=\prod_{1\leq i<j\leq N}
 \left(s^{-1/2}|x^i-x^j|_\varepsilon\right)^{-\nu/N}.
 \label{weights}
\end{equation}

\subsubsection{Regularized many-particle Hardy inequality}

Set
\begin{equation}
 K_\varepsilon(z):=\frac{z}{|z|_\varepsilon^2},
 \qquad
 V_\varepsilon(z):=\operatorname{div}K_\varepsilon(z)
 =\frac{(d-2)|z|^2+d\varepsilon}{(|z|^2+\varepsilon)^2}.
 \label{regularized-potential}
\end{equation}

\begin{lemma}
\label{regularized-hardy-lemma}
For every $f\in W^{1,2}(\mathbb R^{dN})$,
\begin{equation}
 \frac{d-2}{N}\sum_{i<j}\int_{\mathbb R^{dN}}
 V_\varepsilon(x^i-x^j)|f(x)|^2dx
 \leq\int_{\mathbb R^{dN}}|\nabla f(x)|^2dx.
 \label{regularized-hardy}
\end{equation}
\end{lemma}

\begin{proof}
Put
\[
 F_i(x):=\sum_{j\ne i}K_\varepsilon(x^i-x^j),
 \qquad F:=(F_1,\ldots,F_N).
\]
Let us first record several elementary identities and inequalities for $F$.
The fact that $K_\varepsilon$ is odd gives $\sum_iF_i=0$ and
$
 \sum_{i=1}^N|F_i|^2
 =\sum_{i<j}(F_i-F_j)\cdot K_\varepsilon(x^i-x^j).
$
Moreover,
\[
 \sum_{i<j}|F_i-F_j|^2
 =N\sum_{i=1}^N|F_i|^2.
\]
Cauchy-Schwarz applied to the preceding two identities yields
\begin{equation}
 \sum_{i=1}^N|F_i|^2
 \leq N\sum_{i<j}|K_\varepsilon(x^i-x^j)|^2
 \leq\frac{N}{d-2}\sum_{i<j}V_\varepsilon(x^i-x^j),
 \label{algebraic-bound}
\end{equation}
where the last inequality follows directly from the definition
\eqref{regularized-potential}.  We also have
\begin{equation}
 \operatorname{div}_xF
 =2\sum_{i<j}V_\varepsilon(x^i-x^j).
 \label{div-F}
\end{equation}
For a real-valued $f\in C_c^\infty(\mathbb R^{dN})$, we have
\begin{align*}
 0&\leq\int_{\mathbb R^{dN}} |\nabla f+\frac{d-2}{N}Ff|^2dx\\
  &=\int_{\mathbb R^{dN}} |\nabla f|^2dx+\left(\frac{d-2}{N}\right)^2\int_{\mathbb R^{dN}} |F|^2f^2dx
    -2\frac{d-2}{N}\sum_{i<j}\int_{\mathbb R^{dN}} V_\varepsilon(x^i-x^j)f^2dx.
\end{align*}
By \eqref{algebraic-bound}, the second term in the last line is
at most $\frac{d-2}{N}\sum_{i<j}\int V_\varepsilon f^2$.  This proves
\eqref{regularized-hardy} for smooth real-valued $f$. A straightforward density argument extends the inequality to all $f \in W^{1,2}$. 

Finally, for a complex-valued $f$, applying the established inequality to $|f|$ and using $|\nabla|f|| \leq |\nabla f|$ in the right-hand side, we obtain the result.  
\end{proof}

Taking $\varepsilon \downarrow 0$ and using Fatou's lemma, we recover
\eqref{hardy-many}, but of course the focus of the above lemma is on the uniformity in $\varepsilon>0$.

We will also need the following modification of Lemma \ref{regularized-hardy-lemma}:

\begin{lemma}
\label{divergence-lemma}
For every $f\in W^{1,2}(\mathbb R^{dN})$,
\begin{equation}
 \int_{\mathbb R^{dN}}
 \left[-\operatorname{div}\left(\frac{\nabla\varphi}{\varphi}\right)\right]
 |f|^2dx
 \leq\frac{2\nu}{d-2}\int_{\mathbb R^{dN}}|\nabla f|^2dx
      +\frac{C}{s}\int_{\mathbb R^{dN}}|f|^2dx.
 \label{divergence-bound}
\end{equation}
\end{lemma}

\begin{proof}
This would be the assertion of Lemma \ref{regularized-hardy-lemma} if we had weight $\psi_{s,\varepsilon}$ instead of $\varphi=\varphi_{s,\varepsilon}$ in the left-hand side (in which case we could take $C=0$). We are going to reduce the proof to the use of Lemma \ref{regularized-hardy-lemma}, but for this we need to compare these two weights. Recall that the weight $\varphi$ is defined in terms of function $\eta$, that is, 
$\log\eta(r)=(\nu/N)a(r)$, see \eqref{eta_def}. If we replace $a(r)$ by $-\log r$, then we obtain instead the definition of weight $\psi$. Set
\[
 g(r):=a(r)+\log r,
 \qquad
 R_s(x):=\frac{\nu}{N}\sum_{i<j}
 g\left(s^{-1/2}|x^i-x^j|_\varepsilon\right).
\]

By \eqref{weights} and the definition of
$\varphi_{s,\varepsilon}$,
\begin{equation}
 \log\varphi=\log\psi_{s,\varepsilon}+R_s.
 \label{log-comparison}
\end{equation}
The function $g$ vanishes on $]0,1[$, has bounded first and second
derivatives on $[1,2]$, and equals
$1-2\log2+\log r$ for $r>2$. 

Let $z \in \mathbb R^d$. We note that on the ``transition region''
$1\leq s^{-1/2}|z|_\varepsilon \leq2$ we have
\begin{equation}
 |\nabla_z (s^{-1/2}|z|_\varepsilon)|\leq s^{-1/2},
 \qquad
 |\Delta_z (s^{-1/2}|z|_\varepsilon)|
 =\frac{(d-1)|z|^2+d\varepsilon}
        {\sqrt{s}(|z|^2+\varepsilon)^{3/2}}
 \leq\frac d{s}.
 \label{rho-derivatives}
\end{equation}

Combining the previous two observations, we obtain
\begin{equation}
 -\Delta R_s\leq\frac{b_0}{s},
 \qquad 0\leq b_0\leq C.
 \label{R-bound}
\end{equation}
Indeed, the left-hand side vanishes in the inner region, is bounded by
$C/s$ in the transition region, and in the outer region it is bounded by $-2(\nu/N)V_\varepsilon(z)\leq0$. 

On the other hand,
\begin{equation}
 -\Delta\log\psi_{s,\varepsilon}
 =\frac{2\nu}{N}\sum_{i<j}V_\varepsilon(x^i-x^j).
 \label{psi-divergence}
\end{equation}
Applying \eqref{log-comparison}, \eqref{R-bound}, and \eqref{psi-divergence} and integrating against $|f|^2$, we obtain using
Lemma \ref{regularized-hardy-lemma}:
\begin{equation*}
 \int_{\mathbb R^{dN}} (-\Delta\log\varphi)|f|^2
 \leq\frac{2\nu}{d-2}\int_{\mathbb R^{dN}} |\nabla f|^2
      +\frac Cs\int_{\mathbb R^{dN}}  |f|^2,
\end{equation*}
which is exactly \eqref{divergence-bound}.
\end{proof}

\subsubsection{Weighted Sobolev inequality}

Define the Sobolev exponent
\[
 \ell:=\frac{dN}{dN-2},
 \qquad
 \|u\|_{p,\varphi}:=\langle|u|^p\rangle_\varphi^{1/p}.
\]

\begin{lemma}[Weighted Sobolev inequality]
\label{weighted-sobolev-lemma}
If $\nu<2(d-2)$, then
\begin{equation}
 \|u\|_{2\ell,\varphi}^2
 \leq C_1\langle|\nabla u|^2\rangle_\varphi
      +\frac{C_2}{s}\|u\|_{2,\varphi}^2,
 \qquad u\in W^{1,2}(\mathbb R^{dN}),
 \label{weighted-sobolev}
\end{equation}
where $C_1,C_2$ are independent of $s$ and $\varepsilon$.
\end{lemma}

\begin{proof}
It is enough to consider a real $u\in C_c^\infty$. The general case
follows by a density argument and the inequality $|\nabla|u||\leq|\nabla u|$.

Set
\[
 v:=u\varphi^{1/(2\ell)}.
\]
The ordinary Sobolev inequality on $\mathbb R^{dN}$ gives
\begin{equation}
 \|u\|_{2\ell,\varphi}^2
 =\|v\|_{2\ell}^2
 \leq C_S\int_{\mathbb R^{dN}} |\nabla v|^2dx.
 \label{ordinary-sobolev}
\end{equation}
Since $
 \varphi^{1/(2\ell)}\nabla u
 =\nabla v-\frac{v}{2\ell}\nabla\log\varphi$,
we have, using the integration by parts,
\begin{align*}
 \int_{\mathbb R^{dN}}  |\nabla v|^2dx
 &=\int_{\mathbb R^{dN}} \varphi^{1/(2\ell)}\nabla u\cdot\nabla v\,dx
   +\frac1{2\ell}
       \int_{\mathbb R^{dN}} v\nabla v\cdot\nabla\log\varphi\,dx\\
 &=\int_{\mathbb R^{dN}}  \varphi^{1/(2\ell)}\nabla u\cdot\nabla v\,dx
   +\frac1{4\ell}
       \int_{\mathbb R^{dN}} (-\Delta\log\varphi)v^2dx.
\end{align*}
Next, for every $\delta>0$, Young's inequality gives
\[
 \int_{\mathbb R^{dN}}  \varphi^{1/(2\ell)}\nabla u\cdot\nabla v\,dx
 \leq\frac{\delta}{1+\delta}\int_{\mathbb R^{dN}} |\nabla v|^2dx
      +\frac{1+\delta}{4\delta}
       \int_{\mathbb R^{dN}} \varphi^{1/\ell}|\nabla u|^2dx.
\]
Therefore,
\begin{equation}
 \int_{\mathbb R^{dN}}  |\nabla v|^2dx
 \leq\frac{(1+\delta)^2}{4\delta}
       \int_{\mathbb R^{dN}}  \varphi^{1/\ell}|\nabla u|^2dx
      +\frac{1+\delta}{4\ell}
       \int_{\mathbb R^{dN}} (-\Delta\log\varphi)v^2dx.
 \label{v-key}
\end{equation}
The uniform lower bound \eqref{uniform-weight-lower} gives
$\varphi=\varphi_{s,\varepsilon}\geq c_\varphi>0$, so
$\varphi^{1/\ell}\leq C\varphi$. Thus, we can estimate the first term in the right-hand side of \eqref{v-key} as follows:
\[
 \int_{\mathbb R^{dN}}  \varphi^{1/\ell}|\nabla u|^2dx
 \leq C\langle|\nabla u|^2\rangle_\varphi,
 \quad
 \|v\|_2^2\leq C\|u\|_{2,\varphi}^2.
\]
Now, we apply the many-particle Hardy type inequality of Lemma \ref{divergence-lemma} to the second term in the right-hand side of \eqref{v-key}.  We obtain
\[
 \left(1-\frac{(1+\delta)\nu}{2(d-2)\ell}\right)
 \int_{\mathbb R^{dN}} |\nabla v|^2dx
 \leq C_\delta\langle|\nabla u|^2\rangle_\varphi
      +\frac Cs\|u\|_{2,\varphi}^2.
\]
Since $\nu<2(d-2)$ and $\ell>1$, the constant $\delta>0$ can be chosen so that
$(1+\delta)\nu<2(d-2)\ell$. Therefore, the coefficient on the left is
positive.  Together with
\eqref{ordinary-sobolev}, this proves
\eqref{weighted-sobolev}.
\end{proof}

\subsubsection{Proof of the upper bound, completed}

The auxiliary scale $s$ remains fixed, so the weight $\varphi$ is time-independent. At the very end we will take $s=t$.

We run Moser's iteration procedure and use the ``Davies device'', see Fabes-Stroock \cite{FS}.

\begin{lemma}
\label{moser-lemma}
For every bounded real-valued Lipschitz function $\zeta$ on
$\mathbb R^{dN}$ and every $t>0$,
\begin{align}
 &\bigl\|e^\zeta e^{-tA}e^{-\zeta}
       \bigr\|_{L^2_\varphi\rightarrow L^\infty}
 +\bigl\|e^{-\zeta}e^{-tA}e^\zeta
       \bigr\|_{L^2_\varphi\rightarrow L^\infty}
 \notag\\
 &\hspace{25mm}\leq
 Ct^{-dN/4}
 \exp\left\{
       \frac{Ct}{s}
       +C\|\nabla\zeta\|_\infty^2t
      \right\}.
 \label{two-infinity}
\end{align}
The constants are independent of $s$, $\varepsilon$, and $\zeta$.
\end{lemma}

\begin{proof}
The positivity preservation of $e^{-tA}$ gives
$
 \left|e^\zeta e^{-tA}e^{-\zeta}f\right|
 \leq e^\zeta e^{-tA}e^{-\zeta}|f|,
$
so it is enough to consider $f\geq0$.  Set
\[
 u(r):=e^\zeta e^{-rA}e^{-\zeta}f.
\]
For $p\geq2$, multiplying the equation for $u$ by $u^{p-1}$ and integrating,
gives
\begin{align*}
 -\frac1p\frac{d}{dr}\|u(r)\|_{p,\varphi}^p
 &=\left\langle
     \nabla(e^{-\zeta}u),
     \nabla(e^\zeta u^{p-1})
   \right\rangle_\varphi                                      \\
 &=\frac{4(p-1)}{p^2}
      \left\langle\left|\nabla u^{p/2}\right|^2
      \right\rangle_\varphi
   -\frac{2(p-2)}p
      \left\langle
        u^{p/2}\nabla u^{p/2}\cdot\nabla\zeta
      \right\rangle_\varphi                                   \\
 &\quad
   -\left\langle|\nabla\zeta|^2u^p\right\rangle_\varphi .
\end{align*}
Young's inequality therefore yields
\[
 -\frac{d}{dr}\|u(r)\|_{p,\varphi}^p
 \geq
 c\left\langle
       \left|\nabla u(r)^{p/2}\right|^2
   \right\rangle_\varphi
 -Cp^2\|\nabla\zeta\|_\infty^2
  \|u(r)\|_{p,\varphi}^p.
\]
Applying the weighted Sobolev inequality
\eqref{weighted-sobolev} gives
\begin{equation}
 -\frac{d}{dr}\|u(r)\|_{p,\varphi}^p
 \geq
 c\|u(r)^{p/2}\|_{2\ell,\varphi}^2
 -Cp^2\left(
        \|\nabla\zeta\|_\infty^2+\frac1s
       \right)
  \|u(r)\|_{p,\varphi}^p.
 \label{p-sobolev}
\end{equation}
Moreover,
\[
 \|u^{p/2}\|_{2,\varphi}
 \leq
 \|u^{p/2}\|_{1,\varphi}^{\,2/(dN+2)}
 \|u^{p/2}\|_{2\ell,\varphi}^{\,dN/(dN+2)}.
\]
Consequently,
\[
 \|u^{p/2}\|_{2\ell,\varphi}^2
 \geq
 \|u\|_{p,\varphi}^{p(1+2/dN)}
 \|u\|_{p/2,\varphi}^{-2p/dN},
\]
and \eqref{p-sobolev} implies
\begin{align}
 -\frac{d}{dr}\|u(r)\|_{p,\varphi}^p
 &\geq
 c\|u(r)\|_{p,\varphi}^{p(1+2/dN)}
  \|u(r)\|_{p/2,\varphi}^{-2p/dN}
 \notag\\
 &\quad
 -Cp^2\left(
        \|\nabla\zeta\|_\infty^2+\frac1s
       \right)
  \|u(r)\|_{p,\varphi}^p.
 \label{nonlinear}
\end{align}

For $p\geq4$, apply \eqref{p-sobolev} with $p$ replaced
by $p/2$ and discard its positive term.  Gronwall's inequality
then gives
\[
 \|u(r)\|_{p/2,\varphi}
 \leq
 \exp\left\{
 Cp\left(\|\nabla\zeta\|_\infty^2+\frac1s\right)r
 \right\}
 \|f\|_{p/2,\varphi}.
\]
On the other hand, \eqref{nonlinear} gives
\begin{align*}
 \frac{d}{dr}\|u(r)\|_{p,\varphi}^{-2p/dN}
 &+
 Cp^2\left(
       \|\nabla\zeta\|_\infty^2+\frac1s
      \right)
 \|u(r)\|_{p,\varphi}^{-2p/dN}                                 \\
 &\geq c\|u(r)\|_{p/2,\varphi}^{-2p/dN}.
\end{align*}
Using the previous growth estimate on the right-hand side,
multiplying by an integrating factor, and integrating from $0$ to
$t$, we obtain
\[
 \|u(t)\|_{p,\varphi}^{-2p/dN}
 \geq
 ct\exp\left\{
 -Cp^2\left(\|\nabla\zeta\|_\infty^2+\frac1s\right)t
 \right\}
 \|f\|_{p/2,\varphi}^{-2p/dN}.
\]
It follows that
\begin{equation}
 \bigl\|e^\zeta e^{-tA}e^{-\zeta}
       \bigr\|_{L^{p/2}_\varphi\rightarrow L^p_\varphi}
 \leq
 C^{1/p}t^{-dN/(2p)}
 \exp\left\{
 Cp\left(\|\nabla\zeta\|_\infty^2+\frac1s\right)t
 \right\},
 \qquad p\geq4.
 \label{one-step}
\end{equation}
This estimate is first obtained for
$f\in L^{p/2}_\varphi\cap L^p_\varphi$ and then extended to
$L^{p/2}_\varphi$ by approximation.

To iterate \eqref{one-step}, let
\[
 p_j:=2^{j+2},\qquad j\geq0.
\]
Use the $L^2_\varphi$ growth estimate during the first half of the
time interval and then apply \eqref{one-step} successively
with exponent $p_j$ during intervals of length
\[
 \frac{6t}{p_j^2},\qquad j=0,1,\ldots
\]
Indeed,
\[
 \sum_{j=0}^\infty\frac{6t}{p_j^2}=\frac t2,
 \qquad
 \sum_{j=0}^\infty\frac1{p_j}=\frac12,
 \qquad
 \sum_{j=0}^\infty\frac{6t}{p_j}=3t.
\]
The product of the time factors satisfies
\[
 \prod_{j=0}^\infty
 \left(\frac{6t}{p_j^2}\right)^{-dN/(2p_j)}
 =
 t^{-dN/4}
 \prod_{j=0}^\infty
 \left(\frac{p_j^2}{6}\right)^{dN/(2p_j)}
 \leq Ct^{-dN/4},
\]
because $\sum_jp_j^{-1}\log p_j<\infty$.  The remaining constants
and exponential factors also form convergent products.  We conclude
that
\[
 \bigl\|e^\zeta e^{-tA}e^{-\zeta}
       \bigr\|_{L^2_\varphi\rightarrow L^\infty}
 \leq
 Ct^{-dN/4}
 \exp\left\{
 Ct\left(\|\nabla\zeta\|_\infty^2+\frac1s\right)
 \right\}.
\]
Since $e^{-tA}$ is self-adjoint, the adjoint of
$e^\zeta e^{-tA}e^{-\zeta}$ is
$e^{-\zeta}e^{-tA}e^\zeta$.  Replacing $\zeta$ by $-\zeta$
therefore proves the second estimate in
\eqref{two-infinity}.
\end{proof}

Using the duality and the semigroup property, we obtain from Lemma
\ref{moser-lemma}:
\begin{align*}
 \bigl\|e^\zeta e^{-tA}e^{-\zeta}
       \bigr\|_{L^1_\varphi\rightarrow L^\infty}
 &\leq
 \bigl\|e^\zeta e^{-(t/2)A}e^{-\zeta}
       \bigr\|_{L^2_\varphi\rightarrow L^\infty}
 \bigl\|e^\zeta e^{-(t/2)A}e^{-\zeta}
       \bigr\|_{L^1_\varphi\rightarrow L^2_\varphi}                    \\
 &=
 \bigl\|e^\zeta e^{-(t/2)A}e^{-\zeta}
       \bigr\|_{L^2_\varphi\rightarrow L^\infty}
 \bigl\|e^{-\zeta}e^{-(t/2)A}e^\zeta
       \bigr\|_{L^2_\varphi\rightarrow L^\infty}                       \\
 &\leq
 Ct^{-dN/2}
 \exp\left\{
       \frac{Ct}{s}
       +C\|\nabla\zeta\|_\infty^2t
      \right\}.
\end{align*}
The positive operator $e^\zeta e^{-tA}e^{-\zeta}$ has, with respect
to $\varphi(y)\,dy$, the kernel
\[
 e^{\zeta(x)-\zeta(y)}q_{s,\varepsilon}(t,x,y).
\]
It follows that
\begin{equation}
 q_{s,\varepsilon}(t,x,y)
 \leq
 Ct^{-dN/2}
 \exp\left\{
       \frac{Ct}{s}
       +C\|\nabla\zeta\|_\infty^2t
       +\zeta(y)-\zeta(x)
      \right\}.
 \label{bounded-weight-kernel}
\end{equation}
We now apply this bound with
$$
\zeta(z):=\max\{-R,\min\{R,\beta\cdot z\}\}, 
$$
where $\beta\in\mathbb R^{dN}$ and $R \geq 1$.
We have
$\|\nabla\zeta\|_\infty\leq|\beta|$, so taking $R \rightarrow \infty$, yields
\[
 q_{s,\varepsilon}(t,x,y)
 \leq
 Ct^{-dN/2}
 \exp\left\{
       \frac{Ct}{s}
       +C|\beta|^2t
       +\beta\cdot(y-x)
      \right\}.
\]
Minimizing the quadratic expression in $\beta$, i.e.\,taking $\beta=\frac{x-y}{2Ct}$ -- this is the final step in Davies' argument --  yields
\begin{equation}
 q_{s,\varepsilon}(t,x,y)
 \leq
 Ct^{-dN/2}
 \exp\left(
       \frac{Ct}{s}
       -\frac{|x-y|^2}{Ct}
      \right),
 \qquad t>0.
 \label{q-upper-A}
\end{equation}
This is the required upper estimate for the same symmetric kernel
$q$ used in the lower-bound proof.

Finally, choose $s=t$. Then
\[
 q_{t,\varepsilon}(t,x,y)
 \leq
 Ct^{-dN/2}
 \exp\left(-\frac{|x-y|^2}{Ct}\right).
\]
Lemma
\ref{gaussian-transfer}(a) with $M_0=1$ therefore yields,
after changing the Gaussian constants,
\[
 p_\varepsilon(t,x,y)
 \leq
 Ct^{-dN/2}
 \exp\left(-\frac{|x-y|^2}{Ct}\right)
 \varphi_{t,\varepsilon}(y).
\]
\hfill \qed

\bigskip

\section{Proof of Corollary \ref{cor1}}
\label{proof-cor1}

Fix $s,\varepsilon>0$ and keep the notation
$\varphi=\varphi_{s,\varepsilon}$,
$\psi=\psi_{s,\varepsilon}$, and
$h_\varepsilon=b_\varepsilon-\bar b_\varepsilon$ from the proof of
Theorem \ref{thm1}.  Thus
$h_\varepsilon=\nabla\log(\varphi/\psi)$, and \eqref{bb_est} says that
\begin{equation}
 \|h_\varepsilon\|_\infty\leq\frac C{\sqrt s}.
 \label{cor-h-bound}
\end{equation}
Constants denoted by $c,C$ below may additionally depend on
$\|C\|_\infty$.  They do not depend on $s,\varepsilon$, or derivatives
of $C$.

We use $C$ and $-C$ in parallel.  The only new notation needed in the
proof is
\begin{equation}
 k_{s,\varepsilon}^{\pm C}(t,x,y)
 :=\frac{p_\varepsilon^{\pm C}(t,x,y)}{\varphi(y)}.
 \label{cor-kernel}
\end{equation}
For fixed $\varepsilon$, $b_\varepsilon$ is bounded, $I\pm C$ is
bounded, and its symmetric part is $I$.  The standard existence and
uniqueness theorem for uniformly parabolic
divergence-form equations yields conservative fundamental
solutions, i.e.
\begin{equation}
 \left\langle k_{s,\varepsilon}^{\pm C}(t,x,\cdot)
 \right\rangle_\varphi=1.
 \label{cor-mass}
\end{equation}
We will also use the identity $
 \Lambda_\varepsilon^{\pm C}
 =-\psi^{-1}\nabla\cdot\bigl(\psi(I\pm C)\nabla\bigr)$.
For $f,g\in C_c^\infty(\mathbb R^{dN})$, integration by parts gives
\[
 \int_{\mathbb R^{dN}}(\Lambda_\varepsilon^C f)g\,\psi\,dy
 =\int_{\mathbb R^{dN}}f(\Lambda_\varepsilon^{-C}g)\,\psi\,dy.
\]
By uniqueness, the same  holds for the
semigroups, so
\begin{equation}
 \frac{p_\varepsilon^C(t,x,y)}{\psi(y)}
 =\frac{p_\varepsilon^{-C}(t,y,x)}{\psi(x)}.
 \label{cor-duality}
\end{equation}

We next obtain the three estimates from the proof of Theorem
\ref{thm1} that will be used below.  In $L^2_\varphi$, integration by
parts gives, for either sign,
\begin{equation}
 \left\langle\Lambda_\varepsilon^{\pm C}f,g\right\rangle_\varphi
 =\left\langle(I\pm C)\nabla f,\nabla g\right\rangle_\varphi
  +\left\langle
    \bigl[h_\varepsilon\cdot(I\pm C)\nabla f\bigr]g
   \right\rangle_\varphi.
 \label{cor-form}
\end{equation}
The Moser iteration argument in Appendix \ref{upper_bound_app}
applies to this form and its adjoint.  Indeed, skew-symmetry cancels
the term $(C\nabla u)\cdot\nabla u$; the remaining terms containing
$C$ also contain $\nabla\zeta$ and are estimated by Young's
inequality.  The terms containing $h_\varepsilon$ are treated in the
same way using \eqref{cor-h-bound}.  Thus, for both the semigroup
and its adjoint, the calculation leading to
\eqref{p-sobolev} becomes
\[
 -\frac d{dr}\|u(r)\|_{p,\varphi}^p
 \geq c\left\langle|\nabla u(r)^{p/2}|^2\right\rangle_\varphi
 -Cp^2\left(\|\nabla\zeta\|_\infty^2+\frac1s\right)
       \|u(r)\|_{p,\varphi}^p.
\]
For the adjoint equation this is the same calculation with $u$ and
$u^{p-1}$ interchanged in the right-hand side of
\eqref{cor-form}; in particular, no derivative of $C$ appears.
The weighted Sobolev inequality, the iteration in Lemma
\ref{moser-lemma}, and the final minimization in $\zeta$ are
unchanged.  Hence
\begin{equation}
 k_{s,\varepsilon}^{\pm C}(t,x,y)
 \leq Ct^{-dN/2}
 \exp\left(\frac{Ct}{s}-\frac{|x-y|^2}{Ct}\right).
 \label{cor-upper}
\end{equation}

We now repeat the Nash entropy calculation.  Fix either sign and put
$u(t,\cdot)=k_{s,\varepsilon}^{\pm C}(t,x,\cdot)$ -- this solves the adjoint equation in $L^2_\varphi$. Hence the first
derivative below is the right-hand side of \eqref{cor-form} with
$f=1+\log u$ and $g=u$.  Using also \eqref{cor-mass} and
$C^\top=-C$, we obtain
\begin{align}
 \frac d{dt}\bigl[-\langle u\log u\rangle_\varphi\bigr]
 &=\left\langle\frac{|\nabla u|^2}{u}\right\rangle_\varphi
   +\left\langle h_\varepsilon\cdot(I\pm C)\nabla u
     \right\rangle_\varphi \notag\\
 &\geq\frac12
   \left\langle\frac{|\nabla u|^2}{u}\right\rangle_\varphi
   -\frac Cs,                                                \label{cor-entropy-derivative}\\
 \left|\frac d{dt}\langle|x-\cdot|u\rangle_\varphi\right|
 &\leq C
   \left\langle\frac{|\nabla u|^2}{u}\right\rangle_\varphi^{1/2}
   +\frac C{\sqrt s}.                                       \label{cor-moment-derivative}
\end{align}
It follows from
\eqref{cor-entropy-derivative} that, for a sufficiently large
constant $K$,
\[
 \frac d{dt}\left[-\langle u\log u\rangle_\varphi+K\frac ts\right]
 \geq c\left(
   \left\langle\frac{|\nabla u|^2}{u}\right\rangle_\varphi
   +\frac1s\right).
\]
Therefore, the right-hand side of
\eqref{cor-moment-derivative} is bounded by a constant times the
square root of this derivative.  Moreover,
\eqref{cor-upper} and \eqref{cor-mass} give the lower entropy
bound used in Lemma \ref{prop-entropy}.  The proof of that lemma (from
\eqref{Mt} through \eqref{m_est} followed by Lemma
\ref{exp-convolution}) now applies without any other change.  The
added multiple of $t/s$ is bounded on the interval under
consideration.  We have therefore proved the Nash entropy estimate
\begin{equation}
 -C\leq-\langle u(t,\cdot)\log u(t,\cdot)\rangle_\varphi
          -\tilde Q(t)\leq C,
 \qquad \frac{s}{4\tau}\leq t\leq\frac{s}{\tau},
 \label{cor-entropy}
\end{equation}
for every fixed $\tau\geq1$, uniformly for the two signs.

It remains to check the modification of the $G$-function argument.
Put $u(t,\cdot)=k_{s,\varepsilon}^{\pm C}(t,z,\cdot)$.  This again satisfies the adjoint equation.
For $\varepsilon_1>0$, its derivative is the negative of the
right-hand side of \eqref{cor-form} with
$f=\Gamma_{s,\xi}/(\varepsilon_1+u)$ and $g=u$.  Thus
\begin{align}
 \frac d{dt}\left\langle\Gamma_{s,\xi}
       \log(\varepsilon_1+u)\right\rangle_\varphi
 &=\left\langle\Gamma_{s,\xi}
       |\nabla\log(\varepsilon_1+u)|^2\right\rangle_\varphi \notag\\
 &\quad-\left\langle (I\pm C)\nabla\Gamma_{s,\xi},
       \nabla\log(\varepsilon_1+u)\right\rangle_\varphi \notag\\
 &\quad-\left\langle\frac{u}{\varepsilon_1+u}
       h_\varepsilon\cdot(I\pm C)\nabla\Gamma_{s,\xi}
       \right\rangle_\varphi \notag\\
 &\quad+\left\langle\Gamma_{s,\xi}\frac{u}{\varepsilon_1+u}
       h_\varepsilon\cdot(I\pm C)
       \nabla\log(\varepsilon_1+u)\right\rangle_\varphi.
 \label{cor-G-derivative}
\end{align}
Since $0\leq u/(\varepsilon_1+u)\leq1$, Young's inequality,
\eqref{cor-h-bound}, and Lemma \ref{mass2} imply
\begin{equation}
 \frac d{dt}\left\langle\Gamma_{s,\xi}
       \log(\varepsilon_1+u)\right\rangle_\varphi
 \geq\frac12\left\langle\Gamma_{s,\xi}
       |\nabla\log(\varepsilon_1+u)|^2\right\rangle_\varphi
       -\frac Cs.
 \label{cor-G-energy}
\end{equation}
Choose $\tau\geq2$ sufficiently large.  Since
$m_{s,\varepsilon}(\xi)\geq c_\varphi$ and $t\leq s/\tau$, the last
term is absorbed by
$m_{s,\varepsilon}(\xi)\tilde Q'(t)$.  The translated spectral gap inequality
\eqref{translated-gap} then gives the counterpart of
\eqref{G-deriv}.  After increasing $\tau$ if necessary, estimate
\eqref{cor-upper} gives the counterpart of \eqref{estim-varphi},
with $k_{s,\varepsilon}^{\pm C}$ in place of $q$, and
\eqref{cor-entropy} gives the entropy input in
\eqref{Phi-ode}.  Thus the ordinary differential inequality argument in the proof of Lemma
\ref{N-G-bound}, including the limit $\varepsilon_1\downarrow0$,
gives
\begin{equation}
 \frac{\left\langle\Gamma_{s,\xi}
       \log k_{s,\varepsilon}^{\pm C}(t,z,\cdot)
       \right\rangle_\varphi}{m_{s,\varepsilon}(\xi)}
 \geq-\tilde Q(t)-C,
 \label{cor-G-bound}
\end{equation}
whenever $s/(2\tau)\leq t\leq s/\tau$ and
$z\in B_{\sqrt t}(\xi)$. 

We finish as in Step~1 of the proof of Theorem \ref{thm1}, replacing
symmetry by \eqref{cor-duality}.  Let $r>0$, set $s=\tau r$, assume
$|u-v|\leq\sqrt r$, and put $\xi=(u+v)/2$.  The reproduction formula
and \eqref{cor-duality} give
\begin{align*}
 k_{s,\varepsilon}^{C}(r,u,v)
 &=\left\langle
    k_{s,\varepsilon}^{C}(r/2,u,\cdot)
    k_{s,\varepsilon}^{-C}(r/2,v,\cdot)
    \exp\left\{
      \log\frac{\varphi(\cdot)}{\psi(\cdot)}
      -\log\frac{\varphi(v)}{\psi(v)}
    \right\}
   \right\rangle_\varphi.
\end{align*}
Insert
$e^{-|\cdot-\xi|^2/(4s)}=(4\pi s)^{dN/2}\Gamma_{s,\xi}\leq1$
and apply Jensen's inequality with the probability measure
$m_{s,\varepsilon}(\xi)^{-1}\Gamma_{s,\xi}\varphi\,dy$.  It follows
that
\begin{align}
 \log k_{s,\varepsilon}^{C}(r,u,v)
 &\geq\frac{dN}{2}\log(4\pi s)+\log m_{s,\varepsilon}(\xi) \notag\\
 &\quad+
 \frac{\left\langle\Gamma_{s,\xi}
       \log k_{s,\varepsilon}^{C}(r/2,u,\cdot)
       \right\rangle_\varphi}{m_{s,\varepsilon}(\xi)} \notag\\
 &\quad+
 \frac{\left\langle\Gamma_{s,\xi}
       \log k_{s,\varepsilon}^{-C}(r/2,v,\cdot)
       \right\rangle_\varphi}{m_{s,\varepsilon}(\xi)} \notag\\
 &\quad+
 \frac{\left\langle\Gamma_{s,\xi}
       \left(\log\frac{\varphi(\cdot)}{\psi(\cdot)}
             -\log\frac{\varphi(v)}{\psi(v)}\right)
       \right\rangle_\varphi}{m_{s,\varepsilon}(\xi)}.
 \label{cor-Jensen}
\end{align}
By \eqref{bb_est} and Lemma \ref{mass2}, the absolute value of the
last quotient is at most
\[
 \|h_\varepsilon\|_\infty
 \left(
 |v-\xi|+
 \frac{\langle\Gamma_{s,\xi}|\cdot-\xi|\rangle_\varphi}
      {m_{s,\varepsilon}(\xi)}
 \right)\leq C.
\]
Here $|v-\xi|\leq\sqrt r/2$ and $s=\tau r$.  At the same time,
$r/2=s/(2\tau)$ and
$u,v\in B_{\sqrt{r/2}}(\xi)$, so \eqref{cor-G-bound} applies to
the other two quotients.  Equations \eqref{cor-Jensen} and
\eqref{mass-bound} therefore yield
\begin{equation}
 p_\varepsilon^C(r,u,v)
 \geq cr^{-dN/2}\varphi_{\tau r,\varepsilon}(v),
 \qquad |u-v|\leq\sqrt r.
 \label{cor-local}
\end{equation}

This is the near diagonal estimate \eqref{p-local} used in
Step~2 of the proof of Theorem \ref{thm1}.  That step uses only the
semigroup property, \eqref{uniform-weight-lower}, and the
monotonicity of $\eta$, and hence can be repeated for
$p_\varepsilon^C$.  We obtain
\[
 p_\varepsilon^C(t,x,y)
 \geq ct^{-dN/2}
       \exp\left(-C\frac{|x-y|^2}{t}\right)
       \varphi_{t,\varepsilon}(y),
 \qquad 0<t\leq T,
\]
as claimed. \hfill \qed

\bigskip

\makeatletter
\let\savedtocwrite\@tocwrite
\let\@tocwrite\@gobbletwo
\section*{AI usage disclosure}
\let\@tocwrite\savedtocwrite
\makeatother

The authors used OpenAI ChatGPT for mathematical discussions with the goal of improving the constants in the estimates in their first version of the proofs, in particular, concerning the regularization of the particle system in dimension $d=3$. We also used ChatGPT to grasp some ideas coming from several complex variables in order to treat the case $d=2$. Finally, we used ChatGPT for editorial purposes, i.e.\,fixing typos, improving language and carrying out analysis of the existing literature that had found a number of important references that we were not aware of when the first version of this paper was completed.

\bigskip

\raggedbottom

\end{document}